\documentclass[11pt, reqno]{amsart}

\usepackage[usenames,dvipsnames]{pstricks}
\usepackage{epsfig}

\usepackage{color}
\date{\today}

\usepackage[latin1]{inputenc}
\usepackage[T1]{fontenc}
\usepackage{amsmath, amsthm}
\usepackage[english]{babel}
\usepackage{amssymb}
\usepackage{latexsym}
\usepackage{fancyhdr}
\usepackage{graphicx}
\usepackage[all]{xy}

\theoremstyle{definition}
\newtheorem{definition}{{\bf Definition}}[section]
\newtheorem{apart}[definition]{}
\newtheorem{remark}[definition]{{\bf Remark}}

\newtheorem{examples}[definition]{{\bf Examples}}
\newtheorem{theorem}[definition]{{\bf Theorem}}
\newtheorem{proposition}[definition]{{\bf Proposition}}

\newtheorem{corollary}[definition]{{\bf Corollary}}
\newtheorem{lemma}[definition]{{\bf Lemma}}
\newtheorem{notation}[definition]{{\bf Notation}}

{{\noindent{\bf Proof: }}\newline }%
{$\hfill\Box$\vspace{0.25cm}} %

\newcommand{\ot}{\otimes}
\newcommand{\co}{\circ}
\newcommand{\CC}{\mathcal{C}}

\begin{document}

\begin{center}
{\bf\huge{ The monoidal category of  Yetter-Drinfeld modules over a
weak braided Hopf algebra
\\
}}

\ \\

{\bf J.N. Alonso \'Alvarez$^{1,\ast}$, J.M. Fern\'andez
Vilaboa$^{2}$, R. Gonz\'{a}lez Rodr\'{\i}guez$^{3}$ and C. Soneira
Calvo$^{4}$}

\end{center}

\ \\
\hspace{-0,25cm}$^{1}$ {\bf Corresponding author}. Departamento de
Matem\'{a}ticas, Universidad de Vigo, Campus Universitario
Lagoas-Marcosende, E-36280 Vigo, Spain (e-mail: jnalonso@uvigo.es)
\ \\
\hspace{-0,25cm}$^{2}$ Departamento de \'Alxebra, Universidad de
Santiago de Compostela,  E-15771 Santiago de Compostela, Spain
(e-mail: josemanuel.fernandez@usc.es)
\ \\
\hspace{-0,25cm}$^{3}$ Departamento de Matem\'{a}tica Aplicada II,
Universidad de Vigo, Campus Universitario Lagoas-Marcosende, E-36310
Vigo, Spain (e-mail: rgon@dma.uvigo.es)
\ \\
\hspace{-0,25cm}$^{4}$ Departamento de Pedagox\'{\i}a e
Did\'{a}ctica, Universidade da Coru\~{n}a, Campus Universitario de
Elvi\~{n}a, E-15007 A Coru\~{n}a, Spain (carlos.soneira@udc.es)
\ \\

\hspace{-0,5cm}$\ast$ Corresponding author.

\begin{center}
{\bf Abstract}
\end{center}
{\small In this paper we introduce the notion of weak operator and
the theory of Yetter-Drinfeld modules over a weak braided Hopf
algebra with invertible antipode in a strict monoidal category. We
prove that the class of such objects constitutes a non strict
monoidal category. It is also shown that this category is not
trivial, that is to say that it admits  objects generated by the
adjoint action (coaction) associated to the weak braided Hopf
algebra.}

\vspace{0.5cm}



{\bf Abbreviated title:} The monoidal category of YD-mod over a WBHA.

\begin{abstract}
In this paper we introduce the notion of weak operator and  the
theory of Yetter-Drinfeld modules over a weak braided Hopf algebra
with invertible antipode in a strict monoidal category. We  prove
that the class of such objects constitute a non strict monoidal
category. It is also shown that this category is not trivial, that
is to say that it admits  objects generated by the adjoint action
(coaction) associated to the weak braided Hopf algebra.
\end{abstract}

\section*{Introduction}

$\hspace{0.45cm}$ The notion of Yetter-Drinfeld module was
considered to deal with the  quantum Yang-Baxter equation, specially
in quantum mechanics (see \cite{MAJ1} for a detailed exposition of
its physical implications). Actually, every Yetter-Drinfeld module
gives rise to a solution to the quantum Yang-Baxter equation, as was
proved in \cite{LR}, and  if $H$ is a finite Hopf algebra in a
symmetric category $\CC,$ the category $_H^H\mathcal{YD}$ of
left-left Yetter-Drinfeld modules is isomorphic to the category of
modules over the Drinfeld quantum double, which was originally
conceived to find solutions of the Yang-Baxter equation via
universal matrices. Continuing with physical applications,  any
projection of  a Hopf algebra provides an example of a
Yetter-Drinfeld module (see \cite{Rad}) and this result is the
substrate of the  bosonization process introduced by Majid in
\cite{MAJ2} that gives, for  a quasitriangular Hopf algebra, an
interpretation of cross products in terms of quantum algebras of
observables of dynamical systems, as well as in quantum group gauge
theory.

On the other hand, weak Hopf algebras  (or quantum groupoids in the
terminology of Nikshych and Vainerman \cite{NV}) were introduced by
B\"{o}hm, Nill and Szlach\'anyi in \cite{BNS} as a new
generalization of Hopf algebras and groupoid algebras. The main
difference with other Hopf algebraic constructions, such as
quasi-Hopf algebras and rational Hopf algebras, is that weak Hopf
algebras are coassociative but the coproduct is not required to
preserve the unit or, equivalently, the counit is not an
algebra morphism. Some motivations to study weak Hopf algebras come
from the following facts: firstly, as group algebras and their duals
are the natural examples of Hopf algebras, groupoid algebras and
their duals provide examples of weak Hopf algebras; secondly, these
algebraic structures have a remarkable connection with the theory of
algebra extensions, important applications in the study of dynamical
twists of Hopf algebras and a deep link with quantum field theories
and operator algebras (see \cite{NV}), as well as they are useful
tools in the study of fusion categories in characteristic zero (see
\cite{ENO}). The theory of Yetter-Drinfeld modules for a weak Hopf algebra
was introduced by Böhm in \cite{B}. Later, Nenciu proved in \cite{NEN} that this
 category is isomorphic to the category of modules over the Drinfeld
  quantum double (the interested reader can also see \cite{CWY}).

In \cite{NikaRamon4} we can find the extension of Radford's theory
for projections of Hopf algebras to projections of weak Hopf
algebras in a strict symmetric  monoidal category $\CC$ where every
idempotent morphism splits. The main result of \cite{NikaRamon4},
extended to the braided setting in \cite{Proj}, assures that there
exists  a categorical equivalence between the category  of
isomorphism classes of projections associated to a weak Hopf algebra
$H$ and the category of Hopf algebras in the category of left-left
Yetter-Drinfeld modules over $H$. To show this result, the authors
introduced in \cite{NikaRamon4} the notions of weak Yang-Baxter
operator and weak braided Hopf algebra. Roughly speaking,  a weak
braided Hopf algebra in a strict monoidal category is an
algebra-coalgebra with a weak Yang-Baxter operator, satisfying
some compatibility conditions. This definition generalizes the one
introduced by Takeuchi in \cite{T}, i.e., the definition of braided
Hopf algebra, and the classical notions of Hopf algebra and Hopf
algebra in a braided category. Moreover, as particular instances we
recover the definition of weak Hopf algebra and, if the weak
Yang-Baxter operator is the braiding of a braided category, the
 notion of weak Hopf algebra in a braided monoidal setting is
formulated. The
first non-trivial example of weak braided Hopf algebras can be
constructed by modifying the algebraic structure of a Hopf algebra $D$
in the non-strict braided monoidal category $\;^{H}_{H}{\mathcal
Y}{\mathcal D}$ [\cite{NikaRamon4}, Corollary 2.14]. In this case
 with these new product, coproduct, unit, counit and antipode $D$
is not a Hopf algebra neither a weak Hopf algebra in the usual
sense.

In \cite{Proj} the authors proved that some relevant properties
about projections associated to a weak braided Hopf algebra can be
obtained without the use of a general braiding in the category where
the weak braided Hopf algebra lives. This fact motivates the
following questions: is it possible to establish a Yetter-Drinfeld
module category for a weak braided Hopf algebra in a  general strict
monoidal category where every idempotent morphism splits? is it this
category isomorphic to the center of some monoidal category of
modules? The positive answer to the first question is the main
contribution of this paper.
 To do it we introduce the notion of weak operator which constitutes
 a generalization of the concept
 of weak Yang-Baxter operator and is the key in order to define a  non-strict
 monoidal category of Yetter-Drinfeld modules associated to a weak braided Hopf algebra.
  To illustrate this new notions
 we provide several examples of Yetter-Drinfeld modules in this general setting.
 A family of them comes from projections of weak braided Hopf algebras, while another
 collection is based on the use of the adjoint (co)action that in the weak setting
 is not in general a (co)module structure for the weak braided Hopf algebra.

The organization of the paper is the following. In Section 1  the
general framework is stated recalling  the definitions of weak
Yang-Baxter operator, weak braided bialgebra and weak braided Hopf
algebra; then we introduce the notion of weak operator and  obtain
its main properties. In Section 2 we establish the definition of
left-left Yetter-Drinfeld module over an arbitrary weak braided Hopf
algebra $D$ and prove that these objects constitute a non strict
monoidal category, giving  explicitly all the required constraints
and the base object. Section 3 is devoted to the study of
projections and  the relation between weak Yang-Baxter operators and
weak entwining structures in terms of weak operators. Finally, in
Section 4 we  use the adjoint (co)action to obtain different
examples of Yetter-Drinfeld structures starting from an arbitrary
weak braided Hopf algebra and  include the explicit computations for
the particular cases of groupoid algebras, Frobenius separable
algebras in a braided setting and projections of weak braided Hopf
algebras.

\section{Weak operators}

$\hspace{0.45cm}$ In this paper we denote a  monoidal category $\CC$
as $(\CC, \ot, K, \mathfrak{a}, \mathfrak{l},\mathfrak{r})$ where
$\CC$ is a category and $\ot $ (tensor product) provides $\CC$ with
a monoidal structure with unit object $K$ whose associative
constraint is denoted by $\mathfrak{a}$ and whose left and right
unit constraints are given by $\mathfrak{l}$ and $\mathfrak{r}$
respectively.

We denote the class of objects of ${\mathcal C}$ by $\vert {\mathcal
C} \vert $ and for each object $M\in \vert {\mathcal C}\vert$, the
identity morphism by $id_{M}:M\rightarrow M$. For simplicity of
notation, given objects $M$, $N$, $P$ in ${\mathcal C}$ and a
morphism $f:M\rightarrow N$, we write $P\ot f$ for $id_{P}\ot f$ and
$f \ot P$ for $f\ot id_{P}$.

From now on we assume that ${\mathcal C}$ is strict and  every
 idempotent morphism in ${\mathcal C}$ splits, i.e.
  for every morphism $\nabla_{Y}:Y\rightarrow
Y$ such that $\nabla_{Y}=\nabla_{Y}\co \nabla_{Y}$ there exist an
object $Z$ (called the image of $\nabla_{Y}$) and morphisms
$i_{Y}:Z\rightarrow Y$ and $p_{Y}:Y\rightarrow Z$ such that
$\nabla_{Y}=i_{Y}\co p_{Y}$ and $p_{Y}\co i_{Y} =id_{Z}$. There is
not loss of generality in assuming the strict character for
${\mathcal C}$ because of it is well known that given a monoidal
category we can construct a strict monoidal category ${\mathcal
C}^{st}$ which is tensor equivalent to ${\mathcal C}$ (see
\cite{Karoubi} for the details); neither in assuming that ${\mathcal
C}$ admits split idempotents, having into account that for a given category
${\mathcal C}$ there exists an universal embedding ${\mathcal
C}\rightarrow \hat{\mathcal C}$ such that $\hat{\mathcal C}$ admits
split idempotents, as was proved in  \cite{Karoubi}.

A braided monoidal category ${\mathcal C}$ means a monoidal category
in which there is, for all $M$ and $N$ in ${\mathcal C}$, a
natural isomorphism $c_{M, N}:M\ot N\rightarrow N\ot M$, called the
braiding, satisfying the Hexagon Axiom (see \cite{JS} for
generalities). If the braiding satisfies $c_{N,M}\co
c_{M,N}=id_{M\ot N}$ for all $M$, $N$ in ${\mathcal C}$,
 the category will be called
symmetric.

\begin{definition}
An algebra in ${\mathcal C}$ is a triple $A=(A, \eta_{A}, \mu_{A})$
where $A$ is an object in ${\mathcal C}$ and
 $\eta_{A}:K\rightarrow A$ (unit), $\mu_{A}:A\ot A
\rightarrow A$ (product) are morphisms in ${\mathcal C}$ such that
$\mu_{A}\co (A\ot \eta_{A})=id_{A}=\mu_{A}\co (\eta_{A}\ot A)$,
$\mu_{A}\co (A\ot \mu_{A})=\mu_{A}\co (\mu_{A}\ot A)$. Given two
algebras $A= (A, \eta_{A}, \mu_{A})$ and $B=(B, \eta_{B}, \mu_{B})$,
$f:A\rightarrow B$ is an algebra morphism if $ f\co \eta_{A}=
\eta_{B}$, $\mu_{B}\co (f\ot f)=f\co \mu_{A}$.

A coalgebra in ${\mathcal C}$ is a triple ${D} = (D,
\varepsilon_{D}, \delta_{D})$ where $D$ is an object in ${\mathcal
C}$ and $\varepsilon_{D}: D\rightarrow K$ (counit),
$\delta_{D}:D\rightarrow D\ot D$ (coproduct) are morphisms in
${\mathcal C}$ such that $(\varepsilon_{D}\ot D)\co \delta_{D}=
id_{D}=(D\ot \varepsilon_{D})\co \delta_{D}$, $(\delta_{D}\ot D)\co
\delta_{D}=
 (D\ot \delta_{D})\co \delta_{D}.$ If ${D} = (D, \varepsilon_{D},
 \delta_{D})$ and
${ E} = (E, \varepsilon_{E}, \delta_{E})$ are coalgebras,
$f:D\rightarrow E$ is a coalgebra morphism if $\varepsilon_{E}\co f
=\varepsilon_{D}$, $(f\ot f)\co \delta_{D} =\delta_{E}\co f$.

If $A$ is an algebra, $B$ is a coalgebra and $\alpha:B\rightarrow
A$, $\beta:B\rightarrow A$ are morphisms, we define the convolution
product by $\alpha\wedge \beta=\mu_{A}\co (\alpha\ot \beta)\co
\delta_{B}$.

If $(D, \eta_D,\mu_D)$ is an algebra in $\CC,$  the pair
$(M,\varphi_M)$,  with $M\in|\CC| $ and $\varphi_M:D\ot M\rightarrow
D$ is said to be a  left $D$-module if $\varphi_M\co(\eta_D\ot
M)=id_M$ and $\varphi_M\co(D\ot\varphi_M)=\varphi_M\co(\mu_D\ot M).$
Given two left $D$-modules $(M, \varphi_M)$
 and $(N, \varphi_N)$, $f:M\rightarrow N$ is a morphism of left $D$-modules if $\varphi_N\circ (D\ot f)=f\circ \varphi_M$.

If $(D,\varepsilon_D,\delta_D)$ is a coalgebra in $\CC,$ the pair
$(M, \varrho_M)$ with $M\in|\CC|$ and $\varrho_M:M\rightarrow D\ot
M$ is said to be a left $D$-comodule if $(\varepsilon_D\ot
M)\co\varrho_M=id_M$ and $(D\ot\varrho_M)\co\varrho_M=(\delta_D\ot
M)\co \varrho_M$.  Given two
  left $D$-comodules $(M, \varrho_M)$ and $(N, \varrho_N)$, $f:M\rightarrow N$ is a morphism of left $D$-comodules if
  $\varrho_N\circ f=(D\ot f)\circ \varrho_M$.
The notions of right $D$-(co)module are defined analogously.
\end{definition}

\begin{definition}
Let $D$ be in ${\mathcal C}$ and let $t_{D,D}:D\ot
D\rightarrow D\ot D$ be a morphism in ${\mathcal C}$. We will say
that $t_{D,D}$ satisfies the Yang-Baxter equation if
\begin{equation}
(t_{D,D}\ot D)\co (D\ot t_{D,D})\co (t_{D,D}\ot D)=(D\ot t_{D,D})\co
(t_{D,D}\ot D)\co (D\ot t_{D,D}).
\end{equation}

\end{definition}

Weak Yang-Baxter operators are generalizations of Yang-Baxter
operators (see \cite{JS}) and were introduced by Alonso,
Gonz\'{a}lez and Rodr\'{i}guez in \cite{NikaRamon4}. In \cite{Proj}
we prove that one axiom of the original definition can be dropped. We rewrite the
improved definition:

\begin{definition}\label{wybopdef}
 Let $D$ be in ${\mathcal C}$. A weak Yang-Baxter
operator is a morphism $t_{D,D}:D\ot D\rightarrow D\ot D$ in
${\mathcal C}$ such that:
\begin{itemize}
\item[(a1)] $t_{D,D}$ satisfies the Yang-Baxter equation.

\item[(a2)] There exists an idempotent morphism
$\nabla_{D,D}:D\ot D\rightarrow D\ot D$ satisfying the following
identities:
\begin{itemize}
\item[(a2-1)] $(\nabla_{D,D}\ot D)\co (D\ot \nabla_{D,D})=(D\ot \nabla_{D,D})\co (\nabla_{D,D}\ot D),$

\item[(a2-2)] $(\nabla_{D,D}\ot D)\co (D\ot t_{D,D})=(D\ot
t_{D,D})\co (\nabla_{D,D}\ot D),$

\item[(a2-3)]$(t_{D,D}\ot D)\co (D\ot \nabla_{D,D})=(D\ot
\nabla_{D,D})\co (t_{D,D}\ot D),$

\item[(a2-4)]$t_{D,D}\co \nabla_{D,D}=\nabla_{D,D}\co
t_{D,D}=t_{D,D}.$

\end{itemize}

\item[(a3)] There exists a morphism $t^{\prime}_{D,D}:D\ot
D\rightarrow D\ot D$ such that:

\begin{itemize}

\item[(a3-1)] The morphism $p_{D,D}\co t_{D,D}\co i_{D,
D}:D\times D\rightarrow D\times D$ is an isomorphism with inverse
$p_{D, D}\co t^{\prime}_{D,D}\co i_{D, D}:D\times D\rightarrow
D\times D$, where $p_{D, D}$ and $i_{D, D}$ are the morphisms such
that $i_{D, D}\co p_{D, D}=\nabla_{D,D}$ and $p_{D, D}\co i_{D,
D}=id_{D\times D}$ being $D\times D$ the image of $\nabla_{D,D}$.

\item[(a3-2)]$t^{\prime}_{D,D}\co \nabla_{D,D}=\nabla_{D,D}\co t^{\prime}_{D,D}=t^{\prime}_{D,D}.$

\end{itemize}
\end{itemize}

Note that if $\nabla_{D,D}=id_{D\ot D}$ then $t_{D,D}$ is an
isomorphism and we recover the definition of Yang-Baxter operator
introduced by Joyal and Street in \cite{JS}. Also,
by [\cite{NikaRamon4}, Definition 1.2 and Proposition 1.3], we get that $t_{D,D}$ is a weak Yang-Baxter
operator with associated idempotent $\nabla_{D,D}$ if and only if so
is $t^{\prime}_{D,D}.$  Moreover, in this case
\begin{equation}\label{defnabla}
{t^\prime}_{D,D}\co t_{D,D}=t_{D,D}\co
{t^\prime}_{D,D}=\nabla_{D,D}.
\end{equation}

Finally, using the identities (2)-(5) of \cite{NikaRamon4} we
obtain:
\begin{equation}
\label{tprim1} (D\ot t_{D,D})\circ (t_{D,D}\ot D)\circ (D\ot
t^{\prime}_{D,D})=(t^{\prime}_{D,D}\ot D)\circ (D\ot t_{D,D})\circ
(t_{D,D}\ot D),
\end{equation}
\begin{equation}
\label{tprim2} ( t_{D,D}\ot D)\circ (D\ot t_{D,D})\circ (
t^{\prime}_{D,D}\ot D)=(D\ot t^{\prime}_{D,D})\circ (t_{D,D}\ot
D)\circ (D\ot t_{D,D}),
\end{equation}
\begin{equation}
\label{tprim3} (D\ot t^{\prime}_{D,D})\circ (t^{\prime}_{D,D}\ot
D)\circ (D\ot t_{D,D})=(t_{D,D}\ot D)\circ (D\ot
t^{\prime}_{D,D})\circ (t^{\prime}_{D,D}\ot D),
\end{equation}
\begin{equation}
\label{tprim4} ( t^{\prime}_{D,D}\ot D)\circ (D\ot
t^{\prime}_{D,D})\circ ( t_{D,D}\ot D)=(D\ot t_{D,D})\circ
(t^{\prime}_{D,D}\ot D)\circ (D\ot t^{\prime}_{D,D}).
\end{equation}

\end{definition}

\begin{examples}
\label{Yang-Baxter-examples}

{\bf (1)} In this first example we assume that ${\mathcal
C}$ is symmetric. The categories of Yetter-Drinfeld modules over weak Hopf
algebras provide non-trivial examples of weak Yang-Baxter operators.
A weak Hopf algebra $H$ is an object in ${\mathcal C}$ with an algebra structure $(H,
\eta_{H},\mu_{H})$ and a coalgebra structure $(H,
\varepsilon_{H},\delta_{H})$ such that the following axioms hold:
\begin{itemize}
\item[(i)] $\delta_{H}\circ \mu_{H}=(\mu_{H}\otimes \mu_{H})\circ (H\otimes c_{H,H}\otimes H)\circ
(\delta_{H}\otimes \delta_{H})$,
 \item[(ii)]$\varepsilon_{H}\circ \mu_{H}\circ
(\mu_{H}\otimes H)=(\varepsilon_{H}\otimes \varepsilon_{H})\circ
(\mu_{H}\otimes \mu_{H})\circ (H\otimes \delta_{H}\otimes H)$
\item[]$=(\varepsilon_{H}\otimes \varepsilon_{H})\circ (\mu_{H}\otimes
\mu_{H})\circ (H\otimes (c_{H,H}\circ\delta_{H})\otimes H),$
\item[(iii)]$(\delta_{H}\otimes H)\circ \delta_{H}\circ
\eta_{H}=(H\otimes \mu_{H}\otimes H)\circ (\delta_{H}\otimes
\delta_{H})\circ (\eta_{H}\otimes \eta_{H})$
\item[ ]$=(H\otimes
(\mu_{H}\circ c_{H,H})\otimes H)\circ (\delta_{H}\otimes
\delta_{H})\circ (\eta_{H}\otimes \eta_{H}).$

\item[(iv)] There exists a morphism $\lambda_{H}:H\rightarrow H$
in ${\mathcal C}$ (called the antipode of $H$) verifiying:
\begin{itemize}
\item[(iv-1)] $id_{H}\wedge \lambda_{H}=((\varepsilon_{H}\circ
\mu_{H})\otimes H)\circ (H\otimes c_{H,H})\circ ((\delta_{H}\circ
\eta_{H})\otimes H),$ \item[(iv-2)] $\lambda_{H}\wedge
id_{H}=(H\otimes(\varepsilon_{H}\circ \mu_{H}))\circ (c_{H,H}\otimes
H)\circ (H\otimes (\delta_{H}\circ \eta_{H})),$
\item[(iv-3)]$\lambda_{H}\wedge id_{H}\wedge
\lambda_{H}=\lambda_{H}.$
\end{itemize}
\end{itemize}

If we define the morphisms $\Pi_{H}^{L}$ (target), $\Pi_{H}^{R}$
(source), as
$$\Pi_{H}^{L}=((\varepsilon_{H}\circ \mu_{H})\otimes
H)\circ (H\otimes c_{H,H})\circ ((\delta_{H}\circ \eta_{H})\otimes
H),$$
$$\Pi_{H}^{R}=(H\otimes(\varepsilon_{H}\circ \mu_{H}))\circ
(c_{H,H}\otimes H)\circ (H\otimes (\delta_{H}\circ \eta_{H})),$$ it
is straightforward to show that they are idempotent.

The first family of examples of weak Hopf algebras cames from the
theory of groupoid algebras. Recall that a groupoid $G$ is simply a
small category where all morphisms are isomorphisms. In this example,
we consider finite groupoids, i.e. groupoids with a finite number of
objects. The set of objects of $G$, called also the base of $G$,
will be denoted by $G_{0}$ and the set of morphisms by $G_{1}$. The
identity morphism on $x\in G_{0}$ will  be denoted by $id_{x}$ and
for a morphism $\sigma:x\rightarrow y$ in $G_{1}$, we write
$s(\sigma)$ and $t(\sigma)$, respectively for the source and the
target of $\sigma$.

Let $G$ be a groupoid and $R$ a commutative ring. The groupoid
algebra is the direct product in $R$-Mod
$$RG=\bigoplus_{\sigma\in G_{1}}R\sigma$$
where the product of two morphisms is equal to their composition
if the latter is defined and $0$  otherwise, i.e.
$\mu_{RG}(\tau\otimes \sigma)=\tau\circ \sigma$ if
$s(\tau)=t(\sigma)$ and $\mu_{RG}(\tau\otimes \sigma)=0$ if
$s(\tau)\neq t(\sigma)$. The unit element is $1_{RG}=\sum_{x\in
G_{0}}id_{x}$. The algebra $RG$ is a cocommutative weak Hopf
algebra, with coproduct $\delta_{RG}$, counit $\varepsilon_{RG}$ and
antipode $\lambda_{RG}$ given by the formulas:
$$\delta_{RG}(\sigma)=\sigma\otimes \sigma, \;\;\;\varepsilon_{RG}
(\sigma)=1,\;\;\; \lambda_{RG}(\sigma)=\sigma^{-\dot{}1}.$$

For the weak Hopf algebra $RG$ target and source morphisms  are
respectively,
$$\Pi_{RG}^{L}(\sigma)=id_{t(\sigma)},\;\;\;
\Pi_{RG}^{R}(\sigma)=id_{s(\sigma)}$$ and $\lambda_{RG}\circ
\lambda_{RG}=id_{RG}$.

If $(M,\varphi_{M})$ and $(N,\varphi_{N})$ are left $H$-modules we
denote by $\varphi_{M\otimes N}$ the morphism $\varphi_{M\otimes
N}:H\otimes M\otimes N\rightarrow M\otimes N$ defined by
$$\varphi_{M\otimes N}=(\varphi_{M}\otimes \varphi_{N})\circ
(H\otimes c_{H,M}\otimes N)\circ (\delta_{H}\otimes M\otimes N).$$

For two left $H$-comodules $(M,\varrho_{M})$ and $(N,\varrho_{N})$,
 we denote by $\varrho_{M\otimes N}$ the
morphism $\varrho_{M\otimes N}: M\otimes N\rightarrow H\otimes
M\otimes N$ defined by
$$\varrho_{M\otimes N}=(\mu_{H}\otimes M\otimes N)\circ
(H\otimes c_{M,H}\otimes N)\circ (\varrho_{M}\otimes \varrho_{N}).$$

 Let $(M,\varphi_{M})$,
$(N,\varphi_{N})$ be left $H$-modules. The morphism
$$\nabla_{M\otimes N}=\varphi_{M\otimes N}\circ (\eta_{H}\otimes
M\otimes N):M\otimes N\rightarrow M\otimes N$$  is an idempotent. In
this setting we denote by $M\times N$ the image of $\nabla_{M\otimes
N}$ and by $p_{M\otimes N}:M\otimes N\rightarrow M\times N$,
$i_{M\otimes N}:M\times N\rightarrow M\otimes N$ the morphisms such
that $i_{M\otimes N}\circ p_{M\otimes N}=\nabla_{M\otimes N}$ and
$p_{M\otimes N}\circ i_{M\otimes N}=id_{M\times N}$. It is not
difficult to see that the object $M\times N$ is a left $H$-module
with action $\varphi_{M\times N}=p_{M\otimes N}\circ
\varphi_{M\otimes N}\circ (H\otimes i_{M\otimes N}):H\otimes
M\times N\rightarrow M\times N$.

In a similar way, if $(M,\varrho_{M})$ and $(N,\varrho_{N})$ are
left $H$-comodules the morphism
$$\nabla_{M\otimes N}^{\prime}=(\varepsilon_{H}\otimes M\otimes N)
\circ \varrho_{M\otimes N}:M\otimes N\rightarrow M\otimes N$$ is an
idempotent. We denote by $M\odot N$ the image of $\nabla_{M\otimes
N}^{\prime}$ and by $p_{M\otimes N}^{\prime}:M\otimes N\rightarrow
M\odot N$, $i_{M\otimes N}^{\prime}:M\odot N\rightarrow M\otimes N$
the morphisms such that $i_{M\otimes N}^{\prime}\circ p_{M\otimes
N}^{\prime}=\nabla_{M\otimes N}^{\prime}$ and $p_{M\otimes
N}^{\prime}\circ i_{M\otimes N}^{\prime}=id_{M\odot N}$. In a
similar way to the preceding case, $M\odot N$ is a left $H$-comodule
with coaction $\varrho_{M\odot N}=(H\otimes p_{M\otimes
N}^{\prime})\circ \varrho_{M\otimes N}\circ i_{M\otimes
N}^{\prime}:M\odot N\rightarrow H\otimes (M\odot N)$.

We shall denote by ${}_H^H\mathcal{YD}$ the category of left-left
Yetter-Drinfeld modules over $H$, e.g.;
 $(M,\varphi_M,\rho_M)$ is an object in ${}_H^H\mathcal{YD}$ if $(M,\varphi_M)$ is a
left $H$-module, $(M,\rho_M)$ is a left $H$-comodule and
\begin{itemize}

\item[(yd1)] $\rho_M=(\mu_H\ot\varphi_M)\co(H\ot c_{H,H}\ot
M)\co(\delta_H\ot\rho_M)\co(\eta_D\ot M).$

\item[(yd2)]
\begin{itemize}
\item[ ]$\hspace{0.38cm} (\mu_H\ot M)\co(H\ot c_{M,H})\co((\rho_M\co\varphi_M)\ot
H)\co(H\ot c_{H,M})\co(\delta_H\ot M)$

\item[ ]$= (\mu_H\ot\varphi_M)\co(H\ot c_{H,H}\ot M)\co(\delta_H\ot\rho_M).$

\end{itemize}

\end{itemize}
Let $M,$ $N$ be in $_H^H{\mathcal YD}.$ The  morphism
$f:M\rightarrow N$ is a morphism of left-left Yetter-Drinfeld
modules if $f\co\varphi_M=\varphi_N\co(D\ot f)$ and $(D\ot
f)\co\rho_M=\rho_N\co f.$

If $(M,\varphi_M,\varrho_M)$ is a left-left Yetter-Drinfeld module
over $H$ then it obeys the following equality [\cite{NikaRamon4},
Proposition 1.12]:
\begin{equation}
\label{nabla-nabla} \nabla_{M\otimes N}=\nabla_{M\otimes
N}^{\prime}.
\end{equation}

Then if the antipode of $H$ is an isomorphism the
category ${}_H^H\mathcal{YD}$ is a non-strict braided monoidal
category. We expose briefly the braided monoidal structure.

For two left-left Yetter-Drinfeld modules
$(M,\varphi_{M},\varrho_{M})$, $(N,\varphi_{N},\varrho_{N})$ the
tensor product is defined as the image of
$\nabla_{M\otimes N}$. By (\ref{nabla-nabla}), $M\times N=M\odot N$
and this object is a left-left Yetter-Drinfeld module with the
following action and coaction:
$$
\varphi_{M\times N}=p_{M\otimes N}\circ \varphi_{M\otimes N}\circ
(H\otimes i_{M\otimes N})
$$
$$
 \varrho_{M\times N}=(H\otimes p_{M\otimes
N})\circ \varrho_{M\otimes N}\circ  i_{M\otimes N}.
$$

The base object is the image of the target morphism denoted by
$H_{L}$, which is a left-left Yetter-Drinfeld module  with
(co)module structure
$$
\varphi_{H_{L}}=p_{L}\circ \mu_{H}\circ (H\otimes i_{L}),\;\;\;\;
\varrho_{H_{L}}=(H\otimes p_{L})\circ \delta_{H}\circ i_{L},
$$
where $p_{L}:H\rightarrow H_{L}$ and $i_{L}:H_{L}\rightarrow H$ are
the morphisms such that $\Pi_{H}^{L}=i_{L}\circ p_{L}$ and
$p_{L}\circ i_{L}=id_{H_{L}}$.

The unit constrains are:
$$
\mathfrak{l}_{M}=\varphi_{M}\circ (i_{L}\otimes M)\circ
i_{H_{L}\otimes M}:H_{L}\times M\rightarrow M,
$$

$$
\mathfrak{r}_{M}=\varphi_{M}\co c_{M,H}\co (M\otimes
(\overline{\Pi}_{H}^{L}\circ i_{L}))\co i_{M\otimes H_{L}}:M\times
H_{L}\rightarrow M.
$$

These morphisms are isomorphisms with inverses:
$$
\mathfrak{l}_{M}^{-1}=p_{H_{L}\otimes M}\circ (p_{L}\otimes
\varphi_{M})\circ ((\delta_{H}\circ \eta_{H})\otimes M):M\rightarrow
H_{L}\times M,
$$
$$
\mathfrak{r}_{M}^{-1}=p_{M\otimes H_{L}}\circ (\varphi_{M}\otimes
p_{L})\circ (H\otimes c_{H,M})\circ ((\delta_{H}\circ
\eta_{H})\otimes M):M\rightarrow M\times H_{L}.
$$

If $M$, $N$, $P$ are objects in the category $\;^{H}_{H}{\mathcal
Y}{\mathcal D}$, the associativity constrain is defined by
$$
\mathfrak{a}_{M,N,P}=p_{(M\times N)\otimes P}\circ (p_{M\otimes
N}\otimes P)\circ (M\otimes i_{N\otimes P})\circ i_{M\otimes
(N\times P)}:M\times (N\times P)\rightarrow (M\times N)\times P.
$$
and its inverse is
$$
\mathfrak{a}_{M,N,P}^{-1}=p_{M\otimes (N\times P)}\circ (M\otimes
p_{N\otimes
 P})\circ (i_{M\otimes N}\otimes P)\circ i_{(M\times N)\otimes P}:
 (M\times N)\times P \rightarrow M\times (N\times P).
$$

If $\gamma:M\rightarrow M^{\prime}$ and $\phi:N\rightarrow
N^{\prime}$ are morphisms in the category, we define
$$
\gamma\times \phi=p_{M^{\prime}\times N^{\prime}}\circ
(\gamma\otimes \phi)\circ i_{M\otimes N}:M\times N\rightarrow
M^{\prime}\times N^{\prime},
$$
that is a morphism in $\;^{H}_{H}{\mathcal Y}{\mathcal D}$ and
$$
(\gamma^{\prime}\times \phi^{\prime})\circ (\gamma\times \phi)=
(\gamma^{\prime}\circ \gamma)\times (\phi^{\prime}\circ \phi),
$$
where $\gamma^{\prime}:M^{\prime}\rightarrow M^{\prime\prime}$ and
$\phi^{\prime}:N^{\prime}\rightarrow N^{\prime\prime}$ are morphisms
in $\;^{H}_{H}{\mathcal Y}{\mathcal D}$.

Finally, the braiding is
\begin{equation}
\tau_{M,N}=p_{N\otimes M}\circ t_{M,N}\circ i_{M\otimes N}:M\times
N\rightarrow N\times M
\end{equation}
where
\begin{equation}\label{tdeYD}
t_{M,N}=(\varphi_{N}\otimes M)\circ (H\otimes c_{M,N})\circ
(\varrho_{M}\otimes N):M\otimes N\rightarrow N\otimes M.
\end{equation}
 The
morphism $\tau_{M,N}$ is a natural isomorphism with inverse:
\begin{equation}
\tau_{M,N}^{-1}=p_{M\otimes N}\circ t_{M,N}^{\prime}\circ
i_{N\otimes M}:N\times M\rightarrow M\times N
\end{equation}
where
\begin{equation}
\label{tinvdeYD} t_{M,N}^{\prime}=c_{N,M}\circ (\varphi_{N}\otimes
M)\circ (c_{N,H}\otimes M)\circ (N\otimes \lambda_{H}^{-1}\otimes
M)\circ (N\otimes \varrho_{M}).
\end{equation}

By [\cite{NikaRamon4}, Proposition 1.15] we have that given
$(M,\varphi_{M},\varrho_{M})$ in ${}_H^H\mathcal{YD},$ the morphism
$t_{M,M}:M\otimes M\rightarrow M\otimes M$ defined in (\ref{tdeYD})
by
$$t_{M,M}=(\varphi_{M}\otimes M)\circ (H\otimes c_{M,M})\circ
(\varrho_{M}\otimes M)$$
 is a weak
Yang-Baxter operator where by (\ref{tinvdeYD}) we have
$$t_{M,M}^{\prime}=c_{M,M}\circ
(\varphi_{M}\otimes M)\circ (c_{M,H}\otimes M)\circ (M\otimes
\lambda_{H}^{-1}\otimes M)\circ (M\otimes \varrho_{M})$$ and
$\nabla_{M,M}=\nabla_{M\otimes M}$.

A similar result can be obtained by working with Yetter-Drinfeld
modules associated to a weak Hopf algebra in a braided monoidal
category (see \cite{Proj} for the details).

{\bf (2)} Let $D$ be in $\mathcal{C}$
 where every idempotent morphism splits. If  $\Omega:D\otimes D\rightarrow
D\otimes D$ is an idempotent morphism such that
\begin{equation}
\label{weak-groupoid} (\Omega\otimes D)\circ (D\otimes
\Omega)=(D\otimes \Omega)\circ (\Omega\otimes D) \end{equation}
then $\Omega$ is a
weak Yang-Baxter operator where $t_{D,D}=t_{D,D}^{\prime}=\nabla_{D,
D}=\Omega$.

Then, as a consequence of (a2-1), the idempotent morphism
$\nabla_{D, D}$ of Definition \ref{wybopdef} is an example of weak
Yang-Baxter operator.

It is possible to construct more examples of this kind of weak
Yang-Baxter operators working with exact factorizations of
groupoids. Previously we recall the definition of wide subgroupoid. A groupoid $H$
is a wide subgroupoid of a groupoid $G$ if $H$ is
a subcategory of $G$ provided with a functor $F:H\rightarrow G$
which is the identity on the objects, and it induces inclusions
$hom_{H}(x,y)\subset hom_{G}(x,y)$, i.e., it has the same base, and
(perhaps) less arrows.

Let $G$ be a groupoid. An exact factorization of $G$ is a pair of
wide subgroupoids of $G$, $H$ and $V$, such that for any $\sigma\in
G_{1}$, there exist unique $\sigma_{V}\in V_{1}$, $\sigma_{H}\in
H_{1}$, such that $\sigma=\sigma_{H}\circ \sigma_{V}$.

If $G$ is a groupoid with exact factorization we define
$$\Omega: RG\otimes RG\rightarrow RG\otimes RG$$
as
$$\Omega(\sigma\otimes \tau)=\sigma_{H}\otimes \tau_{V}.$$

Then $\Omega$ is an idempotent morphism satisfying
(\ref{weak-groupoid}) and then it is a weak Yang-Baxter operator.

{\bf (3)} In this example we assume that ${\mathcal C}$ is braided.
 Let $D$ be an algebra in ${\mathcal C}$. Then the idempotent morphism
$$\Omega=\eta_{D}\otimes (\mu_{D}\circ c_{D,D}):D\otimes
D\rightarrow D\otimes D$$ does not satisfy (\ref{weak-groupoid}) but it
is a weak Yang-Baxter operator where
$$t_{D,D}=t^{\prime}_{D,D}=\nabla_{D,D}=\Omega.$$

Also, if $D$ is a coalgebra in ${\mathcal C}$, the idempotent
morphism
$$\Omega^{\prime}=\varepsilon_{D}\otimes (c_{D,D}\circ \delta_{D}):D\otimes
D\rightarrow D\otimes D$$ is a weak Yang-Baxter operator where
$$t_{D,D}=t^{\prime}_{D,D}=\nabla_{D,D}=\Omega^{\prime}.$$

\end{examples}

Now we recall the definition of weak braided bialgebra and weak
braided Hopf algebra introduced by Alonso, Gonz\'{a}lez and
Rodr\'{i}guez in \cite{NikaRamon4} (see also \cite{Accadj}). The interested reader can see the
main properties in [\cite{IND}, Section 2].

\begin{definition}
\label{WBH} A weak braided bialgebra (WBB for short) $D$  is an
object in ${\mathcal C}$ with an algebra structure $(D,
\eta_{D},\mu_{D})$ and a coalgebra structure $(D,
\varepsilon_{D},\delta_{D})$ such that there exists a weak
Yang-Baxter operator $t_{D,D}:D\ot D\rightarrow D\ot D$ with
associated idempotent $\nabla_{D,D}$ satisfying the following
conditions:
\begin{itemize}
\item[(b1)] We have
\begin{itemize}
\item[(b1-1)] $\mu_{D}\co \nabla_{D,D}=\mu_{D},$

\item[(b1-2)] $\nabla_{D,D}\co (\mu_{D}\ot D)=(\mu_{D}\ot D)\co
(D\ot \nabla_{D,D}), $

\item[(b1-3)]$\nabla_{D,D}\co (D\ot \mu_{D})=(D\ot \mu_{D})\co
(\nabla_{D,D}\ot D). $

\end{itemize}

\item[(b2)] We have

\begin{itemize}
\item[(b2-1)] $\nabla_{D,D}\co \delta_{D}=\delta_{D},$

\item[(b2-2)] $(\delta_{D}\ot D)\co \nabla_{D,D}=(D\ot
\nabla_{D,D})\co (\delta_{D}\ot D), $

\item[(b2-3)]$(D\ot \delta_{D})\co \nabla_{D,D}=(\nabla_{D,D}\ot D)\co (D\ot \delta_{D}). $

\end{itemize}

\item[(b3)] The morphisms $\mu_{D}$
and $\delta_{D}$ commute with $t_{D,D}$, i.e.,

\begin{itemize}

\item[(b3-1)] $t_{D,D}\co (\mu_{D}\ot D)=(D\ot \mu_{D})\co
(t_{D,D}\ot D)\co (D\ot t_{D,D}),$

\item[(b3-2)] $t_{D,D}\co (D\ot \mu_{D})=(\mu_{D}\ot D)\co (D\ot
t_{D,D})\co (t_{D,D}\ot D),$

\item[(b3-3)]$(\delta_{D}\ot D)\co t_{D,D}=(D\ot t_{D,D})\co
(t_{D,D}\ot D)\co (D\ot \delta_{D}),$

\item[(b3-4)]$(D\ot \delta_{D})\co t_{D,D}=(t_{D,D}\ot D)\co (D\ot
t_{D,D})\co (\delta_{D}\ot D).$

\end{itemize}

\item[(b4)] $\delta_{D}\co \mu_{D}=(\mu_{D}\ot \mu_{D})\co (D\ot
t_{D,D}\ot D)\co (\delta_{D}\ot \delta_{D}).$

\item[(b5)]$\varepsilon_{D}\co \mu_{D}\co (\mu_{D}\ot
D)=((\varepsilon_{D}\co \mu_{D})\ot (\varepsilon_{D}\co \mu_{D}))\co
(D\ot \delta_{D}\ot D)$ \item[ ]$=((\varepsilon_{D}\co \mu_{D})\ot
(\varepsilon_{D}\co \mu_{D}))\co (D\ot
(t^{\prime}_{D,D}\co\delta_{D})\ot D).$

\item[(b6)]$(\delta_{D}\ot D)\co \delta_{D}\co \eta_{D}=(D\ot
\mu_{D}\ot D)\co ((\delta_{D}\co \eta_{D}) \ot (\delta_{D}\co
\eta_{D}))$  \item[ ]$=(D\ot (\mu_{D}\co t^{\prime}_{D,D})\ot D)\co
((\delta_{D}\co \eta_{D}) \ot (\delta_{D}\co \eta_{D})).$

\end{itemize}

A weak braided bialgebra $D$ is said to be a weak braided Hopf
algebra (WBHA for short) if:
\begin{itemize}
\item[(b7)] There exists a morphism $\lambda_{D}:D\rightarrow D$
in ${\mathcal C}$ (called the antipode of $D$) satisfying:
\begin{itemize}
\item[(b7-1)] $id_D\wedge \lambda_{D}=((\varepsilon_{D}\co
\mu_{D})\ot D)\co (D\ot t_{D,D})\co ((\delta_{D}\co \eta_{D})\ot
D),$ \item[(b7-2)] $\lambda_{D}\wedge id_D=(D\ot(\varepsilon_{D}\co
\mu_{D}))\co (t_{D,D}\ot D)\co (D\ot (\delta_{D}\co \eta_{D})),$
\item[(b7-3)]$\lambda_{D}\wedge id_D\wedge \lambda_{D}=\lambda_{D}.$
\end{itemize}
\end{itemize}

Let $D$, $B$ be WBHA. We will say that $f:D\rightarrow B$ is a
morphism of WBHA if $f$ is an algebra coalgebra morphism and
$t_{B,B}\co (f\ot f)=(f\ot f)\co t_{D,D}$ and $t_{B,B}^{\prime}\co
(f\ot f)=(f\ot f)\co t^{\prime}_{D,D}$.

\end{definition}

\begin{examples}

{\bf (1)} Suppose that ${\mathcal C}$ is symmetric
and $t_{D,D}=t^{\prime}_{D,D}$ is the braiding of the category. Then if  $D$ is a WBHA
$\nabla_{D,D}=id_{D\ot D}$ and we obtain the well known definition
of weak Hopf algebra [Examples \ref{Yang-Baxter-examples}, (1)].

{\bf (2)} Now we assume that ${\mathcal C}$ is a
braided category with braiding $c$ and $t_{D,D}=c_{D,D}$ and
$t^{\prime}_{D,D}=c^{-1}_{D,D}$. Then $\nabla_{D,D}=id_{D\ot D}$ and
we say that $D$ is a weak Hopf algebra in ${\mathcal C}$ (see \cite{NikaRamon4} and \cite{IND} ).
 Obviously,
classical Hopf algebras are weak Hopf algebras in this setting and
it is not difficult to see that braided Hopf algebras considered by
Takeuchi in \cite{T} are examples of weak braided Hopf algebras.

{\bf (3)} Let the category ${\mathcal C}$ be symmetric and let $H$ be a weak Hopf algebra with invertible antipode in
 ${\mathcal C}$. We know [Examples \ref{Yang-Baxter-examples}, (1)] that the category of
 left-left Yetter-Drinfeld modules ${}_H^H\mathcal{YD}$  is a
 non-strict braided monoidal category.

An object $(A,\varphi_{A}, \varrho_{A})\in \;^{H}_{H}{\mathcal
Y}{\mathcal D}$ is called an algebra  if there exist morphisms
$u_{A}:H_{L}\rightarrow A$ and $m_{A}:A\times A\rightarrow A$ in
$\;^{H}_{H}{\mathcal Y}{\mathcal D}$ such that
\begin{equation}
m_{A}\circ (u_{A}\times A)\circ \mathfrak{l}_{A}^{-1}=id_{A}=
m_{A}\circ (A\times u_{A})\circ \mathfrak{r}_{A}^{-1},
\end{equation}
\begin{equation}
m_{A}\circ (m_{A}\times A)\circ \mathfrak{a}_{A,A,A}=m_{A}\circ
(A\times m_{A}).
\end{equation}

In a dual way, $(C,\varphi_{C}, \varrho_{C})\in \;^{H}_{H}{\mathcal
Y}{\mathcal D}$ is a coalgebra  if there exist morphisms
$e_{C}:C\rightarrow H_{L}$ and $\Delta_{C}:C\rightarrow C\times C$
in $\;^{H}_{H}{\mathcal Y}{\mathcal D}$ such that
\begin{equation}
\mathfrak{l}_{C}\circ (e_{C}\times C)\circ
\Delta_{C}=id_{C}=\mathfrak{r}_{C}\circ (C\times e_{C})\circ
\Delta_{C},
\end{equation}
\begin{equation}
(C\times \Delta_{C})\circ \Delta_{C}= \mathfrak{a}_{C,C,C}\circ
(\Delta_{C}\times C)\circ \Delta_{C}.
\end{equation}

A Hopf algebra $D$ in $\;^{H}_{H}{\mathcal Y}{\mathcal D}$ is an
algebra-coalgebra in $\;^{H}_{H}{\mathcal Y}{\mathcal D}$, $(D,
u_{D}, m_{D}, e_{D}, \Delta_{D})$ with a morphism
$\lambda_{D}:D\rightarrow D$ in $\;^{H}_{H}{\mathcal Y}{\mathcal D}$
(called the antipode) such that

\begin{itemize}

\item[(i)] $ \Delta_{D}\co m_{D}=(m_{D}\times m_{D}) \co
\mathfrak{a}_{D,D,D\times D}\co (D\times
\mathfrak{a}_{D,D,D}^{-1})\co (D\times (\tau_{D,D}\times D))\co
(D\times \mathfrak{a}_{D,D,D})\co$

\item[ ]$\hspace{2cm} \mathfrak{a}_{D,D,D\times D}^{-1}\co (\Delta_{D}\times
\Delta_{D}),$

\item[(ii)] $\Delta_{D}\co u_{D}=(u_{D}\times u_{D})\co
\mathfrak{l}_{H_{L}}^{-1},$

\item[(iii)] $m_{D}\co (D\times \lambda_{D})\co \Delta_{D}=m_{D}\co
(\lambda_{D}\times D)\co \Delta_{D}=\mathfrak{r}_{D} \co
(u_{D}\times e_{D})\co \mathfrak{l}_{D}^{-1}.$
\end{itemize}

If we define $\eta_{D}=u_{D}\co p_{L}\co \eta_{H}$,
$\mu_{D}=m_{D}\co p_{D\otimes D}$,
$\varepsilon_{D}=\varepsilon_{H}\co i_{L}\co
 e_{D}$ and $\delta_{D}=i_{D\otimes D}\co \Delta_{D}$, we have that
 $(D, \eta_{D}, \mu_{D}, \varepsilon_{D}, \delta_{D}, \lambda_{D})$
is a WBHA in ${\mathcal C}$ [\cite{NikaRamon4}, Corollary 2.14].
Note that this example is non-trivial, i.e., $D$ is not a Hopf
algebra neither a weak Hopf algebra in the usual sense. For example, if we assume $\varepsilon_{D}\co
\mu_{D}=\varepsilon_{D}\ot \varepsilon_{D}$ we obtain that
$\Pi_{H}^{L}=\varepsilon_{H}\ot \eta_{H}$, or equivalently, $H$ is a
Hopf algebra in ${\mathcal C}$ [\cite{NikaRamon4}, Remark 2.15]. By
an analogous calculus, if $\eta_{D}\ot \eta_{D}=\delta_{D}\co
\eta_{D}$, we obtain that $H$ is a Hopf algebra. Moreover,  if
$\lambda_{D}\wedge id_{D}=\varepsilon_{D}\ot \eta_{D}$ we have
$u_{D}\co e_{D}=\eta_{D}\co \varepsilon_{D}$ and then
$id_{H_{L}}=p_{L}\co \eta_{H}\co \varepsilon_{H}\co i_{L}$.
Therefore, $\Pi_{H}^{L}=\varepsilon_{H}\ot \eta_{H}$ and we obtain
that $H$ also is a Hopf algebra. Finally, $D$ is not a weak Hopf
algebra since the condition (i) is equivalent to
$$
\Delta_{D}\co m_{D}=p_{D, D}\co (\mu_{D}\ot \mu_{D})\co (D\ot
t_{D,D}\ot D)\co (\delta_{D}\ot \delta_{D})\co i_{D, D},
$$
[\cite{NikaRamon4}, Proposition 2.8] and this one does not imply
$\delta_{D}\co \mu_{D}=(\mu_{D}\ot \mu_{D})\co (D\ot c_{D,D}\ot
D)\co (\delta_{D}\ot \delta_{D})$ where $c_{D,D}$ is the symmetric
braiding of ${\mathcal C}$.

\end{examples}

\begin{apart}

Let $D$ be a WBB.  The following identities hold (see \cite{Aus})
\begin{equation}\label{b3-1}
t_{D,D}\co (\eta_{D}\ot D)=\nabla_{D,D}\co (D\ot
\eta_{D})=t^{\prime}_{D,D}\co (\eta_{D}\ot D),
\end{equation}
\begin{equation}\label{b3-2}
t_{D,D}\co (D\ot \eta_{D})=\nabla_{D,D}\co (\eta_{D}\ot
D)=t^{\prime}_{D,D}\co (D\ot \eta_{D}),
\end{equation}
\begin{equation}\label{b3-4}
(D\ot \varepsilon_{D})\co t_{D,D} =( \varepsilon_{D}\ot D)\co
\nabla_{D,D}=(D\ot \varepsilon_{D})\co t^{\prime}_{D,D},
\end{equation}
\begin{equation}\label{b3-5}
(\varepsilon_{D}\ot D)\co t_{D,D} =(D\ot \varepsilon_{D})\co
\nabla_{D,D}=(\varepsilon_{D}\ot D)\co t^{\prime}_{D,D}.
\end{equation}

Moreover, we have
\begin{equation}
\label{b1}
 t^{\prime}_{D,D}\co (\mu_{D}\ot D)=(D\ot \mu_{D})\co
(t^{\prime}_{D,D}\ot D)\co (D\ot t^{\prime}_{D,D}),
\end{equation}
\begin{equation}
\label{b2} t^{\prime}_{D,D}\co (D\ot \mu_{D})=(\mu_{D}\ot D)\co
(D\ot t^{\prime}_{D,D})\co (t^{\prime}_{D,D}\ot D),
\end{equation}
\begin{equation}
\label{b3} (\delta_{D}\ot D)\co t^{\prime}_{D,D}=(D\ot
t^{\prime}_{D,D})\co (t^{\prime}_{D,D}\ot D)\co (D\ot \delta_{D}),
\end{equation}
\begin{equation}
\label{b4} (D\ot \delta_{D})\co
t^{\prime}_{D,D}=(t^{\prime}_{D,D}\ot D)\co (D\ot
t^{\prime}_{D,D})\co (\delta_{D}\ot D).
\end{equation}

\end{apart}

\begin{apart}
Let $D$ be a WBHA. The morphisms $\Pi_D^L$ (target), $\Pi_D^R$
(source), $\overline{\Pi}_D^L$ and $\overline{\Pi}_D^R$ are defined
as follows:
\[\Pi_D^L= ((\varepsilon_D\co\mu_D)\ot D)\co(t_{D,D}\ot D)\co((\delta_D\co\eta_D)\ot
D),\]
\[\Pi_D^R= (D\ot(\varepsilon_D\co\mu_D))\co(D\ot
t_{D,D})\co((D\ot(\delta_D\co\eta_D)),\]
\[\overline{\Pi}_D^L= (D\ot(\varepsilon_D\co\mu_D)) \co((\delta_D\co\eta_D)\ot D),\]
\[\overline{\Pi}_D^R= ((\varepsilon_D\co\mu_D)\ot D) \co(D\ot(\delta_D\co\eta_D)).\]

It is easy to prove that they are idempotent and leave the unit and
the counit invariant. Moreover, they  satisfy:
\begin{equation}\label{landaconvolucion}
\Pi_{D}^{L}=id_D\wedge \lambda_{D}, \;\;\;
\Pi_{D}^{R}=\lambda_{D}\wedge id_D,\;\;
\lambda_{D}=\lambda_{D}\wedge \Pi_{D}^{L}=\Pi_{D}^{R}\wedge
\lambda_{D},
\end{equation}
and applying (b4) we get
\begin{equation}\label{Piconvolucion}
 id_D\wedge \lambda_{D}\wedge id_D=\Pi_{D}^{L}\wedge
id_D=id_D\wedge \Pi_{D}^{R}=id_D.
\end{equation}

Moreover, the following equalities are satisfied [\cite{IND},
Proposition 2.10]
\begin{equation}\label{PiePibarra1}
\Pi_{H}^{L}\co
\overline{\Pi}_{D}^{L}=\Pi_{D}^{L},\;\;\;\Pi_{D}^{L}\co
\overline{\Pi}_{D}^{R}=\overline{\Pi}_{D}^{R},\;\;\;\overline{\Pi}_{D}^{L}\co
\Pi_{D}^{L}=\overline{\Pi}_{D}^{L},\;\;\;\overline{\Pi}_{D}^{R}\co
\Pi_{D}^{L}=\Pi_{D}^{L},
\end{equation}
\begin{equation}\label{PiePibarra2}
\Pi_{D}^{R}\co \overline{\Pi}_{D}^{L}=\overline{\Pi}_{D}^{L},\;\;\;
\Pi_{D}^{R}\co
\overline{\Pi}_{D}^{R}=\Pi_{D}^{R},\;\;\;\overline{\Pi}_{D}^{L}\co
\Pi_{D}^{R}=\Pi_{D}^{R},\;\;\; \overline{\Pi}_{D}^{R}\co
\Pi_{D}^{R}=\overline{\Pi}_{D}^{R},
\end{equation}
\begin{equation}\label{Pielambda}
\Pi_{D}^{L}\co \lambda_{D}=\Pi_{D}^{L}\co \Pi_{D}^{R}=
\lambda_{D}\co \Pi_{D}^{R},\;\;\;\Pi_{D}^{R}\co
\lambda_{D}=\Pi_{D}^{R}\co \Pi_{D}^{L}= \lambda_{D}\co \Pi_{D}^{L},
\end{equation}
\begin{equation}\label{Pielambda2}
\Pi_{D}^{L}=\overline{\Pi}_{D}^{R}\co \lambda_{D}=\lambda_{D}
\co\overline{\Pi}_{D}^{L},\;\;\;\Pi_{D}^{R}=
\overline{\Pi}_{D}^{L}\co \lambda_{D}=\lambda_{D} \co
\overline{\Pi}_{D}^{R}.
\end{equation}

Finally by [\cite{IND}, Proposition 2.20] we have that the antipode
is antimultiplicative, anticomultiplicative and leaves the unit and
the counit invariant,i.e.:
\begin{equation}
\label{ant-mu} \lambda_{D}\co \mu_{D}=\mu_{D}\co t_{D,D}\co
(\lambda_{D}\ot \lambda_{D}),
\end{equation}
\begin{equation}
\label{ant-delta} \delta_{D}\co \lambda_{D}=(\lambda_{D}\ot
\lambda_{D})\co t_{D,D}\co \delta_{D},
\end{equation}
\begin{equation}
\label{ant-delta-epsilon} \lambda_{D}\co \eta_{D}=\eta_{D},\;\;
\varepsilon_{D}\co\lambda_{D} =\varepsilon_{D}.
\end{equation}

If $f:D\rightarrow B$ is a morphism of weak braided Hopf algebras,
by (\ref{defnabla}) we obtain $\nabla_{B,B}\co (f\ot f)=(f\ot f)\co
\nabla_{D,D}$. It is not difficult to see that, if $f:D\rightarrow
B$ is a morphism of weak braided Hopf algebras, then $f\co
\lambda_{D}=\lambda_{B}\co f$ (see \cite{NikaRamon4} for details).

\end{apart}

Once the general framework is stated  we introduce the concept of
weak operator, that turns out to be essential to define the notion
of Yetter-Drindel'd module  in a general monoidal context. Actually,
it will allow us  to conceive the collection of Yetter-Drinfeld
modules as the objects of a monoidal category, being this structure
relevant in order to get an operative theory, it is said, a general
framework where formal manipulations and effective calculations can
be done. It will be obvious from the definition below that weak
operators constitute a generalization of the notion of weak
Yang-Baxter operator.

\begin{definition}
\label{WO} Let $D$ be a WBHA and let $M$ be an object of
$\CC.$ A weak operator between $M$ and $D,$ (from now on  referred
as $(M,D)$-WO) is defined as a quadruple $(r_M,r_M^\prime,s_M,
s_M^\prime)$ comprised of four morphisms in $\CC$:

\[r_M: M\ot D\rightarrow D\ot M, \qquad r_M^\prime: D\ot M \rightarrow M\ot D,\]

\[s_M: D\ot M \rightarrow M\ot D, \qquad s_M^\prime: M\ot D\rightarrow D\ot M,\]

satisfying the following compatibility conditions:

\begin{itemize}

\item[(c1)] Compatibility with the weak Yang-Baxter operator:

\begin{itemize}

\item[(c1-1)]  $(D\ot r_M)\co(r_M\ot D)\co(M\ot t_{D,D})=(t_{D,D}\ot M)\co(D\ot
r_M)\co(r_M\ot D),$

\item[(c1-2)] $(r_M^\prime\ot D)\co(D\ot r_M^\prime)\co(t_{D,D}\ot M)=
(M\ot t_{D,D})\co(r_M^\prime\ot D)\co(D\ot r_M^\prime),$

\item[(c1-3)] $(s_M\ot D)\co(D\ot s_M)\co(t_{D,D}\ot M)=
(M\ot t_{D,D})\co(s_M\ot D)\co(D\ot s_M),$

\item[(c1-4)]  $(D\ot s_M^\prime)\co(s_M^\prime\ot D)\co(M\ot t_{D,D})=(t_{D,D}\ot
M)\co(D\ot s_M^\prime)\co(s_M^\prime\ot D).$

\end{itemize}

The analogous equalities with $t^{\prime}_{D,D}$ instead of
$t_{D,D}$ are also required to be satisfied.

\item[(c2)] Mixed Yang-Baxter equations
\begin{itemize}

\item[(c2-1)] $(r_M^\prime\ot D)\co(D\ot s_M)\co(t_{D,D}\ot M)=(M\ot
t_{D,D})\co(s_M\ot D)\co(D\ot r_M^\prime),$

\item[(c2-2)] $(s_M\ot D)\co(D\ot r_M^\prime)\co(t^{\prime}_{D,D}\ot M)=(M\ot
t^{\prime}_{D,D})\co(r_M^\prime\ot D)\co(D\ot s_M),$

\item[(c2-3)] $(D\ot s_M^\prime)\co(r_M\ot D)\co(M\ot t_{D,D})= (t_{D,D}\ot M)\co
(D\ot r_M)\co (s_M^\prime\ot D),$

\item[(c2-4)] $(D\ot r_M)\co(s_M^\prime\ot D)\co(M\ot t^{\prime}_{D,D})=
(t^{\prime}_{D,D}\ot M)\co(D\ot s_M^\prime)\co(r_M\ot D).$

\end{itemize}

We want to point out that in this case, as in general for all the
mixed equations along the paper, we cannot replace $t_{D,D}$ by
$t^{\prime}_{D,D}$ or $t^{\prime}_{D,D}$ by $t_{D,D}$.

\item[(c3)] The morphisms: $\nabla_{r_M} := r_M^\prime\co r_M$, $\nabla_{r_M^\prime} := r_M\co r_M^\prime$,
 $\nabla_{s_M} := s_M^\prime\co s_M$ and $\nabla_{s_M^\prime} := s_M\co s_M^\prime$ satisfy
\begin{itemize}

\item[(c3-1)] $\nabla_{r_M}=(((\varepsilon_D\ot M)\co r_M)\ot
D)\co(M\ot \delta_D)=(M\ot\mu_D)\co(((r_M^\prime\co(\eta_D\ot M))\ot
D),$

\item[(c3-2)] $\nabla_{r_M^\prime}=(D\ot((M\ot\varepsilon_D)\co
r_M^\prime))\co(\delta_D\ot M)=(\mu_D\ot M)\co(D\ot
(r_M\co(M\ot\eta_D))),$

\item[(c3-3)] $\nabla_{s_M}=(D\ot((M\ot\varepsilon_D)\co
s_M))\co(\delta_D\ot M)=(\mu_D\ot M)\co(D\ot
(s_M^\prime\co(M\ot\eta_D))),$

\item[(c3-4)] $\nabla_{s_M^\prime}=(((\varepsilon_D\ot M)\co s_M^\prime)\ot
D)\co(M\ot \delta_D)=(M\ot\mu_D)\co(((s_M\co(\eta_D\ot M))\ot
D).$

\end{itemize}

\item[(c4)] Compatibility with the (co) multiplication:
\begin{itemize}

\item[(c4-1)] $r_M\co(M\ot \mu_D)=(\mu_D\ot M)\co(D\ot r_M)\co(r_M\ot D),$

\item[(c4-2)] $r_M^\prime\co(\mu_D\ot M)=(M\ot \mu_D)\co(
r_M^\prime\ot D)\co(D\ot r_M^\prime),$

\item[(c4-3)] $(D\ot r_M)\co(r_M\ot D)\co(M\ot \delta_D)=(\delta_D\ot M)\co r_M,$

\item[(c4-4)] $(r_M^\prime\ot D)\co(D\ot r_M^\prime)\co(\delta_D\ot M)=(M\ot
\delta_D)\co r_M^\prime,$

\item[(c4-5)] $s_M\co(\mu_D\ot M)=(M\ot \mu_D)\co(
s_M\ot D)\co(D\ot s_M),$

\item[(c4-6)] $s_M^\prime\co(M\ot \mu_D)=(\mu_D\ot M)\co(D\ot s_M^\prime)\co(s_M^\prime\ot D),$

\item[(c4-7)] $(s_M\ot D)\co(D\ot s_M)\co(\delta_D\ot M)=(M\ot
\delta_D)\co s_M,$

\item[(c4-8)] $(D\ot s_M^\prime)\co(s_M^\prime\ot D)\co(M\ot \delta_D)=(\delta_D\ot M)\co s_M^\prime.$

\end{itemize}

\item[(c5)] Compatibility with the antipode:

\begin{itemize}

\item[(c5-1)]  $(M\ot\lambda_D)\co\nabla_{r_M}=\nabla_{r_M}\co(M\ot\lambda_D),$

\item[(c5-2)]$(\lambda_D\ot M)\co \nabla_{r_M^\prime}=\nabla_{r_M^\prime}\co(\lambda_D\ot M),$

\item[(c5-3)]$(\lambda_D\ot M)\co \nabla_{s_M}=\nabla_{s_M}\co(\lambda_D\ot M),$

\item[(c5-4)]  $(M\ot\lambda_D)\co\nabla_{s_M^\prime}=\nabla_{s_M^\prime}\co(M\ot\lambda_D).$

\end{itemize}

\end{itemize}

\end{definition}

\begin{remark}\label{redundante1}
 As a consequence of  Definition \ref{WO} the following equalities hold for a WBHA $D$:

\begin{equation}\label{nabla_Dbaixaporr} (D\ot r_M)\co (r_M\ot D)\co(M\ot
\nabla_{D,D})=(\nabla_{D,D}\ot M)\co(D\ot r_M)\co(r_M\ot D),
\end{equation}

\begin{equation}\label{nabla_Dbaixaporrprima} (r_M^\prime\ot D)\co (D\ot
r_M^\prime)\co(\nabla_{D,D}\ot M)=(M\ot
\nabla_{D,D})\co(r_M^\prime\ot D)\co(D\ot r_M^\prime),
\end{equation}

\begin{equation}\label{nabla_Dbaixaporrsprima} (D\ot r_M)\co (s_M^\prime\ot
D)\co(M\ot \nabla_{D,D})=(\nabla_{D,D}\ot M)\co(D\ot
r_M)\co(s_M^\prime\ot D),
\end{equation}

\begin{equation}\label{nabla_Dbaixaporsprimar} (s_M\ot D)\co (D\ot
r_M^\prime)\co(\nabla_{D,D}\ot M)=(M\ot \nabla_{D,D})\co(s_M\ot
D)\co(D\ot r_M^\prime),
\end{equation}

\begin{equation}\label{epsilontrocalado}
(M\ot \varepsilon_D)\co\nabla_{r_M}=(\varepsilon_D\ot M)\co r_M;
\quad (\varepsilon_D\ot M)\co\nabla_{r_M^\prime}=(M\ot
\varepsilon_D)\co r_M^\prime,
\end{equation}

\begin{equation}\label{nablatrocalado}
\nabla_{r_M}\co(M\ot\eta_D)=r_M^\prime\co(\eta_D\ot M); \quad
\nabla_{r_M^\prime}\co(\eta_D\ot M)=r_M\co(M\ot\eta_D),
\end{equation}

\begin{equation}\label{nablabaixapormu}
\nabla_{r_M}\co(M\ot\mu_D)=(M\ot\mu_D)\co(\nabla_{r_M}\ot D),
\end{equation}

\begin{equation}\label{nablaprimabaixapormu}
\nabla_{r_M^\prime}\co(\mu_D\ot M)=(\mu_D\ot
M)\co(D\ot \nabla_{r_M^\prime}),
\end{equation}

\begin{equation}\label{nablabaixapordelta}
(\nabla_{r_M}\ot D)\co(M\ot\delta_D)=(M\ot\delta_D)\co\nabla_{r_M},
\end{equation}

\begin{equation}\label{nablaprimabaixapordelta}
(D\ot \nabla_{r_M^\prime})\co(\delta_D\ot M)=(\delta_D\ot
M)\co \nabla_{r_M^\prime},
\end{equation}

\begin{equation}\label{deltacomenabla_r}
(D\ot\nabla_{r_M})\co(r_M\ot D)\co(M\ot\delta_D)=(r_M\ot
D)\co(M\ot\delta_D),
\end{equation}

\begin{equation}\label{deltacomenabla_rprima}
(\nabla_{r_M^\prime}\ot D)\co(D\ot r_M^\prime)\co(\delta_D\ot
M)=(D\ot r_M^\prime)\co(\delta_D\ot M),
\end{equation}

\begin{equation}\label{nabla_Dbaixapors} (s_M\ot D)\co (D\ot
s_M)\co(\nabla_{D,D}\ot M)=(M\ot
\nabla_{D,D})\co(s_M\ot D)\co(D\ot s_M),
\end{equation}

\begin{equation}\label{nabla_Dbaixaporsprima} (D\ot s_M^\prime)\co (s_M^\prime\ot D)\co(M\ot
\nabla_{D,D})=(\nabla_{D,D}\ot M)\co(D\ot s_M^\prime)\co(s_M^\prime\ot D),
\end{equation}

\begin{equation}\label{nabla_Dbaixaporrprimas} (r_M^\prime\ot D)\co (D\ot
s_M)\co(\nabla_{D,D}\ot M)=(M\ot \nabla_{D,D})\co(r_M^\prime\ot
D)\co(D\ot s_M),
\end{equation}

\begin{equation}\label{nabla_Dbaixaporrprimas} (D\ot s_M^\prime)\co (r_M\ot
D)\co(M\ot \nabla_{D,D})=(\nabla_{D,D}\ot M)\co(D\ot
s_M^\prime)\co(r_M\ot D),
\end{equation}

\begin{equation}\label{epsilontrocaladoparas}
(\varepsilon_D\ot M)\co\nabla_{s_M}=(M\ot
\varepsilon_D)\co s_M;
\quad (M\ot \varepsilon_D)\co\nabla_{s_M^\prime}=(\varepsilon_D\ot M)\co s_M^\prime,
\end{equation}

\begin{equation}\label{nablatrocaladoparas}
 \nabla_{s_M^\prime}\co(M\ot\eta_D)=s_M\co(\eta_D\ot M);
 \quad
\nabla_{s_M}\co(\eta_D\ot M)=s_M^\prime\co(M\ot\eta_D),
\end{equation}

\begin{equation}\label{nablaprimabaixapormuparas}
\nabla_{s_M}\co(\mu_D\ot M)=(\mu_D\ot
M)\co(D\ot \nabla_{s_M}),
\end{equation}

\begin{equation}\label{nablabaixapormuparas}
\nabla_{s_M^\prime}\co(M\ot\mu_D)=(M\ot\mu_D)\co(\nabla_{s_M^\prime}\ot D),
\end{equation}

\begin{equation}\label{nablaprimabaixapordeltaparas}
(D\ot{s_M})\co(\delta_D\ot M)=(\delta_D\ot
M)\co\nabla_{s_M},
\end{equation}

\begin{equation}\label{nablabaixapordeltaparas}
(\nabla_{s_M^\prime}\ot D)\co(M\ot\delta_D)=(M\ot\delta_D)\co\nabla_{s_M^\prime},
\end{equation}

\begin{equation}\label{deltacomenabla_sprima}
(D\ot\nabla_{s_M^\prime})\co(s_M^\prime\ot D)\co(M\ot\delta_D)=(s_M^\prime\ot
D)\co(M\ot\delta_D),
\end{equation}

\begin{equation}\label{deltacomenabla_s}
(\nabla_{s_M}\ot D)\co(D\ot s_M)\co(\delta_D\ot
M)=(D\ot s_M)\co(\delta_D\ot M).
\end{equation}

Note that $(r_M,r_M^\prime, s_M, s_M^\prime)$ constitutes an
$(M,D)$-WO iff so does $(s_M^\prime, s_M, r_M^\prime, r_M).$

\end{remark}

\begin{proposition} \label{cancelation}Let $D$ be a WBHA and let $M$ be an object
of $\CC$ such that $(r_M, r_M^\prime, s_M, s_M^\prime)$ constitutes
an $(M,D)$-WO. Then it holds that:

\begin{itemize}

\item[(i)] The morphisms $\nabla_{r_M}$, $\nabla_{r_M^\prime}$, $\nabla_{s_M}$ and
$\nabla_{s_M^\prime}$ are idempotent.

\item[(ii)] Cancelation laws:

\item[]
\begin{equation}
\label{newrM}
r_M=\nabla_{r_M^\prime}\co r_M=r_M\co\nabla_{r_M},
\end{equation}

\item[]
\begin{equation}
\label{newrMprima}
r_M^\prime=r_M^\prime\co\nabla_{r_M^\prime}=\nabla_{r_M}\co
r_M^\prime,
\end{equation}

\item[]
\begin{equation}
\label{newsM}
s_M=\nabla_{s_M^\prime}\co
s_M=s_M\co\nabla_{s_M},
\end{equation}

\item[]
\begin{equation}
\label{newsMprima}
s_M^\prime=s_M^\prime\co\nabla_{s_M^\prime}=\nabla_{s_M}\co s_M^\prime.
\end{equation}

\end{itemize}

\end{proposition}

\begin{proof}
For (i):

\begin{itemize}

\item[ ]$\hspace{0.38cm} \nabla_{r_M}\co\nabla_{r_M}$
\item[ ]$=(((\varepsilon_D\ot M)\co r_M)\ot D)\co(((\varepsilon_D\ot M)\co r_M)\ot \delta_D)\co(M\ot\delta_D)$
\item[ ]$=(\varepsilon_D\ot\varepsilon_D\ot M\ot D)\co(\delta_D\ot M\ot D)\co(r_M\ot D)\co(M\ot\delta_D)$
\item[ ]$=\nabla_{r_M}$

\end{itemize}

by coassociativity and (c4-3). The proofs for
the remaining morphisms are analogous.

 To prove (ii) it suffices to use the suitable characterization of the corresponding
morphisms and then apply the compatibility with the
(co)multiplication. We write the first equality of (\ref{newrM}) to illustrate the procedure:

\begin{itemize}

\item[ ]$\hspace{0.38cm} r_M$
\item[ ]$=r_M\co (M\ot (\mu_D\co (D\ot \eta_D))$
\item[ ]$=(\mu_D\ot M)\co (D\ot r_M)\co (r_M\ot \eta_D)$
\item[ ]$=\nabla_{r_M^\prime}\co r_M$

\end{itemize}

\end{proof}

\begin{proposition}\label{nablatypecongruence} Let $D$ be a WBHA, $M$ any object of
the category and $(r_M, r_M^\prime, s_M, s_M^\prime)$  an
$(M,D)$-WO. Then we have:

\begin{equation}\label{tnablarM}
(D\ot\nabla_{r_M})\co(r_M\ot D)\co(M\ot t_{D,D})=(r_M\ot D)\co(M\ot
t_{D,D})\co(\nabla_{r_M}\ot D),
\end{equation}

\begin{equation}\label{tnablarprimaM}
(t_{D,D}\ot M)\co(D\ot r_M)\co(\nabla_{r_M^\prime}\ot D)=(D\ot
\nabla_{r_M^\prime})\co(t_{D,D}\ot M)\co(D\ot r_M),
\end{equation}

\begin{equation}\label{tnablaprimarprimaM}
(\nabla_{r_M^\prime}\ot D)\co(D\ot r_M^\prime)\co(t_{D,D}\ot M)= (D\ot
r_M^\prime)\co(t_{D,D}\ot M)\co(D\ot\nabla_{r_M^\prime}),
\end{equation}

\begin{equation}\label{tnablaerprimaM}
(M\ot t_{D,D})\co(r_M^\prime\ot
D)\co(D\ot\nabla_{r_M})=(\nabla_{r_M}\ot D)\co(M\ot
t_{D,D})\co(r_M^\prime\ot D),
\end{equation}

\begin{equation}\label{tnablasM}
(\nabla_{s_M}\ot D)\co(D\ot s_M)\co(t_{D,D}\ot M)= (D\ot
s_M)\co(t_{D,D}\ot M)\co(D\ot\nabla_{s_M}),
\end{equation}

\begin{equation}\label{tnablasprimaM}
(M\ot t_{D,D})\co(s_M\ot
D)\co(D\ot\nabla_{s_M^\prime})=(\nabla_{s_M^\prime}\ot D)\co(M\ot
t_{D,D})\co(s_M\ot D),
\end{equation}

 \begin{equation}\label{tnablaprimasprimaM}
(D\ot\nabla_{s_M^\prime})\co(s_M^\prime\ot D)\co(M\ot t_{D,D})=(s_M^\prime\ot D)\co(M\ot
t_{D,D})\co(\nabla_{s_M^\prime}\ot D),
\end{equation}

 \begin{equation}\label{tnablaesprimaM}
(t_{D,D}\ot M)\co(D\ot s_M^\prime)\co(\nabla_{s_M}\ot D)=(D\ot
\nabla_{s_M})\co(t_{D,D}\ot M)\co(D\ot s_M^\prime).
\end{equation}

The equalities remain true if we change  $t_{D,D}$ by
$t^{\prime}_{D,D}$.

\end{proposition}

\begin{proof}

For (\ref{tnablarM});

\begin{itemize}

\item[ ]$\hspace{0.38cm} (D\ot\nabla_{r_M})\co(r_M\ot D)\co(M\ot t_{D,D}) $

\item[ ]$= (D\ot((\varepsilon_D\ot M)\co r_M)\ot D)\co(r_M\ot\delta_D)\co(M\ot
t_{D,D})$

\item[ ]$=(D\ot((\varepsilon_D\ot M)\co r_M)\ot D)\co(r_M\ot D\ot D)\co(M\ot
t_{D,D}\ot D)\co(M\ot D\ot t_{D,D})$
\item[]$\hspace{0.38cm}\co(M\ot \delta_D\ot D)$

\item[ ]$=(((D\ot\varepsilon_D)\co t_{D,D})\ot M\ot D)\co(D\ot r_M\ot D)\co(r_M\ot
t_{D,D})\co(M\ot\delta_D\ot D)$

\item[ ]$=(((\varepsilon_D\ot D)\co \nabla_{D,D})\ot M\ot D)\co(D\ot r_M\ot
D)\co(r_M\ot t_{D,D})\co(M\ot\delta_D\ot D)$

\item[ ]$=(\varepsilon_D\ot r_M\ot D)\co(r_M\ot D\ot D)\co(M\ot \nabla_{D,D}\ot
D)\co(M\ot D\ot t_{D,D})\co(M\ot\delta_D\ot D)$

\item[ ]$=(\varepsilon_D\ot r_M\ot D)\co(r_M\ot t_{D,D})\co(M\ot\delta_D\ot D)$

\item[ ]$=(r_M\ot D)\co(M\ot t_{D,D})\co(\nabla_{r_M}\ot D),$

\end{itemize}

where we used (c3-1), the conditions (b3-4) and (c1-1), the properties of the weak
Yang-Baxter operator and the equalities (\ref{b3-4})  and (\ref{nabla_Dbaixaporr}).

The proof of the remaining equalities follows a similar procedure.

\end{proof}

\begin{proposition}\label{mixednablatypecongruence} Let $D$ be a WBHA, $M$ any object of
the category and $(r_M, r_M^\prime, s_M, s_M^\prime)$  an
$(M,D)$-WO. Then we have:

\begin{equation}\label{trMnablasM}
(t_{D,D}\ot M)\co(D\ot r_M)\co(\nabla_{s_M}\ot
D)=(D\ot\nabla_{s_M})\co(t_{D,D}\ot M)\co(D\ot r_M),
\end{equation}

\begin{equation}\label{trMnablasprimaM}
(r_M\ot D)\co(M\ot t_{D,D})\co(\nabla_{s_M^\prime}\ot
D)=(D\ot\nabla_{s_M^\prime})\co(r_M\ot D)\co(M\ot t_{D,D}),
\end{equation}

\begin{equation}\label{tprimasprimaMnablarM}
(s_M^\prime\ot D)\co(M\ot t^{\prime}_{D,D})\co(\nabla_{r_M}\ot
D)=(D\ot\nabla_{r_M})\co(s_M^\prime\ot D)\co(M\ot t^{\prime}_{D,D}),
\end{equation}

\begin{equation}\label{tprimasprimaMnablarprimaM}
(t^{\prime}_{D,D}\ot M)\co(D\ot
s_M^\prime)\co(\nabla_{r_M^\prime}\ot
D)=(D\ot\nabla_{r_M^\prime})\co(t^{\prime}_{D,D}\ot M)\co(D\ot
s_M^\prime).
\end{equation}

\end{proposition}

\begin{proof}

We will show (\ref{trMnablasM}):

\begin{itemize}

\item[ ]$\hspace{0.38cm} (t_{D,D}\ot M)\co(D\ot r_M)\co(\nabla_{s_M}\ot D) $
\item[ ]$=(D\ot \mu_D \ot M)\co (t_{D,D}\ot D\ot M)\co (D\ot t_{D,D}\ot M)\co (D\ot D\ot r_M)\co (D\ot (s_M^\prime\co(M\ot\eta_D))\ot D)$
\item[ ]$=(D\ot \mu_D \ot M)\co (t_{D,D}\ot s_M^\prime)\co (D\ot r_M\ot D)\co (D\ot M\ot (t_{D,D}\co (\eta_D\ot D)))$
\item[ ]$=(D\ot \mu_D \ot M)\co (t_{D,D}\ot s_M^\prime)\co (D\ot r_M\ot D)\co (D\ot M\ot (\nabla_{D,D}\co (D\ot \eta_D)))$
\item[ ]$=(D\ot \mu_D \ot M)\co (t_{D,D}\ot D\ot M)\co (D\ot \nabla_{D,D}\ot M)\co (D\ot D\ot (s_M^\prime\co(M\ot\eta_D)))\co (D\ot r_M)$
\item[ ]$=(D\ot\nabla_{s_M})\co(t_{D,D}\ot M)\co(D\ot r_M),$

\end{itemize}

where we used (c3-4), the conditions (b1-1) and (c2-3), the properties of the weak
Yang-Baxter operator and the equalities (\ref{b3-1}) and (\ref{nabla_Dbaixaporrprimas}).

\end{proof}

\begin{proposition}\label{Y-Btypecongruence} Let $D$ be a WBHA, $M$ any object of
the category and $(r_M, r_M^\prime, s_M, s_M^\prime)$  an
$(M,D)$-WO. Then it holds that:

\begin{itemize}

\item[(i)]
\begin{equation}\label{rMtrprimaM}
(r_M\ot D)\co (M\ot t_{D,D})\co(r_M^\prime\ot D)=(D\ot
r_M^\prime)\co(t_{D,D}\ot M)\co(D\ot r_M),
\end{equation}

\begin{equation}\label{sMtprimasprimaM}
(D\ot s_M)\co(t_{D,D}\ot M)\co(D\ot
s_M^\prime)=(s_M^\prime\ot D)\co(M\ot t_{D,D})\co(s_M\ot D).
\end{equation}

The previous equalities remain true if we change $t_{D,D}$ by
$t^{\prime}_{D,D}$.

\item[(ii)]

\begin{equation}\label{rara1}
(r_M\ot D)\co(M\ot t_{D,D})\co(s_M\ot D)=(D\ot s_M)\co(t_{D,D}\ot
M)\co(D\ot r_M),
\end{equation}

\begin{equation}\label{rara2}
(s_M^\prime\ot D)\co(M\ot t^{\prime}_{D,D})\co(r_M^\prime\ot
D)=(D\ot r_M^\prime)\co(t^{\prime}_{D,D}\ot M)\co(D\ot s_M^\prime).
\end{equation}

\end{itemize}

\end{proposition}

\begin{proof}
We prove the first equality of (i), the remaining being analogous:

\begin{itemize}

\item[ ]$\hspace{0.38cm}  (r_M\ot D)\co(M\ot t_{D,D})\co(r_M^\prime\ot D)$
\item[ ]$=(r_M\ot D)\co(M\ot t_{D,D})\co((\nabla_{r_M}\co r_M^\prime)\ot D)$
\item[ ]$=(D\ot \nabla_{r_M})\co(r_M\ot D)\co(M\ot t_{D,D})\co(r_M^\prime\ot D)$
\item[ ]$=(D\ot r_M^\prime)\co(t_{D,D}\ot M)\co(D\ot r_M)\co((r_M\co r_M^\prime)\ot D)$
\item[ ]$=(D\ot r_M^\prime)\co(t_{D,D}\ot M)\co(D\ot r_M)\co(\nabla_{r_M^\prime}\ot D)$
\item[ ]$=(D\ot (r_M^\prime\co\nabla_{r_M^\prime}))\co(t_{D,D}\ot M)\co(D\ot r_M)$
\item[ ]$=(D\ot r_M^\prime)\co(t_{D,D}\ot M)\co(D\ot r_M),$

\end{itemize}

In the above equalities, the first and the last ones follow by part
(ii) of Proposition \ref{cancelation}, the second and the fifth by
(\ref{tnablarM}) and (\ref{tnablarprimaM}), respectively.
In the third we use (c1-1)
 and the fourth follows by the definition of $\nabla_{r_M^\prime}$.

The proof of (ii) is analogous to the one of (i) but applying (\ref{trMnablasM})
 and (\ref{trMnablasprimaM}) instead of (\ref{tnablarM}) and (\ref{tnablarprimaM})
and the condition (c2-3).
\end{proof}

\begin{remark}\label{exdeWO}
In view of Definition \ref{WO},  it follows that if $M=D$ is a WBHA
in $\CC,$ the associated weak Yang-Baxter operator $t_{D,D}$ is an
example of $(D,D)$-WO with $r_M=s_M=t_{D,D}$ and
$r_M^\prime=s_M^\prime=t^{\prime}_{D,D};$ the claim remaining true
if we take $t^{\prime}_{D,D}$ instead of $t_{D,D}$ and vice versa.

Of course if $(\CC, \ot, c)$ is a braided monoidal category, the
quadruples $(c_{M,D}, c_{M,D}^{-1}, c_{D,M}, c_{D,M}^{-1})$ and
$(c_{D,M}^{-1}, c_{D,M}, c_{M,D}^{-1}, c_{M,D})$ are examples of
$(M,D)$-WO for any object $M$ of $\CC$.
\end{remark}

Moreover, going into the interpretation of the notion of weak
operator as a generalization of that of   weak Yang-Baxter operator
we point out the following series of results (See \cite{IND}).

\begin{proposition}\label{etadeltaWO} With the assumptions and notation of Proposition
\ref{Y-Btypecongruence}, we have:

\begin{equation}\label{etadeltar}
(r_M\ot D)\co(M\ot(\delta_D\co\eta_D))= (D\ot
r_M^\prime)\co((\delta_D\co \eta_D)\ot M),
\end{equation}

\begin{equation}\label{rmuepsilon} ((\varepsilon_D\co\mu_D)\ot M)\co(D\ot r_M)=
(M\ot(\varepsilon_D\co\mu_D))\co(r_M^\prime\ot D),
\end{equation}

\begin{equation}\label{etadeltas} (s_M^\prime\ot D)\co(M\ot(\delta_D\co\eta_D))= (D\ot
s_M)\co((\delta_D\co \eta_D)\ot M),
\end{equation}

\begin{equation}\label{smuepsilon} ((\varepsilon_D\co\mu_D)\ot M)\co(D\ot s_M^\prime)=
(M\ot(\varepsilon_D\co\mu_D))\co(s_M\ot D).
\end{equation}

\end{proposition}

\begin{proof}

We prove (\ref{etadeltar}), the others being analogous:

\begin{itemize}

\item[ ]$\hspace{0.38cm} (r_M\ot D)\co(M\ot(\delta_D\co\eta_D))$
\item[ ]$=((r_M\co\nabla_{r_M})\ot D)\co (M\ot (\delta_D\co\eta_D))$
\item[ ]$=(r_M\ot D)\co(M\ot\delta_D)\co \nabla_{r_M}\co(M\ot\eta_D)$
\item[ ]$=(r_M\ot D)\co(M\ot\delta_D)\co r_M^\prime\co(\eta_D\ot M)$
\item[ ]$= ((r_M\co r_M^\prime)\ot D)\co(D\ot r_M^\prime)\co((\delta_D\co \eta_D)\ot M)$
\item[ ]$= (D\ot((M\ot\varepsilon_D)\co r_M^\prime)\ot D)\co(\delta_D\ot M\ot
D)\co(D\ot r_M^\prime)\co((\delta_D\co \eta_D)\ot M)$
\item[ ]$= (D\ot r_M^\prime)\co((\delta_D\co \eta_D)\ot M).$

\end{itemize}

In the above equalities we use that $D$ is a coalgebra, part (ii) of Proposition \ref{cancelation}, the conditions (c3-1), (c3-2) and
(c4-4), and the equality (\ref{nablabaixapordelta}).
\end{proof}

\begin{proposition}\label{PibaixapoloWO} Let $D$ be a WBHA and $M$ any object in
$\CC.$ If $(r_M, r_M^\prime, s_M, s_M^\prime)$ is an $(M,D)$-WO,  it
holds that:

\begin{equation}
\label{newpiLnablarM}(M\ot \Pi_D^L)\co\nabla_{r_M}=\nabla_{r_M}\co(M\ot\Pi_D^L),
\end{equation}

\begin{equation}
\label{newpiLnablarMprima}(\Pi_D^L\ot M)\co \nabla_{r_M^\prime}=\nabla_{r_M^\prime}\co(\Pi_D^L\ot M),
\end{equation}

\begin{equation}
\label{newpiLnablasM}(\Pi_D^L\ot M)\co \nabla_{s_M}=\nabla_{s_M}\co(\Pi_D^L\ot M),
\end{equation}

\begin{equation}
\label{newpiLnablasMprima}(M\ot \Pi_D^L)\co\nabla_{s_M^\prime}=\nabla_{s_M^\prime}\co(M\ot\Pi_D^L).
\end{equation}

The analogous equalities hold writing either $\Pi_D^R,$
$\overline{\Pi}_D^L,$ or $\overline{\Pi}_D^R$ instead of $\Pi_D^L$.

\end{proposition}

\begin{proof} We prove (\ref{newpiLnablarM}) and (\ref{newpiLnablarMprima}), being the others analogous. Applying the
definition of $\Pi_D^L$ and the equalities (\ref{nablabaixapormu}) and
(\ref{nablabaixapordelta}) we have:

\begin{itemize}

\item[ ]$\hspace{0.38cm}(M\ot \Pi^L_D)\circ \nabla_{r_M}$
\item[ ]$=(M\ot \mu_D)\circ (\nabla_{r_M}\ot \lambda_D)\circ (M\ot \delta_D)$
\item[ ]$=\nabla_{r_M}\circ (M\ot (\mu_D\circ (D\ot \lambda_D)\circ \delta_D))$
\item[ ]$=\nabla_{r_M}\circ (M\ot \Pi^L_D).$

\end{itemize}

Now by the definition of $\Pi_D^L$, the condition (c5-2) and the equalities
(\ref{nablaprimabaixapormu}) and (\ref{nablaprimabaixapordelta}) we
get:

\begin{itemize}

\item[ ]$\hspace{0.38cm}(\Pi^L_D\ot M)\circ \nabla_{r'_M}$
\item[ ]$=(\mu_D\ot M)\circ (D\ot \lambda_D\ot M)\circ  (D\ot \nabla_{r'_M})\circ
(\delta_D\ot M)$
\item[ ]$=(\mu_D\ot M)\circ (D\ot \nabla_{r'_M})\circ (D\ot \lambda_D\ot M)\circ
(\delta_D\ot M)$
\item[ ]$=\nabla_{r'_M}\circ ((\mu_D\circ (D\ot \lambda_D)\circ \delta_D)\ot M)$
\item[ ]$=\nabla_{r'_M}\circ (\Pi^L_D\ot M).$

\end{itemize}

Analogously we prove:
\[\nabla_{r_M}\circ (M\ot \Pi^R_D)=(M\ot \Pi^R_D)\circ \nabla_{r_M}\text{ and
}\nabla_{r'_M}\circ (\Pi^R_D\ot M)=(\Pi^R_D\ot M)\circ
\nabla_{r'_M}.\] It is now easy to prove the corresponding equalities
for $\overline{\Pi}_D^L$ and $\overline{\Pi}_D^R$ just using
(\ref{PiePibarra1}) and (\ref{PiePibarra2}).

\end{proof}

\begin{proposition}\label{Pibaixapolors} Let $D$ be a WBHA and $M$ any object in
$\CC.$ If $(r_M, r_M^\prime, s_M, s_M^\prime)$ is an $(M,D)$-WO,  it
holds that:

\begin{equation}
\label{newpiLrM}
(\Pi_D^L\ot M)\co r_M=r_M\co(M\ot \Pi_D^L),
\end{equation}

\begin{equation}
\label{newpiLrMprima}
(M\ot\Pi_D^L)\co r_M^\prime=r_M^\prime\co(\Pi_D^L\ot M) ,
\end{equation}

\begin{equation}
\label{newpiLsM}
(M\ot\Pi_D^L)\co s_M=s_M\co(\Pi_D^L\ot M) ,
\end{equation}

\begin{equation}
\label{newpiLsMprima}
(\Pi_D^L\ot M)\co s_M^\prime=s_M^\prime\co(M\ot \Pi_D^L).
\end{equation}

The analogous equalities hold writing either $\Pi_D^R,$
$\overline{\Pi}_D^L,$ or $\overline{\Pi}_D^R$ instead of $\Pi_D^L$.

\end{proposition}

\begin{proof}

We will show (\ref{newpiLrM}). Firstly note that

\begin{itemize}

\item[ ]$\hspace{0.38cm}  (M\ot \Pi_D^L)\co\nabla_{r_M}$
\item[ ] $=((\varepsilon_D\ot M)\co r_M\ot\Pi_D^L)\co(M\ot\delta_D)$
\item[ ] $=
(((\varepsilon_D\ot M)\co r_M)\ot D)\co(M\ot\mu_D\ot D)\co(M\ot D\ot
t_{D,D})\co(M\ot(\delta_D\co\eta_D)\ot D)$
\item[ ] $=((\varepsilon_D\co\mu_D)\ot M\ot D)\co(D\ot r_M\ot D)\co(r_M\ot
t_{D,D})\co(M\ot(\delta_D\co\eta_D)\ot D)$
\item[ ] $=((\varepsilon_D\co\mu_D)\ot M\ot D)\co(D\ot r_M\ot D)\co(D\ot M\ot
t_{D,D})\co (D\ot r_M^\prime\ot D)$
\item[]$\hspace{0.38cm}\co((\delta_D\co\eta_D)\ot M\ot D)$
\item[ ] $=
((\varepsilon_D\co\mu_D)\ot r_M^\prime)\co(D\ot t_{D,D}\ot
M)\co((\delta_D\co\eta_D)\ot r_M)$
\item[ ] $=r_M^\prime\co(\Pi_D^L\ot M)\co r_M.$

\end{itemize}

In the above calculations, we applied (c3-1), the equality

\begin{equation}\label{deltaPiL}
(D\ot \Pi_{D}^{L})\co \delta_{D}=(\mu_{D}\ot D)\co (D\ot t_{D,D})\co
((\delta_{D}\circ \eta_{D})\ot D),
\end{equation}

the condition (c4-1) and the equalities (\ref{etadeltar}) and (\ref{rMtrprimaM}).

Hence by (\ref{newpiLnablarM}) it holds that:

\begin{equation}
(M\ot \Pi_D^L)\co\nabla_{r_M}=r_M^\prime\co(\Pi_D^L\ot M)\co
r_M=\nabla_{r_M}\co(M\ot \Pi_D^L).
\end{equation}

Now, applying the definition of $\nabla_{r_M^\prime}$, the equality  (\ref{newpiLnablarMprima})  and  part (ii) of
of Proposition \ref{cancelation} we get:

\begin{itemize}

\item[ ]$\hspace{0.38cm} (\Pi_D^L\ot M)\co r_M$
\item[ ] $=(\Pi_D^L\ot M)\co \nabla_{r_M^\prime}\co r_M$
\item[ ]$=\nabla_{r_M^\prime}\co(\Pi_D^L\ot M)\co r_M $
\item[ ]$= r_M\co r_M^\prime\co(\Pi_D^L\ot M)\co r_M$
\item[ ]$=r_M\co\nabla_{r_M}\co(M\ot\Pi_D^L)$
\item[ ] $=r_M\co(M\ot\Pi_D^L).$

\end{itemize}

\end{proof}

\begin{proposition}\label{lambabaixapoloWO} In the hypothesis of Proposition
\ref{Pibaixapolors}, it holds that:

\begin{equation}
\label{newlambdarM}
(\lambda_D\ot M)\co r_M=r_M\co(M\ot\lambda_D),
\end{equation}

\begin{equation}
\label{newlambdarMprima}
 (M\ot\lambda_D)\co r_M^\prime=r_M^\prime\co(\lambda_D\ot M),
\end{equation}

\begin{equation}
\label{newlambdasM}
 (M\ot\lambda_D)\co s_M=s_M\co(\lambda_D\ot M),
 \end{equation}

\begin{equation}
\label{newlambdasMprima}
(\lambda_D\ot M)\co s_M^\prime=s_M^\prime\co(M\ot\lambda_D).
\end{equation}

If $\lambda_D$ is an isomorphism all the corresponding  equalities
obtained writing $\lambda_D^{-1}$ instead of $\lambda_D$ are also
verified.

\end{proposition}

\begin{proof}
To deduce (\ref{newlambdarM}) we can write:

\begin{itemize}

\item[ ]$\hspace{0.38cm}(\lambda_D\ot M)\circ r_M$
\item[ ]$=(\lambda_D\wedge \Pi^L_D\ot M)\circ r_M$
\item[ ]$=((\mu_D\circ (\lambda_D\ot \Pi^L_D))\ot M)\circ (D\ot r_M)\circ (r_M\ot
D)\circ (M\ot \delta_D)$
\item[ ]$=(\mu_D\ot M)\circ (\lambda_D\ot r_M)\circ (r_M\ot \Pi^L_D)\circ (M\ot
\delta_D)$
\item[ ]$=(\mu_D\ot M)\circ (\lambda_D\ot r_M)\circ (r_M\ot \mu_D)\circ (M\ot
\delta_D\ot \lambda_D)\circ (M\ot \delta_D)$
\item[ ]$=(\mu_D\ot M)\circ (\mu_D\ot r_M)\circ (\lambda_D\ot r_M\ot D)\circ (r_M\ot
D\ot \lambda_D)\circ (M\ot \delta_D\ot D)\circ
    (M\ot \delta_D)$
\item[ ]$=(\mu_D\ot M)\circ (\Pi^R_D\ot r_M)\circ (r_M\ot \lambda_D)\circ (M\ot
\delta_D)$
\item[ ]$=r_M\circ (M\ot \Pi^R_D\wedge \lambda_D)$
\item[ ]$=r_M\circ (M\ot \lambda_D),$

\end{itemize}

In the preceding calculations, the first, fourth and eighth
equalities rely on the definition of WBB,
the second, fifth and
sixth on (c4), and the third and seventh ones follow by Proposition
\ref{Pibaixapolors}.

In a similar way we obtain the equality for $r_M^\prime$, $s_M$ and $s_M^\prime$. Finally, by composing with $\lambda_D^{-1}$
 we get the similar equalities involving the inverse of the antipode.

 \end{proof}

\begin{corollary}\label{nablacontes} Let $D$ be a WBHA with invertible antipode and $M$ any object in $\CC$.
 If $(r_M, r_M^\prime, s_M,
s_M^\prime)$ is an $(M,D)$-WO, the following equalities hold:

\begin{itemize}

\item[(i)]

\begin{equation}
\label{newnablarMtDmu}
\nabla_{r_M}=(M\ot(\mu_D\co t_{D,D}))\co (r_M^\prime\ot D)\co(\eta_D\ot M\ot D),
\end{equation}

\begin{equation}
\label{newnablarMtDprimamu}
\nabla_{r_M}=(M\ot(\mu_D\co t^{\prime}_{D,D}))\co (r_M^\prime\ot
D)\co(\eta_D\ot M\ot D),
\end{equation}

\begin{equation}
\label{newnablarMtDdelta}
\nabla_{r_M}=(\varepsilon_D\ot M\ot D)\co (r_M\ot
D)\co(M\ot(t_{D,D}\co\delta_D)),
\end{equation}

\begin{equation}
\label{newnablarMtDprimadelta}
\nabla_{r_M}=(\varepsilon_D\ot M\ot D)\co (r_M\ot
D)\co(M\ot(t^{\prime}_{D,D}\co\delta_D)).
\end{equation}

\item[(ii)]

\begin{equation}
\label{newnablarMprimatDmu}
\nabla_{r_M^\prime}=((\mu_D\co t_{D,D})\ot M)\co (D\ot r_M)\co(D\ot M\ot
\eta_D),
\end{equation}

\begin{equation}
\label{newnablarMprimatDprimamu}
\nabla_{r_M^\prime}=((\mu_D\co t^{\prime}_{D,D})\ot M)\co (D\ot r_M)\co(D\ot
M\ot \eta_D),
\end{equation}

\begin{equation}
\label{newnablarMprimatDdelta}
\nabla_{r_M^\prime}=(D\ot M\ot \varepsilon_D)\co (D\ot r_M^\prime)\co((t_{D,D}\co\delta_D)\ot M),
\end{equation}

\begin{equation}
\label{newnablarMprimatDprimadelta}
\nabla_{r_M^\prime}=(D\ot M\ot \varepsilon_D)\co (D\ot r_M^\prime)\co(( t^{\prime}_{D,D}\co\delta_D)\ot M).
\end{equation}

\item[(iii)]

\begin{equation}
\label{newnablasMtDmu}
\nabla_{s_M}=((\mu_D\co t_{D,D})\ot M)\co (D\ot s_M^\prime)\co(D\ot M\ot
\eta_D),
\end{equation}

\begin{equation}
\label{newnablasMtDprimamu}
\nabla_{s_M}=((\mu_D\co t^{\prime}_{D,D})\ot M)\co (D\ot s_M^\prime)\co(D\ot
M\ot \eta_D),
\end{equation}

\begin{equation}
\label{newnablasMtDdelta}
\nabla_{s_M}=(D\ot M\ot \varepsilon_D)\co (D\ot s_M)\co((t_{D,D}\co\delta_D)\ot M),
\end{equation}

\begin{equation}
\label{newnablasMtDprimadelta}
\nabla_{s_M}=(D\ot M\ot \varepsilon_D)\co (D\ot s_M)\co(( t^{\prime}_{D,D}\co\delta_D)\ot M).
\end{equation}

\item[(iv)]

\begin{equation}
\label{newnablasMprimatDmu}
\nabla_{s_M^\prime}=(M\ot(\mu_D\co t_{D,D}))\co (s_M\ot D)\co(\eta_D\ot M\ot D),
\end{equation}

\begin{equation}
\label{newnablasMprimatDprimamu}
\nabla_{s_M^\prime}=(M\ot(\mu_D\co t^{\prime}_{D,D}))\co (s_M\ot
D)\co(\eta_D\ot M\ot D),
\end{equation}

\begin{equation}
\label{newnablasMprimatDdelta}
\nabla_{s_M^\prime}=(\varepsilon_D\ot M\ot D)\co (s_M^\prime\ot
D)\co(M\ot(t_{D,D}\co\delta_D)),
\end{equation}

\begin{equation}
\label{newnablasMprimatDprimadelta}
\nabla_{s_M^\prime}=(\varepsilon_D\ot M\ot D)\co (s_M^\prime\ot
D)\co(M\ot(t^{\prime}_{D,D}\co\delta_D)).
\end{equation}

\end{itemize}

\end{corollary}

\begin{proof}
Using Proposition \ref{lambabaixapoloWO}, (c3) and the properties of the antipode and its inverse we get (\ref{newnablarMtDmu}):

\begin{itemize}

\item[ ]$\hspace{0.38cm}\nabla_{r_M}=$
\item[ ]$=(M\ot \lambda_D)\circ \nabla_{r_M}\circ (M\ot \lambda_D^{-1})$
\item[ ]$=(M\ot (\lambda_D\circ \mu_D))\circ (r'_M\ot \lambda_D^{-1})\circ
(\eta_D\ot M\ot D)$
\item[ ]$=(M\ot (\mu_D\circ t_{D,D}))\circ (M\ot \lambda_D\ot D)\circ (r'_M\ot
D)\circ (\eta_D\ot M\ot D)$
\item[ ]$=(M\ot (\mu_D\co t_{D,D}))\circ (r'_M\ot D)\circ (\eta_D\ot M\ot D).$

\end{itemize}

In a similar way we obtain (\ref{newnablarMtDprimamu}):

\begin{itemize}

\item[ ]$\hspace{0.38cm}\nabla_{r_M}$
\item[ ]$=(M\ot \lambda_D^{-1})\circ \nabla_{r_M}\circ (M\ot \lambda_D)$
\item[ ]$=(\varepsilon_D\ot M\ot D)\circ (r_M\ot \lambda_D^{-1})\circ (M\ot
(\delta_D\circ \lambda_D))$
\item[ ]$=(\varepsilon_D\ot M\ot D)\circ (r_M\ot D)\circ (M\ot \lambda_D\ot D)\circ
(M\ot (t_{D,D}\circ \delta_D))$
\item[ ]$=(\varepsilon_D\ot M\ot D)\circ (r_M\ot D)\circ (M\ot (t_{D,D}\co\delta_D)).$

\end{itemize}

The remaining equalities can be proved following
the same pattern, composing with $\lambda_D$ and $\lambda_D^{-1}$ in
the suitable order at convenience.
 \end{proof}

\section{The  category of Yetter-Drinfeld modules}
In this section  the category of left-left Yetter-Drinfeld modules
over an arbitrary WBHA $D$ is defined. We deal with WBHA's in a
monoidal  category $\CC$ that is not assumed to be equipped
with a braiding. In this situation, the first task consists on
giving  a suitable definition of Yetter-Drinfeld module such that we
recovered the classic one in the particular case of modules over a
Hopf algebra in a symmetric category as it appears in \cite{RAD},
and also the generalization of the preceding one to the weak Hopf
algebra case introduced in \cite{BNS}.

In the definition of $(M,D)$-WO, we have only
considered a WBHA $D$, while $M$ was simply an arbitrary object of
the  monoidal category. It will be now discussed how the
notion of $(M,D)$-WO can be enriched when the object $M$ is also
equipped with an algebraic structure.

\begin{lemma}\label{extracond}
 Let $D$ be a WBHA, $M$ in $\CC$ and $(r_M, r_M^\prime, s_M, s_M^\prime )$ an $(M,D)$-WO. It holds that:

\begin{itemize}

\item[(i)] If $(M,\varphi_M)$ is a left $D$-module then

\begin{itemize}

\item[(i-1)]

$\varphi_M=\varphi_M\co\nabla_{s_M}$ iff $\varphi_M\co s_M^\prime\co(M\ot \eta_D)=id_M,$

\item[(i-2)]

$\varphi_M=\varphi_M\co\nabla_{r_M^\prime}$ iff $\varphi_M\co r_M\co(M\ot \eta_D)=id_M.$

\end{itemize}

\item[(ii)]  If $(M,\varrho_M)$ is a left $D$-comodule then

\begin{itemize}

\item[(ii-1)]

$\varrho_M=\nabla_{s_M}\co\varrho_M$ iff $ (M\ot\varepsilon_D)\co s_M\co\varrho_M=id_M,$

\item[(ii-2)] $\varrho_M=\nabla_{r_M^\prime}\co\varrho_M$ iff $ (M\ot\varepsilon_D)\co r_M^\prime\co\varrho_M=id_M.$

\end{itemize}

\end{itemize}

\end{lemma}

\begin{proof} For (i-1), to prove the direct implication, using the hypothesis,  (c3)  and the module condition, we have

\begin{itemize}

\item[ ]$\hspace{0.38cm}\varphi_M$
\item[ ]$= \varphi_M\co(\mu_D\ot M)\co(D\ot(s_M^\prime\co (M\ot\eta_D)))$
\item[ ]$= \varphi_M\circ (D\ot (\varphi_M\circ s_M^\prime\co(M\ot\eta_D))),$

\end{itemize}

so composing with $\eta_D\ot M$ the desired equality follows.

On the other hand, if $\varphi_M\co s_M^\prime\co(M\ot \eta_D)=id_M$
then

\begin{itemize}

\item[ ]$\hspace{0.38cm}\varphi_M\co\nabla_{s_M}$
\item[ ]$= \varphi_M\co(\mu_D\ot M)\co(D\ot(s_M^\prime\co (M\ot\eta_D)))$
\item[ ]$=\varphi_M\circ (D\ot (\varphi_M\circ s_M^\prime\co(M\ot\eta_D)))$
\item[ ]$=\varphi_M,$

\end{itemize}
and we obtain the opposite implication. The other statements follow
similarly using (c3) at convenience.
 \end{proof}

Now we introduce  the notion of weak operator compatible with a
(co)module structure of $M$.

\begin{definition}\label{WOmodulecompatible}
Let $D$ be a WBHA,  $M$ an object of $\CC$ and  $(r_M, r_M^\prime,
s_M, s_M^\prime)$ an $(M,D)$-WO.

\begin{itemize}

\item[(i)]

If $(M,\varphi_M)$ is a left $D$-module, the $(M,D)$-WO is said to be
compatible with the $D$-module structure provided that it satisfies:

\begin{itemize}

\item[(i-1)] $r_M\co(\varphi_M\ot D)=(D\ot\varphi_M)\co(t_{D,D}\ot M)\co( D\ot r_M),$

\item[(i-2)] $r_M^\prime\co(D\ot\varphi_M)=(\varphi_M\ot D)\co(D\ot
r_M^\prime)\co(t^{\prime}_{D,D}\ot M),$

\item[(i-3)] $s_M^\prime\co(\varphi_M\ot D)=(D\ot\varphi_M)\co(t^{\prime}_{D,D}\ot M)\co( D\ot s_M^\prime),$

\item[(i-4)] $s_M\co(D\ot\varphi_M)=(\varphi_M\ot D)\co(D\ot
s_M)\co(t_{D,D}\ot M).$

\end{itemize}

\item[(ii)]
If $(M,\varrho_M)$ is a left $D$-comodule, the $(M,D)$-WO is said to be
compatible with the $D$-comodule structure provided that it
satisfies:

\begin{itemize}

\item[(ii-1)] $(D\ot\varrho_M)\co r_M=(t_{D,D}\ot M)\co(D\ot r_M)\co(\varrho_M\ot D),$

\item[(ii-2)] $(\varrho_M\ot D)\co r_M^\prime=(D\ot r_M^\prime)\co(t^{\prime}_{D,D}\ot
M)\co(D\ot\varrho_M),$

\item[(ii-3)] $(D\ot\varrho_M)\co s_M^\prime=(t^{\prime}_{D,D}\ot M)\co(D\ot s_M^\prime)\co(\varrho_M\ot D),$

\item[(ii-4)]$(\varrho_M\ot D)\co s_M=(D\ot s_M)\co(t_{D,D}\ot
M)\co(D\ot\varrho_M).$

\end{itemize}

\end{itemize}

\end{definition}

Notice that in the particular case of $\CC$ being a braided category with braiding $c$
the conditions trivialize because of $t_{D,D}=c_{D,D}$, $t^{\prime}_{D,D}=c^{-1}_{D,D}$, $r_M=c_{M,D}$,
$r_M^{\prime}=c^{-1}_{M,D}$, $s_M=c_{D,M}$ and $s_M^{\prime}=c^{-1}_{D,M}$. Then in that context the compatibility is not a
restriction.

\begin{definition}\label{defl-lYD}
 Let $D$ be a WBHA. We say that
$(M,\varphi_M,\varrho_M)$ is a left-left Yetter-Drinfeld module over
$D$  if $(M,\varphi_M)$ is a left $D$-module,  $(M,\varrho_M)$ is a
left $D$-comodule and:

\begin{itemize}

\item[(yd1)]
$\varrho_M=(\mu_D\ot\varphi_M)\co(D\ot t_{D,D}\ot
M)\co(\delta_D\ot\varrho_M)\co(\eta_D\ot M).$

\item[(yd2)]
There exists $(r_M,r_M^\prime, s_M, s_M^\prime)$ an $(M,D)$-WO
compatible with the (co)module structure of $M,$  such that

\begin{itemize}

\item[]
$(\mu_D\ot\varphi_M)\co(D\ot t_{D,D}\ot M)\co(\delta_D\ot\varrho_M)$

\item[]
$=(\mu_D\ot M)\co(D\ot r_M)\co((\varrho_M\co\varphi_M)\ot D)\co(D\ot
s_M)\co(\delta_D\ot M).$

\end{itemize}

\end{itemize}

The class of all left-left Yetter-Drinfeld modules over $D$ will be
denoted by $_D^D\mathcal{YD}$.
\end{definition}

\begin{remark}
Note that when the ambient category $\CC$ is symmetric and we take
both the weak Yang-Baxter operator and the $(M,D)$-WO to be the
braiding of $\CC,$ we recover the classic definitions of Yetter-Drinfeld
 module introduced in \cite{RAD} in the context of Hopf
algebras and generalizated in \cite{B} (see also \cite{CWY} and \cite{NEN}) to
the context of weak Hopf algebras.

Moreover, assuming that $\CC$ is braided with braiding $c$ and
$t_{D,D}=c_{D,D}$, $t^{\prime}_{D,D}=c_{D,D}^{-1}$, if $(M,
\varphi_M)$ is a left $D$-module and $(M, \varrho_M)$ a left
$D$-comodule, $(c_{M,D}, c_{M,D}^{-1}, c_{D,M}, c_{D,M}^{-1})$  is
an $(M,D)$-WO compatible with the (co)module structure of $M.$
Therefore, we can define in this setting a left-left Yetter-Drinfeld
module over $D$ as a left $D$-module $(M,\varphi_M)$ and a left
$D$-comodule
 $(M, \varrho_M)$ such that the following equalities hold:
 \begin{itemize}

\item[(i)]

$\varrho_M=(\mu_D\ot\varphi_M)\co(D\ot c_{D,D}\ot
M)\co(\delta_D\ot\varrho_M)\co(\eta_D\ot M).$

\item[(ii)]

$\hspace{0.30cm}(\mu_D\ot\varphi_M)\co(D\ot c_{D,D}\ot M)\co(\delta_D\ot\varrho_M)$
\item[ ]$=(\mu_D\ot M)\co(D\ot c_{M,D})\co((\varrho_M\co\varphi_M)\ot D)\co(D\ot
c_{D,M})\co(\delta_D\ot M).$

\end{itemize}

\end{remark}

\begin{definition}\label{Y-Dmorphism}
Let $(M,\varphi_M,\varrho_M)$ and $(N,\varphi_N,\varrho_N)$ be in
the class $_D^D\mathcal{YD}$ with associated weak operators
$(r_M,r_M^\prime, s_M,s_M^\prime)$ and $(r_N,r_N^\prime,
s_N,s_N^\prime)$ respectively. It is said that a morphism
$f:M\rightarrow N$ in $\CC$ is a morphism of left-left Yetter-Drinfeld
modules if:

\begin{itemize}

\item[(i)]

$f$ is a left (co)module morphism.

\item[(ii)] $r_N\co(f\ot D)=(D\ot f)\co r_M , \qquad$ $ s_N\co(D\ot f)=(f\ot D)\co s_M.$

\end{itemize}

\end{definition}
\begin{remark}
In the last definition, the verification of the condition (ii) for $r_M$
is equivalent to its verification for $r_M^\prime,$ and the same
happens with $s_M$ and $s_M^\prime.$ Actually, if we  assume (ii)
for $r_M$ using the characterization of $\nabla_{r_M}$ of (c3-1) we
conclude that:
\begin{equation}\nabla_{r_N}\co(f\ot D)=(f\ot D)\co\nabla_{r_M},\end{equation}
and by (c3-2) we deduce:
\begin{equation}\nabla_{r_N^\prime}\co(D\ot f)=(D\ot f)\co\nabla_{r_M^\prime} .
\end{equation}

Combining  the preceding equalities with (c3) and part (ii) of Proposition
\ref{cancelation} we conclude that $(f\ot D)\co
r_M^\prime=r_N^\prime\co(D\ot f)$. Indeed,
\begin{itemize}
\item[ ]$\hspace{0.38cm} (f\ot D)\co r_M^\prime$
\item[ ]$= (f\ot D)\co\nabla_{r_M}\co r_M^\prime$
\item[ ]$= \nabla_{r_N}\co(f\ot D)\co r_M^\prime$
\item[ ]$= r_N^\prime\co r_N\co(f\ot D)\co r_M^\prime$
\item[ ]$= r_N^\prime\co(D\ot f)\co\nabla_{r_M^\prime}$
\item[ ]$= r_N\co \nabla_{r_N^\prime}\co(D\ot f)$
\item[]$=r_N^\prime\co(D\ot f).$
\end{itemize}
The proof for the equality $s_N^\prime\co(f\ot D)=(D\ot f)\co s_M^\prime $ follows by the same argument.
\end{remark}

As the identity morphism $id_M$ satisfies the above conditions for
any object $M$ it can  be introduced the following:

\begin{definition}
Let $D$ be a WBHA. The category of left-left Yetter-Drinfeld modules
is that whose objects are the class $_D^D\mathcal{YD}$ and whose
morphisms
 between objects are those in the conditions of  Definition \ref{Y-Dmorphism}. It
will  be denoted also by $^D_D\mathcal{YD}$.
\end{definition}

Generalizing  the braided symmetric case [\cite{CWY}, Proposition 2.2], the conditions (yd1) and (yd2) can also be restated  in the
following way:

\begin{proposition}\label{yd3}
Let $D$ be a WBHA and let $(M,\varphi_M)$ be a left $D$-module and
$(M,\varrho_M)$ a left $D$-comodule. Assume that there exists
$(r_M,r_M^\prime, s_M, s_M^\prime)$ an $(M,D)$-WO compatible with
the (co)module structures of $M$. Then the conditions (yd1) and (yd2)
are equivalent to
\begin{itemize}
\item[(yd3)] $\varrho_M\co\varphi_M$
\item [  ]$= (\mu_D\ot M)\co(D\ot r_M)\co(((\mu_D\ot\varphi_M)\co(D\ot t_{D,D}\ot
M)\co(\delta_D\ot\varrho_M))\ot\lambda_D)\co(D\ot s_M)$
\item[]$\hspace{0.38cm}\co(\delta_D\ot M).$
\end{itemize}
\end{proposition}

\begin{proof}Indeed, if we assume (yd1) and (yd2) then:
\begin{itemize}
\item [  ]$\hspace{0.38cm} (\mu_D\ot M)\co(D\ot r_M)
\co(((\mu_D\ot\varphi_M)\co(D\ot t_{D,D}\ot
M)\co(\delta_D\ot\varrho_M))\ot\lambda_D)\co(D\ot s_M)$
\item[]$\hspace{0.38cm}\co(\delta_D\ot M)$
\item [  ]$=(\mu_D\ot M)\co(\mu_D\ot r_M)\co(D\ot r_M\ot
D)\co((\varrho_M\co\varphi_M)\ot D\ot D)\co(D\ot s_M\ot\lambda_D)$
\item[]$\hspace{0.38cm}\co(\delta_D\ot s_M)\co(\delta_D\ot M)$
\item [  ]$=(\mu_D\ot M)\co(D\ot r_M)\co((\varrho_M\co\varphi_M)\ot\Pi_D^L)\co(D\ot
s_M)\co(\delta_D\ot M)$
\item [  ]$=(\mu_D\ot M)\co(D\ot r_M)\co((\varrho_M\co\varphi_M)\ot D)\co(\mu_D\ot
s_M)\co(D\ot t_{D,D}\ot M)$
\item[]$\hspace{0.38cm}\co((\delta_D\co\eta_D)\ot D\ot M)$
\item [  ]$=(\mu_D\ot M)\co(D\ot r_M)\co((\varrho_M\co\varphi_M)\ot D)\co(D\ot
s_M)\co((\delta_D\co\eta_D)\ot \varphi_M)$
\item [  ]$=(\mu_D\ot\varphi_M)\co(D\ot t_{D,D}\ot
M)\co(\delta_D\ot\varrho_M)\co(\eta_D\ot\varphi_M)$
\item[]$=\varrho_M\co\varphi_M.$
\end{itemize}
In the preceding calculations, the first and fifth equalities follow
by (yd2), the second by (c4) and the third one by
(\ref{newpiLsM}) and (\ref{deltaPiL}). On the fourth equality we
apply compatibility with the $D$-module structure and on the last
one (yd1).

On the other hand, assuming (yd3) we can deduce (yd1) as follows:
\begin{itemize}
\item [  ]$\hspace{0.38cm} \varrho_M$
\item[]$=  \varrho_M \co\varphi_M\co(\eta_D\ot M)  $
\item[]$= (\mu_D\ot M)\co(D\ot r_M) \co(((\mu_D\ot\varphi_M)\co(D\ot t_{D,D}\ot M)\co(\delta_D\ot\varrho_M))\ot\lambda_D)\co(D\ot s_M)$
\item[]$\hspace{0.38cm}\co((\delta_D\co\eta_D)\ot M)$
\item[]$=(\mu_D\ot M)\co(D\ot r_M)  \co(((\mu_D\ot\varphi_M)\co(D\ot t_{D,D}\ot M)\co(D\ot D\ot\varrho_M))\ot D)$
\item[]$\hspace{0.38cm}\co(\delta_D\ot s_M)\co(D\ot\Pi_D^R\ot M)\co((\delta_D\co\eta_D)\ot M)$
\item[]$=((\mu_D\co(D\ot(\overline{\Pi}_D^L\co\lambda_D)))\ot M)\co(D\ot r_M) \co(((\mu_D\ot\varphi_M)\co(D\ot t_{D,D}\ot M)$
\item[]$\hspace{0.38cm}\co(\delta_D\ot\varrho_M))\ot D)\co(D\ot s_M)\co((\delta_D\co\eta_D)\ot M)$
\item[]$=(D\ot(\varepsilon_D\co\mu_D)\ot M)\co(\delta_D\ot\lambda_D\ot M)\co(D\ot r_M)$
\item[]$\hspace{0.38cm}\co(((\mu_D\ot\varphi_M)\co(D\ot t_{D,D}\ot M)\co(\delta_D\ot\varrho_M))\ot D)\co(D\ot s_M)\co((\delta_D\co\eta_D)\ot M)$
\item[]$=(D\ot(\varepsilon_D\co\mu_D)\ot M)\co(\mu_D\ot\mu_D\ot D\ot M)\co(D\ot t_{D,D}\ot D\ot r_M)$
\item[]$\hspace{0.38cm}\co(\delta_D\ot\delta_D\ot\varphi_M\ot D)\co(D\ot t_{D,D}\ot M\ot D)
\co(D\ot D\ot \varrho_M\ot \lambda_D)\co(D\ot D\ot s_M)$
\item[]$\hspace{0.38cm}\co(D\ot\delta_D\ot M)\co((\delta_D\co\eta_D)\ot M)$
\item[]$=(D\ot(\varepsilon_D\co\mu_D)\ot M)\co(D\ot D\ot r_M)\co(\mu_D\ot\mu_D\ot \varphi_M\ot D)$
\item[]$\hspace{0.38cm}
\co(D\ot t_{D,D}\ot  t_{D,D}\ot M\ot D)\co(\delta_D\ot t_{D,D}\ot
\varrho_M\ot \lambda_D)\co(D\ot D\ot D\ot s_M)$
\item[]$\hspace{0.38cm}
\co(D\ot D\ot t_{D,D}\ot M)\co(D\ot \delta_D\ot
\varrho_M)\co((\delta_D\co\eta_D)\ot M)$
\item[]$=(D\ot(\varepsilon_D\co\mu_D)\ot M)\co(D\ot D\ot r_M )\co(D\ot\mu_D\ot\varphi_M\ot\lambda_D)\co(D\ot D\ot t_{D,D}\ot M\ot D )$
\item[]$\hspace{0.38cm}\co(D\ot\delta_D\ot\varrho_M\ot D)\co(D\ot D\ot s_M)\co(\mu_D\ot\delta_D\ot M)\co(D\ot t_{D,D}\ot M)$
\item[]$\hspace{0.38cm}\co((\delta_D\co\eta_D)\ot\varrho_M)$
\item[]$=(D\ot((\varepsilon_D\ot M)\co\varrho_M))\co(\mu_D\ot\varphi_M)\co(D\ot t_{D,D}\ot M)\co((\delta_D\co\eta_D)\ot\varrho_M)$
\item[]$=(\mu_D\ot\varphi_M)\co(D\ot t_{D,D}\ot M)\co((\delta_D\co\eta_D)\ot\varrho_M).$
\end{itemize}

The first equality follows by the condition of $D$-module for $M$. In the second and nineth ones we apply the hypothesis; the third one uses
(\ref{newlambdasM}) and the equality
$$(D\ot\lambda_D)\co\delta_D\co\eta_D=(D\ot\Pi_D^R)\co\delta_D\co\eta_D.$$
The fourth equality relies on   Proposition \ref{Pibaixapolors} and (\ref{Pielambda2}); the fifth is a consequence of the equality
$$\mu_D\co(D\ot\overline{\Pi}_D^L)=((D\ot(\varepsilon_D\co\mu_D))\co(\delta_D\ot D),$$
and the sixth and eighth ones follow because of $D$ is a WBHA. In the seventh equality we apply compatibility of the $D$-module structure for $M$; finally, in the last one we use the condition of $D$-comodule for $M$.

Using the same technics we get:

\begin{itemize}
\item [  ]$\hspace{0.38cm}(\mu_D\ot M)\co(D\ot r_M)\co((\varrho_M\co\varphi_M)\ot D)\co (D\ot s_M)\co(\delta_D\ot M)$
\item[]$=(\mu_D\ot M)\co(\mu_D\ot r_M)\co(D\ot r_M\ot D)\co(((\mu_D\ot\varphi_M)\co(D\ot t_{D,D}\ot M)$
\item[]$\hspace{0.38cm}\co(\delta_D\ot\varrho_M))\ot\lambda_D\ot D)\co(D\ot s_M\ot D)\co(\delta_D\ot s_M)\co(\delta_D\ot M)$
\item[]$=(\mu_D\ot M)\co(D\ot r_M)\co(((\mu_D\ot\varphi_M)\co(D\ot t_{D,D}\ot M)\co(\delta_D\ot\varrho_M))\ot\Pi_D^R)\co (D\ot s_M)$
\item[]$\hspace{0.38cm}\co(\delta_D\ot M)$
\item[]$=(\mu_D\ot M)\co(D\ot r_M)\co(((\mu_D\ot\varphi_M)\co(D\ot t_{D,D}\ot M)\co(\delta_D\ot\varrho_M))\ot D)\co (D\ot s_M)$
\item[]$\hspace{0.38cm}\co(\mu_D\ot \lambda_D\ot M)\co(D\ot(\delta_D\co\eta_D)\ot M)$
\item[]$=(\mu_D\ot M)\co(D\ot r_M)\co(\mu_D\ot\varphi_M\ot D)\co(D\ot t_{D,D}\ot M\ot D)\co(\mu_D\ot\mu_D\ot\varrho_M\ot D)$
\item[]$\hspace{0.38cm}\co(D\ot t_{D,D}\ot D\ot s_M)\co(\delta_D\ot\delta_D\ot\lambda_D\ot M)\co(D\ot(\delta_D\ot\eta_D)\ot M)$
\item[]$=(\mu_D\ot M)\co(D\ot r_M)\co(\mu_D\ot\varphi_M\ot D)\co(\mu_D\ot t_{D,D}\ot\varphi_M\ot D)$
\item[]$\hspace{0.38cm}\co(D\ot t_{D,D}\ot t_{D,D}\ot M\ot D)\co(\delta_D\ot \delta_D\ot\varrho_M\ot\lambda_D)
\co(D\ot D\ot s_M)$
\item[]$\hspace{0.38cm}\co(D\ot(\delta_D\ot\eta_D)\ot M)$
\item[]$=(\mu_D\ot \varphi_M)\co(D\ot t_{D,D}\ot M)\co (\delta_D\ot ((\mu_D\ot M)\co(D\ot r_M)\co(((\mu_D\ot\varphi_M)$
\item[]$\hspace{0.38cm}\co(D\ot t_{D,D}\ot M)\co(\delta_D\ot\varrho_M))\ot\lambda_D)\co(D\ot s_M)\co(\delta_D\ot M)))
\co (D\ot \eta_D\ot M)$
\item[]$=(\mu_D\ot \varphi_M)\co(D\ot t_{D,D}\ot M)\co(\delta_D\ot\varrho_M)\co(D\ot(\varphi_M\co(\eta_D\ot M)))$
\item[]$=(\mu_D\ot \varphi_M)\co(D\ot t_{D,D}\ot M)\co(\delta_D\ot\varrho_M),$
\end{itemize}
so the condition (yd2) can be obtained from (yd3).
\end{proof}

The following properties about Yetter-Drinfeld modules
constitute a generalization of the results obtained in the braided
context. See  \cite{Proj} for the idea of the proof.

\begin{lemma}\label{consecuencias} Let $D$ be a WBHA in $\CC$. If $(M,\varphi_M,\varrho_M)$ is in
$^D_D\mathcal{YD}$ then it obeys the following properties:

\begin{equation}
\label{newpiLphirhomu}
\varrho_M\co\varphi_M\co(\Pi_D^L\ot M)=(\mu_D\ot D)\co(\Pi_D^L\ot\varrho_M),
\end{equation}

\begin{equation}
\label{newpiLphirhodelta}
(\Pi_D^L\ot M)\co\varrho_M\co\varphi_M=(\Pi_D^L\ot\varphi_M)\co(\delta_D\ot M),
\end{equation}

\begin{equation}
\label{newpiRphirhomu}
\varrho_M\co\varphi_M\co(\Pi_D^R\ot M)
=(\mu_D\ot M) \co(D\ot (\lambda_D\co\Pi_D^R)\ot M) \co(D\ot(r_M\co
s_M))\co(t_{D,D}\ot M)\co(D\ot \varrho_M),
\end{equation}

\begin{equation}
\label{newpiRphirhodelta}
(\Pi_D^R\ot M)\co\varrho_M\co\varphi_M
=(D\ot\varphi_M)\co(t_{D,D}\ot M) \co(D\ot(r_M\co s_M))
\co(D\ot(\Pi_D^R\co\lambda_D)\ot M)\co(\delta_D\ot M).
\end{equation}

\end{lemma}

\begin{proposition}\label{compatible}
Let $D$ be a WBHA with invertible antipode and let $(M,\varphi_M,\varrho_M)$ be in
$^D_D\mathcal{YD}$. Then:
\begin{itemize}
\item[(i)] $\varphi_M=\varphi_M\co\nabla_{s_M}$ iff $\varrho_M=\nabla_{s_M}\co\varrho_M$,
\item[(ii)] $\varphi_M=\varphi_M\co\nabla_{r_M^\prime}$ iff $\varrho_M=\nabla_{r_M^\prime}\co\varrho_M.$
\end{itemize}
\end{proposition}

\begin{proof} We will show (i), being (ii) analogous. For the `if' part, in virtue of Proposition \ref{Pibaixapolors},
the equality $(\overline{\Pi}_D^L\ot M)\co \varrho_M=(D\ot \varphi_M)\co((\delta_D\co \eta_D)\ot M)$ which holds by (\ref{newpiLphirhodelta}), compatibility of the module structure, Corollary \ref{nablacontes} and the hypothesis, it results that
\begin{itemize}
\item[]$\hspace{0.38cm} (M\ot \varepsilon_D)\co s_M\co \varrho_M$
\item[]$=(M\ot \varepsilon_D)\co s_M\co (\overline{\Pi}_D^L\ot M)\co \varrho_M$
\item[]$=(M\ot \varepsilon_D)\co s_M\co (D\ot \varphi_M)\co ((\delta_D\co \eta_D)\ot M)$
\item[]$=(\varphi_M\ot \varepsilon_D)\co (D\ot s_M)\co (t_{D,D}\ot M)\co ((\delta_D\co \eta_D)\ot M)$
\item[]$=\varphi_M\co\nabla_{s_M}\co (\eta_D\ot M)$
\item[]$=\varphi_M\co(\eta_D\ot M)$
\item[]$=id_M$.
\end{itemize}

Applying Lemma \ref{extracond} we obtain the equality
$\varrho_M=\nabla_{s_M}\co\varrho_M$.

The opposite implication follows a similar pattern:
\begin{itemize}
\item[]$\hspace{0.38cm}  \varphi_M\co s_M^\prime\co (M\ot\eta_D)$
\item[]$= \varphi_M\co (\Pi_D^L\ot M)\co s_M^\prime\co (M\ot\eta_D)$
\item[]$=((\varepsilon_D\co\mu_D)\ot M)\co(\Pi_D^L\ot\varrho_M)\co s_M^\prime\co (M\ot\eta_D)$
\item[]$=((\varepsilon_D\co\mu_D\co t^\prime_{D,D})\ot M)\co (D\ot s_M^\prime)\co(\varrho_M\ot\eta_D)$
\item[]$=(\varepsilon_D\ot M)\co\nabla_{s_M}\co\varrho_M$
\item[]$=(\varepsilon_D\ot M)\co\varrho_M$
\item[]$=id_M,$
\end{itemize}
and by Lemma \ref{extracond} we have that
$\varphi_M=\varphi_M\co\nabla_{s_M}$.

\end{proof}

\begin{corollary}\label{nablamorreconphi} Let $D$ be a WBHA with invertible antipode and  $(M,\varphi_M,\varrho_M)$  an object in $_D^D\mathcal{YD}.$ It holds that \begin{equation}\label{nablamorreconphiequation}
\varphi_M\co\nabla_{s_M}=\varphi_M\co\nabla_{r_M^\prime}=\varphi_M;\quad
\nabla_{s_M}\co\varrho_M=\nabla_{r_M^\prime}\co\varrho_M=\varrho_M.
\end{equation}
\end{corollary}

\begin{proof}
Using that $M$ is a $D$-comodule, the condition (yd-3) twice, (c4) and the counit property we can
write
\begin{itemize}
\item[]$\hspace{0.38cm} \varphi_M\co\nabla_{s_M}$
\item[]$=(\varepsilon_D\ot M)\co\varrho_M\co\varphi_M\co\nabla_{s_M}$
\item[]$=((\varepsilon_D\co\mu_D)\ot M)\co (D\ot r_M)\co(((\mu_D\ot\varphi_M)\co(D\ot t_{D,D}\ot M)\co(\delta_D\ot\varrho_M))\ot\lambda_D)$
\item[]$\hspace{0.38cm}    \co(D\ot s_M\ot\varepsilon_D)\co(\delta_D\ot s_M)\co(\delta_D\ot M)$
\item[]$=((\varepsilon_D\co\mu_D)\ot M)\co (D\ot r_M)\co (((\mu_D\ot\varphi_M)\co(D\ot t_{D,D}\ot M)\co(\delta_D\ot\varrho_M))\ot\lambda_D)$
\item[]$\hspace{0.38cm}    \co(D\ot M\ot ((D\ot \varepsilon_D)\co \delta_D))\co (D\ot s_M)\co (\delta_D\ot M)$
\item[]$=((\varepsilon_D\co\mu_D)\ot M)\co (D\ot r_M)\co (((\mu_D\ot\varphi_M)\co(D\ot t_{D,D}\ot M)\co(\delta_D\ot\varrho_M))\ot\lambda_D)$
\item[]$\hspace{0.38cm}    \co (D\ot s_M)\co(\delta_D\ot M)  $
\item[]$=(\varepsilon_D\ot M)\co\varrho_M\co\varphi_M$
\item[]$=\varphi_M.$
\end{itemize}

Now, by Proposition \ref{compatible}   we also know that
$\varrho_M=\nabla_{s_M}\co\varrho_M.$ The remaining equalities can
be proved by similar arguments.
\end{proof}

In this part of the work  the announced monoidal structure of
$_D^D\mathcal{YD}$ in the general case is presented. We want also to
point out that when we restrict to the braided case we recover the
monoidal structure exposed in \cite{Proj}, so it could be said that
the new theory introduced in this work is coherent with the classic
one developped in the Hopf algebra setting.

\begin{apart}\label{idemp}
Let  $D$ be a WBHA. If
$(M,\varphi_{M})$ and $(N,\varphi_{N})$ are left $D$-modules and it
exists a quadruple  $(r_M,r_M^\prime, s_M, s_M^\prime)$ forming  an
$(M,D)$-WO compatible with the (co)module structure, then two
different morphisms arise naturally:
\begin{itemize}
\item[]$\nabla_{M\ot N}=(\varphi_{M}\otimes \varphi_{N})\circ
(D\otimes s_M \otimes N)\circ ((\delta_{D}\co\eta_D)\otimes M\otimes
N),$
\item[]$\Delta_{M\ot N}=((\varepsilon_D\co\mu_D)\otimes M\otimes N)\circ
(D\otimes r_M\otimes N)\circ (\varrho_M\otimes\varrho_N).$
\end{itemize}
\end{apart}

\begin{lemma}
Let  $D$ be a WBHA. If
$(M,\varphi_{M})$ and $(N,\varphi_{N})$  are left $D$-modules and
$(r_M,r_M^\prime, s_M, s_M^\prime)$ is an $(M,D)$-WO compatible with
the module structure, then the morphism $\nabla_{M\ot N}$ is
idempotent. It holds the analogous result for $\Delta_{M\ot N}$ in
the comodule case.
\end{lemma}

\begin{proof}We will give the proof for $\nabla_{M\ot N}$. Using the compatibility,
the module character or $M$ and $N$ and  the conditions (c4) and (b4) we
have:
\begin{itemize}
\item[]$\hspace{0.38cm}    \nabla_{M\ot N}\co\nabla_{M\ot N}$
\item[]$=(\varphi_M\ot\varphi_N)\co (D\ot\varphi_M\ot D\ot\varphi_N)\co (D\ot D\ot
s_M\ot D\ot N)\co (D\ot t_{D,D}\ot s_M\ot N)$
\item[]$\hspace{0.38cm} \co((\delta_D \co\eta_D)\ot (\delta_D \co\eta_D)\ot M\ot N)$
\item[]$=(\varphi_M\ot\varphi_N)\co(D\ot s_M\ot N)\co(( (\mu_D\ot\mu_D)\co(D\ot
t_{D,D}\ot D)$
\item[]$\hspace{0.38cm}\co((\delta_D\co\eta_D)\ot(\delta_D\co\eta_D))   )\ot M\ot N)$
\item[]$= \nabla_{M\ot N}.$
\end{itemize}

\end{proof}

The following two lemmas have been introduced as technical tools to
be used in order to show that the morphisms $\nabla_{M\ot N}$ and
$\Delta_{M\ot N}$ coincide.

\begin{lemma} \label{nablasconmutan}Let $D$ be a WBHA with invertible antipode. If
$(M,\varphi_M,\varrho_M)$ is a left $D$-(co)module and
$(r_M,s_M^\prime, s_M,s_M^\prime)$ an $(M,D)$-WO compatible with the
(co)module structure, then
\begin{equation}
\nabla_{s_M^\prime}\co\nabla_{r_M}=\nabla_{r_M}\co\nabla_{s_M^\prime}.
\end{equation}
\end{lemma}

\begin{proof}  Using the properties of WBHA, (\ref{newnablarMtDdelta}) twice, (\ref{trMnablasprimaM}) and
(\ref{nablabaixapordeltaparas}), we have:
\begin{itemize}
\item[]$\hspace{0.38cm}\nabla_{s_M^\prime}\co\nabla_{r_M}$
\item[]$=(\varepsilon_D\ot\nabla_{s_M^\prime})\co(r_M\ot D)\co(M\ot(t_{D,D}\co\delta_D))$
\item[]$=(\varepsilon_D\ot M\ot D)\co(r_M\ot D)\co(M\ot
t_{D,D})\co(\nabla_{s_M^\prime}\ot D)\co(M\ot
\delta_D)$
\item[]$=(\varepsilon_D\ot M\ot D)\co(r_M\ot D)\co(M\ot(t_{D,D}\co\delta_D))\co\nabla_{s_M^\prime}$
\item[]$=\nabla_{r_M}\co\nabla_{s_M^\prime}.$
\end{itemize}
\end{proof}

As a consequence:

\begin{lemma}\label{nablasprimedesaparece} In the hypothesis of the previous lemma, if
it also holds that $(M,\varphi_M,\varrho_M)\in{} _D^D\mathcal{YD}$
then
\begin{equation}\label{nabla_sprimamorre}
(M\ot(\varepsilon_D\co\mu_D))\co((\nabla_{s_M^\prime}\co
r_M^\prime\co\varrho_M)\ot D) =(M\ot(\varepsilon_D\co\mu_D))\co((
r_M^\prime\co\varrho_M)\ot D).
\end{equation}
\end{lemma}

\begin{proof}
Indeed,
\begin{itemize}
\item[]$\hspace{0.38cm}  (M\ot(\varepsilon_D\co\mu_D))\co((\nabla_{s_M^\prime}\co
r_M^\prime\co\varrho_M)\ot D)   $
\item[]$= (M\ot(\varepsilon_D\co\mu_D\co(\overline{\Pi}_D^R\ot
D)))\co((\nabla_{s_M^\prime}\co r_M^\prime\co\varrho_M)\ot D)$
\item[]$= (M\ot(\varepsilon_D\co\mu_D))\co((\nabla_{s_M^\prime}\co
r_M^\prime\co(\overline{\Pi}_D^R\ot M)\co\varrho_M)\ot D)$
\item[]$= (M\ot(\varepsilon_D\co\mu_D))\co((\nabla_{s_M^\prime}\co
r_M^\prime)\ot D)\co(D\ot\varphi_M\ot D)\co(t_{D,D}\ot M\ot D)$
\item[]$\hspace{0.38cm}\co(D\ot(r_M\co s_M)\ot D)\co ((\delta_D\co\eta_D)\ot M\ot D)$
\item[]$=(M\ot(\varepsilon_D\co\mu_D))\co((\nabla_{s_M^\prime}\co r_M)\ot
D)\co(\varphi_M\ot D\ot D)\co(D\ot s_M\ot D)$
\item[]$\hspace{0.38cm}\co((\delta_D\co\eta_D)\ot M\ot D)$
\item[]$=(M\ot(\varepsilon_D\co\mu_D))\co((\nabla_{r_M}\co
\nabla_{s_M^\prime}\co(\varphi_M\ot D)\co(D\ot s_M)$
\item[]$\hspace{0.38cm}\co((t_{D,D}\co
t^{\prime}_{D,D}\co\delta_D\co\eta_D)\ot M))\ot D)$
\item[]$=(M\ot(\varepsilon_D\co\mu_D))\co(\nabla_{r_M}\ot D)\co(s_M\ot
D)\co(D\ot\varphi_M\ot D)\co((
t^{\prime}_{D,D}\co\delta_D\co\eta_D)\ot M\ot D)$
\item[]$=(M\ot(\varepsilon_D\co\mu_D))\co((\nabla_{r_M}\co(\varphi_M\ot D)\co(D\ot
s_M)\co((t_{D,D}\co t^{\prime}_{D,D}\co\delta_D\co\eta_D)\ot M))\ot D)$
\item[]$=(M\ot(\varepsilon_D\co\mu_D))\co((\nabla_{r_M}\co(\varphi_M\ot D)\co(D\ot
s_M)\co((\delta_D\co\eta_D)\ot M))\ot D)$
\item[]$=(M\ot(\varepsilon_D\co\mu_D))\co(r_M^\prime\ot
D)\co(((\overline{\Pi}_D^R\ot M)\co\varrho_M)\ot D)$
\item[]$=(M\ot(\varepsilon_D\co\mu_D))\co((r_M^\prime\co\varrho_M)\ot D).$
\end{itemize}

In the preceding calculations, the first and the last equalities follow because $\varepsilon_D\co\mu_D\co(\overline{\Pi}_D^R\ot D)=\varepsilon_D\co\mu_D$.
The second uses Propositions \ref{PibaixapoloWO} and \ref{Pibaixapolors}. We get the third and the tenth ones by the equation

\begin{equation}
(\overline{\Pi}_D^R\ot M)\co\varrho_M= (D\ot\varphi_M)\co(t_{D,D}\ot
M)\co((D\ot r_M\co s_M))\co((\delta_D\co\eta_D)\ot M),
\end{equation}
which follows by  Lemma \ref{consecuencias}. In the fourth and seventh equalities we apply the compatibility condition for the module structure; the fifth relies on Lemma \ref{nablasconmutan}, and the sixth is a consequence of (\ref{newsM}). Finally, the nineth equality follows by (b2-1).

\end{proof}
Now it is possible to check that the idempotent morphisms defined in
paragraph \ref{idemp} are the same.

\begin{proposition}\label{asnablascoinciden}
Let  $D$ be a WBHA with invertible antipode. If
$(M,\varphi_{M},\varrho_M)$ and $(N,\varphi_N,\varrho_N)$ are
objects of $_D^D\mathcal{YD}$ then
\[
\nabla_{M\ot N}=\Delta_{M\ot N}.
\]
\end{proposition}

\begin{proof}We have:
\begin{itemize}
\item[]$\hspace{0.38cm}\nabla_{M\ot N}$
\item[]$=((\varphi_M\co s_M^\prime)\ot\varphi_N)\co(M\ot(\delta_D\co\eta_D)\ot N)$
\item[]$=((\varphi_M\co s_M^\prime)\ot N)\co(M\ot((\overline{\Pi}_D^L\ot
N)\co\varrho_N))$
\item[]$=((((\varepsilon_D\co\mu_D)\ot M)\co(D\ot(r_M\co s_M))\co(t_{D,D}\ot
M)\co(D\ot\varrho_M))\ot N)\co(s_M^\prime\ot N)$
\item[]$\hspace{0.38cm}\co(M\ot\varrho_N)$
\item[]$=(((M\ot(\varepsilon_D\co\mu_D))\co(r_M^\prime\ot D)\co(D\ot
s_M)\co((t_{D,D}\co t^{\prime}_{D,D})\ot M)\co(D\ot s_M^\prime))\ot
N)$
\item[]$\hspace{0.38cm}\co(\varrho_M\ot\varrho_N)$
\item[]$=(((M\ot(\varepsilon_D\co\mu_D\co t_{D,D}))\co(s_M\ot D)\co(D\ot
r_M^\prime)\co(t^{\prime}_{D,D}\ot M )\co(D\ot s_M^\prime))\ot N)$
\item[]$\hspace{0.38cm}\co(\varrho_M\ot\varrho_N)$

\item[]$=(((M\ot(\varepsilon_D\co\mu_D\co t_{D,D}))\co(\nabla_{s_M^\prime}\ot
D)\co(M\ot t^{\prime}_{D,D})\co(r_M^\prime\ot D))\ot
N)\co(\varrho_M\ot\varrho_N)$
\item[]$=(((M\ot(\varepsilon_D\co\mu_D\co t_{D,D}))\co(M\ot\mu_D\ot D)\co((s_M\ot(\eta_D\ot M))\ot t^{\prime}_{D,D})$
\item[]$\hspace{0.38cm}\co(r_M^\prime\ot D))\ot N)\co(\varrho_M\ot\varrho_N)$
\item[]$=(((M\ot(\varepsilon_D\co\mu_D))\co(M\ot(\mu_D\co t_{D,D})\ot D)\co ((s_M\ot(\eta_D\ot M))\ot \nabla_{D,D})$
\item[]$\hspace{0.38cm}\co (r_M^\prime\ot D))\ot N)\co(\varrho_M\ot\varrho_N)$

\item[]$=(M\ot(\varepsilon_D\co\mu_D)\ot N)\co((\nabla_{s_M^\prime}\co
r_M^\prime\co\varrho_M)\ot\varrho_N)$
\item[]$=(M\ot(\varepsilon_D\co\mu_D)\ot N)\co(r_M^\prime\ot D\ot
N)\co(\varrho_M\ot\varrho_N)$
\item[]$=\Delta_{M\ot N}.$
\end{itemize}

In the preceding calculations, the first and the last equalities follow by (\ref{etadeltas}) and (\ref{rmuepsilon}), respectively; the second and third ones are consequences of Lemma \ref{consecuencias}; the fourth apply compatibility of the comodule structure. In the fifth and seventh equalities we use (c2-1) and (c3-2), respectively; the sixth follows by (\ref{rara2}), and the eighth and nineth follow because $D$ is a WBHA. Finally, the tenth equality relies on (\ref{nabla_sprimamorre}).

\end{proof}

\begin{apart}\label{MxNbendefinido}
 Let $D$ be a WBHA and let $(M,\varphi_{M},\varrho_M)$ and $(N,\varphi_N,\varrho_N)$ be
objects of $_D^D\mathcal{YD}$. We denote by
$M\times N$ the image of the idempotent $\nabla_{M\otimes N}$ and by
$p_{M\otimes N}:M\otimes N\rightarrow M\times N$, $i_{M\otimes
N}:M\times N\rightarrow M\otimes N$ the morphisms such that
$i_{M\otimes N}\circ p_{M\otimes N}=\nabla_{M\otimes N}$ and
$p_{M\otimes N}\circ i_{M\otimes N}=id_{M\times N}$.
 Actually the object $M\times N$ will be taken as the product of $M$ and $N$ in the
category $_D^D\mathcal{YD}$. In order to provide $_D^D\mathcal{YD}$
with a monoidal structure, first to all, by Definition
\ref{defl-lYD}, the object $M\times N$ must be equipped with a
compatible a weak operator. To do so, we state first some
preliminary results and convenient notation.
\end{apart}

\begin{lemma}\label{nablaMxNbaixapoloWO}
 Let $D$ be a WBHA with invertible antipode. If $(M,\varphi_M,\varrho_M)$ and $(N,\varphi_N,\varrho_N)$ are in
$_D^D\mathcal{YD}$ then:

\begin{equation}
\label{newnablaMNrMN}
(D\ot \nabla_{M\ot N})\circ (r_M\ot N)\circ (M\ot r_N)=(r_M\ot N)\circ (M\ot
r_N)\circ (\nabla_{M\ot N}\ot D),
\end{equation}

\begin{equation}
\label{newnablaMNrprimaMN}
(\nabla_{M\ot N}\ot D)\circ (M\ot r_N^\prime)\circ (r_M^\prime\ot N)=(M\ot r_N^\prime)\circ (r_M^\prime \ot N)\circ (D\ot \nabla_{M\ot N}),
\end{equation}

\begin{equation}
\label{newnablaMNsMN}
(\nabla_{M\ot N}\ot D)\circ (M\ot s_N)\circ (s_M\ot N)=(M\ot s_N)\circ (s_M \ot N)\circ (D\ot \nabla_{M\ot N}),
\end{equation}

\begin{equation}
\label{newnablaMNsprimaMN}
(D\ot \nabla_{M\ot N})\circ (s_M^\prime\ot N)\circ (M\ot s_N^\prime)=(s_M^\prime\ot N)\circ (M\ot
s_N^\prime)\circ (\nabla_{M\ot N}\ot D).
\end{equation}

\end{lemma}

\begin{proof}
We will show (\ref{newnablaMNrMN}), the others being analogous. First at all, using (yd-1) twice, the compatibility with the module structure, (b1) and (a2-4),
\begin{itemize}
\item[ ]$\hspace{0.38cm} (\nabla_{D,D}\ot N)\circ (D\ot r_N)\circ (\varrho_N\ot D)$
\item[ ]$=(\nabla_{D,D}\ot N)\circ (D\ot r_N)\circ (((\mu_D\ot \varphi_N)\co
(D\ot t_{D,D}\ot N)\circ ((\delta_D\circ \eta_D)\ot \varrho_N))\ot
D)$
\item[ ]$=(\nabla_{D,D}\ot \varphi_N)\circ (\mu_D\ot t_{D,D}\ot N)\circ (D\ot
t_{D,D}\ot r_N)\circ ((\delta_D\circ \eta_D)\ot \varrho_N\ot D)$
\item[ ]$=(\mu_D\ot D\ot \varphi_N)\circ (D\ot \nabla_{D,D}\ot D\ot N)\circ (D\ot
D\ot t_{D\ot D}\ot N)\circ (D\ot t_{D,D}\ot r_N)$
    \item[]$\hspace{0.38cm} \circ((\delta_D\circ \eta_D)\ot \varrho_N\ot D)$
\item[ ]$=(\mu_D\ot D\ot \varphi_N)\circ (D\ot D\ot t_{D,D}\ot N)\circ (D\ot
(\nabla_{D,D}\circ t_{D,D})\ot r_N)$
\item[]$\hspace{0.38cm}\circ ((\delta_D\circ \eta_D)\ot \varrho_N\ot
D)$
\item[ ]$=(D\ot r_N)\circ (\varrho_N\ot D).$
\end{itemize}

 Now, by the characterization $\nabla_{M\ot N}=\Delta_{M\ot N}$ obtained in Proposition \ref{asnablascoinciden}, the compatibilities with the comodule structures,
the conditions (c1) and (b3) and the equalities (\ref{b3-4}) and (\ref{nabla_Dbaixaporr}) we get:
\begin{itemize}
\item[ ]$\hspace{0.38cm} (D\ot \nabla_{M\ot N})\circ (r_M\ot N)\circ (M\ot r_N)$
\item[ ]$=(D\ot (((\varepsilon_D\circ \mu_D)\ot M\ot N)\circ (D\ot r_M\ot N)\circ
(\varrho_M\ot \varrho_N)))\circ (r_M\ot N)\circ (M\ot r_N)$
\item[ ]$=(((D\ot (\varepsilon_D\circ \mu_D)\ot M)\circ (t_{D,D}\ot r_M)\circ (D\ot
r_M\ot D)\circ (D\ot M\ot t_{D,D}))\ot N)$
\item[]$\hspace{0.38cm}\circ(D\ot M\ot D\ot r_N)\circ (\varrho_M\ot \varrho_N\ot D)$
\item[ ]$=(((D\ot (\varepsilon_D\circ \mu_D)\ot M)\circ (t_{D,D}\ot D\ot M)\circ
(D\ot t_{D,D}\ot M)\circ(D\ot D\ot r_M))\ot N)$
\item[]$\hspace{0.38cm} \circ(D\ot r_M\ot r_N)\circ (\varrho_M\ot \varrho_N\ot D)$
\item[ ]$=(((D\ot \varepsilon_D\ot M)\circ (t_{D,D}\ot M)\circ (\mu_D\ot r_M))\ot
N)\circ (D\ot r_M\ot r_N)$
\item[]$\hspace{0.38cm}\circ (\varrho_M\ot \varrho_N\ot D)$
\item[ ]$=((((\varepsilon_D\circ \mu_D)\ot D\ot M)\circ (D\ot \nabla_{D,D}\ot
M)\circ (D\ot D\ot r_M))\ot N)\circ(D\ot r_M\ot r_N)$
\item[]$\hspace{0.38cm}\circ(\varrho_M\ot \varrho_N\ot D)$
\item[ ]$=((\varepsilon_D\circ \mu_D)\ot r_M\ot N)\circ (D\ot r_M\ot D\ot N)\circ
(\varrho_M\ot \nabla_{D,D}\ot N)\circ (M\ot D\ot r_N)$
\item[]$\hspace{0.38cm}\circ
(M\ot \varrho_N\ot D)$
\item[ ]$=(r_M\ot N)\circ (M\ot r_N)\circ (\nabla_{M\ot N}\ot D).$
\end{itemize}

\end{proof}

\begin{apart}
 Let $D$ be a WBHA with invertible antipode. Given $(M,\varphi_M,\varrho_M)$ and $(N,\varphi_N,\varrho_N)$ in
$_D^D\mathcal{YD}$ we denote by $\varphi_{M\ot N}$ the morphim:
$\varphi_{M\ot N}:D\ot M\ot N\rightarrow M\ot N$ defined by
\[\varphi_{M\ot N}=(\varphi_M\ot\varphi_N)\co(D\ot s_M\ot N)\co(\delta_D\ot M\ot N),\]
 and by $\varrho_{M\ot N}$ the morphism $\varrho_{M\ot N}:M\ot N\rightarrow D\ot
M\ot N$ defined by
 \[\varrho_{M\ot N}=(\mu_D\ot M\ot N)\co(D\ot r_M\ot N)\co(\varrho_M\ot \varrho_N).\]
Note that $\nabla_{M\ot N}=\varphi_{M\ot N}\co(\eta_D\ot M\ot
N)=(\varepsilon_D\ot M\ot N)\co\varrho_{M\ot N}.$

Using this notation and the compatibility with the correspondent weak
operators, it results that:
\begin{equation}\label{nablaephi_MN}
\varphi_{M\otimes N}\circ (D\otimes \nabla_{M\otimes N})
=\varphi_{M\otimes N}=\nabla_{M\otimes N}\circ \varphi_{M\otimes N},
\end{equation}
\begin{equation}\label{nablamorreconphi}
\varrho_{M\otimes N}\co \nabla_{M\otimes N}=\varrho_{M\otimes
N}=(D\ot \nabla_{M\otimes N})\circ \varrho_{M\otimes N}.
\end{equation}
 Moreover, being $(P,\varphi_P,\varrho_P)$ in $_D^D\mathcal{YD}$ and combining the
above equalities with Lemma \ref{nablaMxNbaixapoloWO} we obtain
\begin{equation}\label{ieptrocandelado}
(i_{M\otimes N}\otimes P)\circ \nabla_{(M\times N)\otimes P}\circ
(p_{M\otimes N}\otimes P )=(M\otimes i_{N\otimes P})\circ
\nabla_{M\otimes (N\times P)}\circ (M\otimes  p_{N\otimes P}),
\end{equation}
and
\begin{equation}\label{simetriadenabla}
(M\otimes i_{N\otimes P})\circ \nabla_{M\otimes (N\times P)}\circ
(M\otimes  p_{N\otimes P})=(\nabla_{M\otimes N}\otimes P)\circ
(M\otimes \nabla_{N\otimes P})=
\end{equation}
\[
(M\otimes \nabla_{N\otimes P})\circ (\nabla_{M\otimes N}\otimes P).
\]
\end{apart}

\begin{proposition}\label{MxNWO} Let $D$ be a WBHA with invertible antipode in $\mathcal{C}$. Given $(M, \varphi_{M}, \varrho_{M})$ and
$(N, \varphi_N, \varrho_N)$ in $_D^D\mathcal{YD}$, the quadruple
$(r_{M\times N}, r_{M\times N}^\prime, s_{M\times N}, s_{M\times
N}^\prime)$ is an $(M\times N, D)$-WO, where:
\[ r_{M\times N}= (D\ot p_{M\ot N})\co (r_M\ot N)\co(M\ot r_N)\co(i_{M\ot N}\ot D),\]
\[  r_{M\times N}^\prime= (p_{M\ot N}\ot D)\co(M\ot r_N^\prime)\co(r_M^\prime\ot
N)\co(D\ot i_{M\ot N}), \]
\[  s_{M\times N}= (p_{M\ot N}\ot D)\co(M\ot s_N)\co(s_M\ot N)\co(D\ot i_{M\ot N}), \]
\[  s_{M\times N}^\prime= (D\ot p_{M\ot N})\co (s_M^\prime \ot N)\co(M\ot
s_N^\prime)\co(i_{M\ot N}\ot D).\]
\end{proposition}

\begin{proof}We must check that the conditions stated in Definition \ref{WO} are
satisfied. The proof of (c1), (c2) and (c4) consists basically on use
twice these conditions referred to $M$ and $N$, apply  the
statements obtained in Lemma \ref{nablaMxNbaixapoloWO} and the
equality $\nabla_{M\ot N}= i_{M\ot N}\co p_{M\ot N}.$ We write
(c1-1) to illustrate the procedure:
\begin{itemize}
\item[]$\hspace{0.38cm} (D\ot r_{M\times N})\co(r_{M\times N}\ot D)\co(M\times N\ot
t_{D,D})$
\item[]$=(D\ot D\ot p_{M\ot N})\co(D\ot r_M\ot N)\co(D\ot M\ot
r_N)\co(((D\ot\nabla_{M\ot N})\co(r_M\ot N)$
\item[]$\hspace{0.38cm}\co(M\ot r_N))\ot D)\co (i_{M\ot N}\ot t_{D,D})$
\item[]$=(D\ot D\ot p_{M\ot N})\co(D\ot r_M\ot N)\co(D\ot M\ot r_N)\co(r_M\ot N\ot
D)\co(M\ot r_N\ot D) $
\item[]$\hspace{0.38cm}\co(i_{M\ot N}\ot t_{D,D})$
\item[]$=(t_{D,D}\ot p_{M\ot N})\co(D\ot r_M\ot N)\co(r_M\ot r_N)\co(M\ot r_N\ot
D)\co(i_{M\ot N}\ot D\ot D)$
\item[]$=(t_{D,D}\ot p_{M\ot N})\co(D\ot r_M\ot N) \co(D\ot M\ot r_N)\co
(((D\ot\nabla_{M\ot N}) \co(r_M\ot N)$
\item[]$\hspace{0.38cm}\co(M\ot r_N)\co(i_{M\ot N}\ot D))\ot D)$
\item[]$=(t_{D,D}\ot M\times N)\co(D\ot r_{M\times N})\co(r_{M\times N}\ot D).$
\end{itemize}

The condition (c5) follows directly applying Proposition
\ref{lambabaixapoloWO} twice for $M$ and $N.$

As far as the condition (c3), we prove only (c3-1) because the others are
analogous. Using the definition of $\nabla_{M\times N},$ Lemma
\ref{nablaMxNbaixapoloWO}, the condition (c3-1) referred to $M$ and $N$,
and the condition (c4) referred to $N$ it follows that:
\begin{itemize}
\item[ ]$\hspace{0.38cm} \nabla_{r_{M\times N}}$
\item[ ]$= r'_{M\times N}\circ r_{M\times N}$
\item[ ]$=(p_{M\ot N}\ot D)\circ (M\ot r'_N)\circ (r'_M\ot N)\circ(D\ot \nabla_{M\ot
N})\circ (r_M\ot N)\circ (M\ot r_N)$
\item[]$\hspace{0.38cm}\circ (i_{M\ot N}\ot D)$
\item[ ]$=(p_{M\ot N}\ot D)\circ (M\ot r'_N)\circ (\nabla_{r_{M}}\ot N)\circ (M\ot
r_N)\circ (i_{M\ot N}\ot D)$
\item[ ]$=(\varepsilon_D\ot p_{M\ot N}\ot D)\circ (r_M\ot r'_N)\circ (M\ot
\delta_D\ot N)\circ (M\ot r_N)\circ (i_{M\ot N}\ot D)$
\item[ ]$=(\varepsilon_D\ot p_{M\ot N}\ot D)\circ (r_M\ot \nabla_{r_{N}})\circ (M\ot
r_N\ot D)\circ (i_{M\ot N}\ot \delta_D)$
\item[ ]$=(((\varepsilon_D\ot p_{M\ot N})\circ (r_M\ot \varepsilon_D\ot N)\circ
(M\ot D\ot r_N))\ot D)\circ (M\ot r_N\ot \delta_D)\circ (i_{M\ot
N}\ot \delta_D)$
\item[ ]$=(((\varepsilon_D\ot p_{M\ot N})\circ (r_M\ot \varepsilon_D\ot N)\circ
(M\ot \delta_D\ot N)\circ (M\ot r_N))\ot D)\circ (i_{M\ot N}\ot
\delta_D)$
\item[ ]$=(((\varepsilon_D\ot M\times N)\circ r_{M\times N})\ot D)\circ (M\times
N\ot \delta_D).$
\end{itemize}

\end{proof}

\begin{proposition}\label{MxNeYD}Let $D$ be a WBHA with invertible antipode. If $(M, \varphi_{M}, \varrho_{M})$ and $(N,
\varphi_N, \varrho_N)$ are objects in $_D^D\mathcal{YD},$ then
$(M\times N,\varphi_{M\times N},\varrho_{M\times N})$ is in
$_D^D\mathcal{YD}$, where
\begin{equation}
\varphi_{M\times N}=p_{M\otimes N}\circ \varphi_{M\otimes N}\circ
(D\otimes i_{M\otimes N}),
\end{equation}
\begin{equation}
 \varrho_{M\times N}=(D\otimes p_{M\otimes
N})\circ \varrho_{M\otimes N}\circ  i_{M\otimes N}.
\end{equation}
\end{proposition}

\begin{proof} In Proposition \ref{MxNWO}, an $(M\times
N,D)$-WO is explicitly defined, so it only remains to prove that
$\varphi_{M\times N}$ and $\varrho_{M\times N}$ are compatible
(co)module structures satisfying the conditions (yd1) and (yd2). We
leave to the reader to show that $(M\times N,\varphi_{M\times N})$
is a left $D$-module and $(M\times N, \varrho_{M\times N})$ is a
left $D$-comodule. As far as compatibility, using compatibilities
for $M$ and $N$, the condition (b3-3) and  the equalities (\ref{nablaephi_MN}) and (\ref{rara1}) referred to
$M$ we have:
\begin{itemize}
\item[]$\hspace{0.38cm}r_{M\times N}\co(\varphi_{M\times N}\ot D)$
\item[]$=(D\ot p_{M\otimes N})\co(r_M\ot N)\co(M\ot r_N)\co((\nabla_{M\ot
N}\co\varphi_{M\ot N}\co (D\ot i_{M\otimes N}))\ot D)$
\item[]$=(D\ot p_{M\otimes N})\co(r_M\ot N)\co(M\ot
r_N)\co(((\varphi_M\ot\varphi_N)\co(D\ot s_M\ot N)$
\item[]$\hspace{0.38cm}\co(\delta_D\ot i_{M\otimes
N}))\ot D)$
\item[]$=(D\ot p_{M\otimes N})\co(D\ot\varphi_M\ot N)\co(t_{D,D}\ot M\ot N)\co(D\ot
((r_M\ot\varphi_N)\co(M\ot t_{D,D}\ot N)$
\item[]$\hspace{0.38cm}\co(s_M\ot r_N)))\co(\delta_D\ot i_{M\ot N}\ot D)$
\item[]$=(D\ot(p_{M\ot N}\co(\varphi_M\ot\varphi_N)))\co(t_{D,D}\ot s_M\ot
N)\co(D\ot t_{D,D}\ot M\ot N)$
\item[]$\hspace{0.38cm}\co(\delta_D\ot((r_M\ot N)\co(M\ot r_N)\co (i_{M\ot N}\ot D)))$
\item[]$=(D\ot p_{M\ot N})\co (D\ot (\varphi_{M\ot N}\co (D\ot\nabla_{M\ot N})))\co
(t_{D,D}\ot M\ot N)$
\item[]$\hspace{0.38cm}\co(D\ot ((r_M\ot N)\co(M\ot r_N)\co (i_{M\ot N}\ot D)))$
\item[]$=(D\ot\varphi_{M\times N})\co(t_{D,D}\ot M\times N)\co(D\ot r_{M\times N}).$
\end{itemize}

 The proofs for $r_{M\times N}^\prime$, $s_{M\times N}$ and $s_{M\times N}^\prime,$ are
analogous. By similar arguments we get the result for the comodule
structure.

To prove the condition (yd1) we write:
\begin{itemize}
\item[]$\hspace{0.38cm}\varphi_{M\times N}\co (\Pi_D^L\ot M\times N)\co
\varrho_{M\times N}$
\item[]$=p_{M\ot N}\co(\varphi_M\ot\varphi_N)\co(D\ot s_M\ot N)\co((\delta_D\co
\Pi_D^L)\ot M\ot N)\co \varrho_{M\ot N}\co i_{M\ot N}$
\item[]$=p_{M\ot N}\co(\varphi_M\ot\varphi_N)\co(\mu_D\ot s_M\ot N)\co (\Pi_D^L\ot
(\delta_D\co \eta_D)\ot M\ot N)\co \varrho_{M\ot N}\co i_{M\ot N}$
\item[]$=p_{M\ot N}\co(\varphi_M\ot N)\co(\Pi_D^L\ot (((\varphi_M\co s_M^\prime)\ot
\varphi_N)\co (M\ot (\delta_D\co \eta_D)\ot N)))\co \varrho_{M\ot
N}\co i_{M\ot N}$
\item[]$=p_{M\ot N}\co(\varphi_M\ot N)\co (\Pi_D^L\ot \nabla_{M\ot N})\co
\varrho_{M\ot N}\co i_{M\ot N}$
\item[]$=p_{M\ot N}\co(\varphi_M\ot N)\co (\Pi_D^L\ot i_{M\ot N})\co
\varrho_{M\times N}$
\item[]$=((\varepsilon_D\co \mu_D)\ot M\times N)\co (\Pi_D^L\ot \varrho_{M\times
N})\co \varrho_{M\times N}$
\item[]$=((\varepsilon_D\co (\Pi_D^L\wedge id_D)\ot M\times N)\co \varrho_{M\times N}$
\item[]$=id_{M\times N}.$
\end{itemize}
In the preceding calculations, the first equality follows by
(\ref{nablamorreconphi}), while in the second one we apply that
$\delta_D\co\Pi_D^L=(\mu_D\ot D)\co(\Pi_D^L\ot(\delta_D\co\eta_D))$.
In the third one we use  (\ref{etadeltas}); the fourth one
follows  because of the characterization
\[\nabla_{M\ot N}=((\varphi_M\co s_M^\prime)\ot \varphi_N)\co (M\ot (\delta_D\co
\eta_D)\ot N)\] obtained in the proof of Proposition
\ref{asnablascoinciden}. The fifth equality follows by the definition of $\varrho_{M\times N}$
, and the sixth one
by  the equality
\[p_{M\ot N}\co(\varphi_M\ot N)\co(\Pi_D^L\ot i_{M\ot
N})=((\varepsilon_D\co\mu_D)\ot M\times
N)\co(\Pi_D^L\ot\varrho_{M\times N}),\] that in turn can be deduced
using (b5) and (\ref{newpiLphirhomu}). Finally, in the
seventh equality we use that $(M\times N, \varrho_{M\times N})$ is a
left $D$ comodule and the last one follows by (\ref{Piconvolucion}).

As a consequence (yd1) holds:
\begin{itemize}
\item[]$\hspace{0.38cm}(\mu_D\ot \varphi_{M\times N})\co (D\ot t_{D,D}\ot M\times
N)\co ((\delta_D\co \eta_D)\ot \varrho_{M\times N})$
\item[]$=(D\ot \varphi_{M\times N})\co (D\ot \Pi_D^L\ot M\times N)\co (\delta_D\ot
M\times N)\co \varrho_{M\times N}$
\item[]$=\varrho_{M\times N}.$
\end{itemize}

To prove (yd2), using similar technics and results together with
(\ref{nablaephi_MN}), (\ref{nablamorreconphi})  and the condition (yd2) referred to $M$
and $N$  we get:
\begin{itemize}
\item[]$\hspace{0.38cm} (\mu_D\ot {M\times N})\co(D\ot r_{M\times
N})\co((\varrho_{M\times N}\co\varphi_{M\times N})\ot D)\co(D\ot
s_{M\times N})\co(\delta_D\ot {M\times N})$
\item[]$=
(\mu_D\ot p_{M\ot N})\co (D\ot r_M\ot N) \co(D\ot M\ot r_N)
\co(((\mu_D\ot \nabla_{M\ot N})\co(D\ot r_M\ot N)$
\item[]$\hspace{0.38cm}\co(\varrho_M\ot\varrho_N)\co \nabla_{M\ot N}\co(\varphi_M\ot\varphi_N)\co(D\ot s_M\ot N)\co(\delta_D\ot i_{M\ot N}))\ot D)$
\item[]$\hspace{0.38cm}
\co(D\ot((p_{M\ot N}\ot D)\co(M\ot s_N)\co(s_M\ot N)))
\co(\delta_D\ot i_{M\ot N})$
\item[]$=(\mu_D\ot p_{M\ot N})\co(D\ot r_M\ot
N)\co((\varrho_M\co\varphi_M)\ot((\mu_D\ot N)\co(D\ot
r_N)\co((\varrho_N\co\varphi_N)\ot D)$
\item[]$\hspace{0.38cm}\co(D\ot s_N)\co(\delta_D\ot N)))\co(D\ot s_M\ot
N)\co(\delta_D\ot i_{M\ot N})$
\item[]$=(\mu_D\ot p_{M\ot N})\co(D\ot r_M\ot
N)\co((\varrho_M\co\varphi_M)\ot((\mu_D\ot\varphi_N)\co(D\ot
t_{D,D}\ot N)\co(\delta_D\ot \varrho_N)))$
\item[]$\hspace{0.38cm}\co(D\ot s_M\ot N)\co(\delta_D\ot i_{M\ot N})$
\item[]$=(\mu_D\ot p_{M\ot N})\co(D\ot r_M\ot\varphi_N)\co(((\mu_D\ot M)\co(D\ot
r_M)\co((\varrho_M\co\varphi_M)\ot D)\co(D\ot s_M)$
\item[]$\hspace{0.38cm}\co(\delta_D\ot M))\ot t_{D,D}\ot N)\co(D\ot
s_M\ot\varrho_N)\co(\delta_D\ot i_{M\ot N})$
\item[]$=(\mu_D\ot p_{M\ot N})\co(D\ot
r_M\ot\varphi_N)\co(((\mu_D\ot\varphi_M)\co(D\ot t_{D,D}\ot M)\co(\delta_D\ot\varrho_M))\ot t_{D,D}\ot N)$
\item[]$\hspace{0.38cm}\co(D\ot
s_M\ot\varrho_N)\co(\delta_D\ot i_{M\ot N})$
\item[]$=(\mu_D\ot p_{M\ot N})\co(D\ot ((D\ot\varphi_M\ot N)\co(t_{D,D}\ot M\ot
N)\co(D\ot r_M\ot\varphi_N)))$
\item[]$\hspace{0.38cm}\co(\mu_D\ot D\ot M\ot t_{D,D}\ot N)\co(D\ot((t_{D,D}\ot
s_M\ot D\ot N)\co(D\ot t_{D,D}\ot M\ot D\ot N)$
\item[]$\hspace{0.38cm}\co(\delta_D\ot\varrho_M\ot\varrho_N)))\co(\delta_D\ot
i_{M\ot N})$
\item[]$=(\mu_D\ot p_{M\ot N})\co(D\ot((\mu_D\ot\varphi_M\ot\varphi_N)\co(D\ot t_{D,D}\ot
s_M\ot N)\co(t_{D,D}\ot t_{D,D}\ot M\ot N)$
\item[]$\hspace{0.38cm}\co(D\ot t_{D,D}\ot r_M\ot
N)\co(\delta_D\ot\varrho_M\ot\varrho_N)))\co(\delta_D\ot i_{M\ot
N})$
\item[]$=(\mu_D\ot (p_{M\ot N}\co \varphi_{M\ot N}))\co (D\ot t_{D,D}\ot M\ot N)\co
(\delta_D\ot (\varrho_{M\ot N}\co i_{M\ot N}))$
\item[]$=(\mu_D\ot\varphi_{M\times N})\co( D\ot t_{D,D}\ot {M\times
N})\co(\delta_D\ot\varrho_{M\times N}).$
\end{itemize}
\end{proof}

We proceed now to state and prove the main result of this work,
giving an explicit description of all the required components of the
monoidal structure for ${}_D^D\mathcal{YD}$.

\begin{theorem}\label{estruturamonoidal}
Let $D$ be a WBHA with
invertible antipode. Then ${}_D^D\mathcal{YD}$ is a non-strict
monoidal category.
\end{theorem}

\begin{proof} Given  $(M\varphi_M, \varrho_M)$ and $(N,\varphi_N,\varrho_N)$ two
objects in $_D^D\mathcal{YD}$ we define as its product $M\times
N$ the image of the idempotent $\nabla_{M\ot N}$, that by
Proposition \ref{MxNeYD} is a left-left Yetter-Drinfeld module with
associated weak operator the one defined in Propositon \ref{MxNWO}.

The base object is $D_{L}=Im(\Pi_{D}^{L});$ with  left
$D$-(co)module structure
\begin{equation}
\varphi_{D_{L}}=p_{L}\circ \mu_{D}\circ (D\otimes i_{L}),\;\;\;\;
\varrho_{D_{L}}=(D\otimes p_{L})\circ \delta_{D}\circ i_{L},
\end{equation}
where $p_{L}:D\rightarrow D_{L}$ and $i_{L}:D_{L}\rightarrow D$ are
the morphisms such that $\Pi_{D}^{L}=i_{L}\circ p_{L}$ and
$p_{L}\circ i_{L}=id_{D_{L}}$.

It holds that $(r_{D_L}, r_{D_L}^\prime, s_{D_L}, s_{D_L}^\prime)$
is a $(D_L,D)$-WO compatible with the (co)module structures of
$D_L$, being
\begin{itemize}
\item[ ] $r_{D_L}:= (D\ot p_L)\co t_{D,D}\co (i_L\ot D);\quad r_{D_L}^\prime:=
(p_L\ot D)\co t^{\prime}_{D,D}\co (D\ot i_L)$
\item[ ] $s_{D_L}:= (p_L\ot D)\co t_{D,D}\co (D\ot i_L);\quad s_{D_L}^\prime:= (D\ot
p_L)\co t^{\prime}_{D,D}\co (i_L\ot D).$
\end{itemize}
The triple  $(D_L,\varphi_{D_L},\varrho_{D_L})$ satisfies (yd1) and
(yd2) because it corresponds to the particular case of the
projection $(D,id_D,id_D)$ over $D$ [\cite{Proj}, Definition 2.7,
Proposition 2.19] and then
$(D_L,\varphi_{D_L},\varrho_{D_L})$ is in $_D^D\mathcal{YD}$.

The unit constrains are:
\begin{equation}
\mathfrak{l}_ M=\varphi_{M}\circ (i_{L}\otimes M)\circ
i_{D_{L}\otimes M}:D_{L}\times M\rightarrow M,
\end{equation}
\begin{equation}
 \mathfrak{r}_{M}=\varphi_{M}\co s_M^\prime\co (M\otimes
(\overline{\Pi}_{D}^{L}\circ i_{L}))\co i_{M\otimes D_{L}}:M\times
D_{L}\rightarrow M.
\end{equation}
These morphisms are isomorphisms with inverses:
\begin{equation}
\mathfrak{l}_{M}^{-1}=p_{D_{L}\otimes M}\circ (p_{L}\otimes
\varphi_{M})\circ ((\delta_{D}\circ \eta_{D})\otimes M):M\rightarrow
D_{L}\times M,
\end{equation}
\begin{equation}
\mathfrak{r}_{M}^{-1}=p_{M\otimes D_{L}}\circ (\varphi_{M}\otimes
p_{L})\circ (D\otimes s_M)\circ ((\delta_{D}\circ \eta_{D})\otimes
M):M\rightarrow M\times D_{L},
\end{equation}
and they are actually morphisms of $_D^D\mathcal{YD}.$ We write the
proof for one of the required equalities, the remaining being
analogous. In fact:
\begin{itemize}
\item[]$\hspace{0.38cm}r_M\co(\mathfrak{r}_{M}\ot D)$
\item[]$=r_M\co(((\varphi_M\co s_M^\prime\co(M\ot(\overline{\Pi}_D^L\co i_L))\co
i_{M\ot D_L}))\ot D)$
\item[]$=(D\ot\varphi_M)\co(t_{D,D}\ot M)\co(D\ot r_M)\co(s_M^\prime\ot
D)\co(M\ot(\overline{\Pi}_D^L\co i_L)\ot D)\co(i_{M\ot D_L}\ot D)$
\item[]$=(D\ot(\varphi_M\co s_M^\prime))\co(r_M\ot D)\co(M\ot
t_{D,D})\co(M\ot(\overline{\Pi}_D^L\co i_L)\ot D)\co(i_{M\ot D_L}\ot
D)$
\item[]$=(D\ot(\varphi_M\co s_M^\prime\co(M\ot (\overline{\Pi}_D^L\co
i_L))))\co(r_M\ot D_L)\co(M\ot r_{D_L})\co(i_{M\ot D_L})\ot D)$
\item[]$=(D\ot(\varphi_M\co s_M^{}\prime\co(M\ot (\overline{\Pi}_D^L\co i_L)
)))\co(D\ot\nabla_{M\ot D_L})\co(r_M\ot D_L)\co(M\ot r_{D_L})$
\item[]$\hspace{0.38cm}\co(i_{M\ot D_L}\ot D)$
\item[]$=(D\ot \mathfrak{r}_{M})\co r_{M\times D_L}.$
\end{itemize}

If $M$, $N$, $P$ are objects in the category ${}_D^D\mathcal{YD}$,
the associativity constrains are defined by
\begin{equation}
\mathfrak{a}_{M,N,P}=p_{(M\times N)\otimes P}\circ (p_{M\otimes
N}\otimes P)\circ (M\otimes i_{N\otimes P})\circ i_{M\otimes
(N\times P)}:M\times (N\times P)\rightarrow (M\times N)\times P.
\end{equation}
Its inverse is
\begin{equation}
\mathfrak{a}_{M,N,P}^{-1}=p_{M\otimes (N\times P)}\circ (M\otimes
p_{N\otimes
 P})\circ (i_{M\otimes N}\otimes P)\circ i_{(M\times N)\otimes P}:
 (M\times N)\times P \rightarrow M\times (N\times P).
\end{equation}
Using (\ref{simetriadenabla}), (\ref{ieptrocandelado}) and Lemma
\ref{nablaMxNbaixapoloWO} we check that they are morphisms of
left-left Yetter-Drinfeld modules, and in turn  this fact allows us
to prove the triangle and the pentagon axioms.

As far as tensor products of morphisms in ${}_D^D\mathcal{YD}$ is
concerned, if $\gamma:M\rightarrow M^{\prime}$ and
$\phi:N\rightarrow N^{\prime}$ are morphisms in the category, we
define
\begin{equation}
\gamma\times \phi=p_{M^{\prime}\times N^{\prime}}\circ
(\gamma\otimes \phi)\circ i_{M\otimes N}:M\times N\rightarrow
M^{\prime}\times N^{\prime},
\end{equation}
which is a morphism in ${}_D^D\mathcal{YD}$ and
\begin{equation}
(\gamma^{\prime}\times \phi^{\prime})\circ (\gamma\times \phi)=
(\gamma^{\prime}\circ \gamma)\times (\phi^{\prime}\circ \phi),
\end{equation}
where $\gamma^{\prime}:M^{\prime}\rightarrow M^{\prime\prime}$ and
$\phi^{\prime}:N^{\prime}\rightarrow N^{\prime\prime}$ are morphisms
in ${}_D^D\mathcal{YD}$. \end{proof}

\begin{remark}
In the particular case where the category $\CC$ is braided with braiding $c$
and we take $(c_{M,D}, c_{M,D}^{-1}, c_{D,M},
c_{D,M}^{-1})$ as the $(M,D)$-WO,  the formal properties of the
braiding simplify the calculations, but it is important to note that
the global definition of the braiding is not an essential component in the notion of
Yetter-Drinfeld module.
\end{remark}

\section{Projections and Yetter-Drinfeld modules}

In this section we illustrate the preceding  definitions with a
family of examples, those coming from projections. These examples
are especially relevant for various reasons. One of them lies on its
physics motivations. In a braided category the bosonization introduced by Majid
 in \cite{MAJ2} induces examples of projections. On the other hand, the Radford theory shows  the key role that
projections play  in the theory of Yetter-Drinfel'd modules.

We briefly recall the definition and main properties of projections of
WBHA. The details can be found in [\cite{Proj}, Section 1].

\begin{definition}\label{WBHAproj}
Let $D$, $B$ be WBHA. A projection for $D$ is a triple
$(B,f,g)$ where $f:D\rightarrow B$ and $g:B\rightarrow D$ are
morphisms of WBHA such that $g\co f={id}_D;$ and satisfying the
following equalities:
\begin{itemize}
\item[(i)]
$(B\ot(f\co g))\co t_{B,B}=t_{B,B}\co((f\co g)\ot B).$
\item[(ii)]
$((f\co g)\ot B)\co t_{B,B}=t_{B,B}\co(B\ot(f\co g)).$
\end{itemize}

A morphism between two projections $(B,f,g)$ and
$(B^{\prime},f^{\prime},g^{\prime})$ associated to $D$ is a
morphism of WBHA $h:B\rightarrow B^{\prime}$ such that $h\co f=f^{\prime}$
and $g^{\prime}\co h=g.$ The set of projections associated to
$D$ and morphisms of projections is a category, which we will
denote by ${\mathcal P}roj(D).$

\end{definition}

\begin{remark}
Notice that simultaneous verification of the conditions $(i)$ and $(ii)$ in Definition \ref{WBHAproj}
for $t_{B,B}$ is equivalent to its verification for
${t^{\prime}}_{B,B}.$
\end{remark}

\begin{proposition}
\label{estruturasB_D}Let  $D$ be a WBHA and let $(B,f,g)$ be an object in
${\mathcal P}roj(D).$ The morphism\\
$q_D^B: B\longrightarrow B,$ defined as
\[q_D^B:={id}_B\wedge(f\co\lambda_D \co g),\]
is an idempotent.

As a consequence there are an object $B_D, $ an epimorphism
$p_D^B:B\rightarrow B_D,$ and a monomorphism $i_D^B: B_D\rightarrow
B$ such that $q_D^B=i_D^B\co p_D^B$ and $p_D^B\co i^B_D={id}_{B_D}$.
Moreover
$(B_D,\eta_{B_D}=p_D^B\co\eta_B,\mu_{B_D}=p_D^B\co\mu_B\co(i_D^B\ot
i_D^B))$ is an algebra and $(B_D,\varepsilon_{B_D}=\varepsilon_B\co
i_D^B,\delta_{B_D}=(p_D^B\ot p_D^B)\co\delta_B\co i_D^B)$ is a
coalgebra in ${\mathcal C}$ and
$(B_D,\varphi_{B_D}=p_D^B\co\mu_B\co(f\ot i_D^B))$ is a left
$D$-module and $(B_D,\rho_{B_D}=(g\ot p_D^B)\co\delta_B\co i_D^B)$
is a left $D$-comodule.
\end{proposition}

\begin{proof} See [\cite{Proj}, Propositions 2.11, 2.13 and 2.17]. \end{proof}

\begin{proposition}\label{B_DWO}
Let $D$ be a WBHA and $(B,f,g)\in |{\mathcal P} roj (D)|$. We
define:
\begin{itemize}
\item[ ] $r_{B_D}:= (g\ot p_D^B)\co t_{B,B}\co (i_D^B\ot f);\quad r_{B_D}^\prime:=
(p_D^B\ot g)\co t^{\prime}_{B,B}\co (f\ot i_D^B),$
\item[ ] $s_{B_D}:= (p_D^B\ot g)\co t_{B,B}\co (f\ot i_D^B);\quad s_{B_D}^\prime:=
(g\ot p_D^B)\co t^{\prime}_{B,B}\co (i_D^B\ot f).$
\end{itemize}

It holds that the quadruple $(r_{B_D},r_{B_D}^\prime,s_{B_D},
s_{B_D}^\prime)$ is a $(B_D,D)$-WO compatible with the (co)module
structure defined for $B_D$ in Proposition \ref{estruturasB_D}.
\end{proposition}

\begin{proof} On each condition, just some parts are proved to illustrate the technics applied,
the remaining being analogous.

On the condition (c1) we check (c1-1) explicitely:
\begin{itemize}
\item[ ]$\hspace{0.38cm}  (D\ot r_{B_D})\co(r_{B_D}\ot D)\co(B_D\ot t_{D,D})$
\item[ ]$= (D\ot((g\ot p_D^B)\co t_{B,B}))\co(((g\ot q_D^B)\co t_{B,B})\ot
B)\co(i_D^B\ot((f\ot f)\co t_{D,D}))$
\item[ ]$=(g\ot g\ot p_D^B)\co(B\ot t_{B,B})\co(t_{B,B}\ot
B)\co(i_D^B\ot (t_{B,B}\co (f\ot f)))$
\item[ ]$=(((g\ot g)\co t_{B,B})\ot p_D^B)\co(B\ot t_{B,B})\co(t_{B,B}\ot
B)\co(i_D^B\ot f\ot f)$
\item[ ]$=(t_{D,D}\ot B_D)\co(g\ot g\ot p_D^B)\co(B\ot t_{B,B})\co(B\ot q_D^B\ot
B)\co(t_{B,B}\ot B)\co(i_D^B\ot f\ot f)$
\item[ ]$=(t_{D,D}\ot B_D)\co(D\ot r_{B_D})\co (r_{B_D}\ot D).$
\end{itemize}
In the preceding calculations, the first and the last equalities hold by
the definition of $r_{B_D}$, on the second one we use [\cite{Proj}, Lemma 2.16] and
the fact that $f$ is a morphism of WBHA. The third one follows because
of $t_{B,B}$ verifies the Yang-Baxter equation, and the fourth one uses
[\cite{Proj}, Lemma 2.16] together with  the character of morphism of WBHA for $g$.

The proof for the condition (c2) is similar, but using the equality
(\ref{tprim1}) instead of the verification of the Yang-Baxter
equation.

For (c3-1), using that $B$ is a WBHA, the definition of projection, [\cite{Proj}, Lemma 2.16],
 (b2-3) and the equality (\ref{b3-5}) we have:
\begin{itemize}
\item[ ]$\hspace{0.38cm}\nabla_{r_{B_D}}$
\item[ ]$=r_{B_D}\co r_{B_D}^\prime$
\item[ ]$=(p_D^B\ot g)\co t^{\prime}_{B,B} \co ((f\co g)\ot q_D^B)\co t_{B,B} \co (i_D^B\ot f)$
\item[ ]$=(p_D^B\ot g)\co \nabla_{B,B} \co (i_D^B\ot f)$
\item[ ]$=(p_D^B\ot \varepsilon_B\ot g)\co (\nabla_{B,B}\ot B) \co (i_D^B\ot (\delta_B\co f))$
\item[ ]$=(\varepsilon_B\ot p_D^B\ot g)\co (t_{B,B}\ot B) \co (i_D^B\ot (\delta_B\co f))$
\item[ ]$=(\varepsilon_B\ot B_D\ot D)\co (r_{B_D}\ot D)\co (B_D\ot \delta_D).$

\end{itemize}

Arguing analogously we obtain that
\[\nabla_{r_{B_D}}=(\mu_D\ot B_D)\co(D\ot(r_{B_D}\co(B_D\ot\eta_D))).\]

For the condition (c4-1), by the definition of projection, [\cite{Proj}, Lemma 2.16]
 and the properties of the weak operator $t_{B,B}$ we
know that:

\begin{itemize}
\item[ ]$\hspace{0.38cm} (\mu_D\ot B_D)\co(D\ot r_{B_D})\co(r_{B_D}\ot D)$
\item[ ]$=(\mu_D\ot B_D)\co(g\ot g\ot p_D^B)\co(B\ot t_{B,B})\co(B\ot q_D^B\ot B)\co(t_{B,B}\ot B)\co(i_D^B\ot f\ot f)$
\item[ ]$=(g\ot p_D^B)\co(\mu_B\ot B)\co(B\ot t_{B,B})\co(t_{B,B}\ot B)\co(i_D^B\ot
f\ot f)$
\item[ ]$=(g\ot p_D^B)\co t_{B,B}\co (i_D^B\ot(\mu_B\co(f\ot f)))$
\item[ ]$=r_{B_D}\co(B_D\ot \mu_D).$
\end{itemize}

Finally, the condition (c5) follows because $f$ and $g$ are
morphisms of WBHA and by [\cite{IND}, Proposition 2.12].

In order to see the compatibility with the (co)module structures  of $B_D,$
 we just state explicitly one of the required equalities to
illustrate the technics. We can write:
\begin{itemize}
\item[ ]$\hspace{0.38cm} (D\ot\varphi_{B_D} )\co(t_{D,D}\ot B_D)\co(D\ot r_{B_D})$
\item[ ]$=(D\ot(p_D^B\co\mu_B\co (f\ot q_D^B)))\co (t_{D,D}\ot B)\co (D\ot g\ot B)\co (D\ot (t_{B,B}\co (i_D^B\ot f)))$
\item[ ]$=(g\ot(p_D^B\co\mu_B))\co(t_{B,B}\ot B)\co(B\ot t_{B,B})\co (f\ot i_D^B\ot f)$
\item[ ]$=(g\ot p_D^B)\co t_{B,B}\co(\mu_B\ot B)\co(f\ot i_D^B\ot f)$
\item[ ]$=(g\ot p_D^B)\co t_{B,B}\co((i_D^B\co p_D^B\co\mu_B)\ot B)\co(f\ot i_D^B\ot
f)$
\item[ ]$=r_{B_D}\co(\varphi_{B_D}\ot D).$
\end{itemize}

The remaining equalities  are analogous. \end{proof}

The above disquisitions allow to state one of the main results of
this section:

\begin{proposition}\label{B_DisYD}
Let $D$ be a WBHA and $(B,f,g)\in |{\mathcal P} roj (D)|$. With the
notation of Proposition \ref{estruturasB_D}, $(B_D,\varphi_{B_D},
\varrho_{B_D})$ is in $_D^D\mathcal{YD}$.
\end{proposition}

\begin{proof} We have already shown that the quadruple $(r_{B_D}, r_{B_D}^\prime,
s_{B_D}, s_{B_D}^\prime)$ defined in Proposition \ref{B_DWO} is a
$(B_D,D)$-WO compatible with the (co)-module structure. For the
conditions (yd1) and (yd2) see [\cite{Proj}, Proposition 1.19].
\end{proof}

\begin{remark}
This example arising from projections of WBHA also suggests that the
requirement introduced in part (ii) of Definition \ref{Y-Dmorphism}
is natural in the sense that it is automatically satisfied in the
case of this generic example. Actually, by definition of morphism
between projections over $D$, given such a morphism $\alpha:
B\rightarrow B^\prime$  between $(B,f,g)$ and
$(B^\prime,f^\prime,g^\prime)$ (see Definition \ref{WBHAproj}) it
induces a (co)module morphism $\alpha_D:B_D\rightarrow B_D^\prime$
such that ${i_D^B}^\prime\co\alpha_D=\alpha\co i_D^B.$ Then we have
\begin{itemize}
\item[ ]$\hspace{0.38cm} r_{B_{D}^{\prime}}\co(\alpha_D\ot D)$
\item[ ]$= (g^\prime\ot p_D^{B^\prime})\co t_{B^\prime,B^\prime} \co
((i_D^{B^\prime}\co\alpha_D)\ot f^\prime)$
\item[ ]$= (g^\prime\ot p_D^{B^\prime})\co t_{B^\prime,B^\prime} \co((\alpha\co
i_D^B)\ot(\alpha\co f))$
\item[ ]$= ((g^\prime\co\alpha)\ot( p_D^{B^\prime}\co\alpha))\co t_{B,B}
\co(i_D^B\ot f)$
\item[ ]$=(g\ot(\alpha_D\co p_D^{B}))\co t_{B,B} \co(i_D^B\ot f)$
\item[ ]$=(D\ot \alpha_D)\co r_{B_D}.$
\end{itemize}

We would obtain similarly the analogous results for
$r_{B_D}^\prime,$ $s_{B_D}$ and $s_{B_D}^\prime$ instead of
$r_{B_D},$ so any morphism in $\mathcal{P}roj(D)$ induces naturally
a morphism in $_D^D\mathcal{YD}$.
\end{remark}

On these examples coming from projections the construction of the
weak operator is based on the weak Yang-Baxter operator $t_{B,B}$
and its properties. We will finish this section seeing a link between
the notions of weak Yang-Baxter operator and weak entwining
structure, being the last one relevant, for example, in order to
give a characterization of weak cleft extensions in terms of weak
Galois extensions  with normal basis, as can be found in \cite{iwes}
. To do so,  the definition of invertible weak entwining  is briefly
recalled (see \cite{iwes} for details).

\begin{definition}\label{wes}
A right-right weak entwining structure is a triple $(A,C,\Psi_{RR})$
where $A$ is an algebra, $C$ a coalgebra and $\Psi_{RR}: C\ot
A\rightarrow A\ot C$ is a morphism that satisfies:
\begin{equation}\label{rrwes1}
(\mu_A\ot C)\co(A\ot\Psi_{RR})\co(\Psi_{RR}\ot A)=\Psi_{RR}\co(C\ot\mu_A),
\end{equation}
\begin{equation}\label{rrwes2}
\Psi_{RR}\co(C\ot\eta_A)=(e_{RR}\ot C)\co\delta_C,
\end{equation}
\begin{equation}\label{rrwes3}
(A\ot\delta_C)\co\Psi_{RR}=(\Psi_{RR}\ot C)\co(C\ot\Psi_{RR})\co(\delta_C\ot
A),
\end{equation}
\begin{equation}\label{rrwes4}
(A\ot\varepsilon_C)\co\Psi_{RR}=\mu_A\co(e_{RR}\ot A).
\end{equation}

 with $e_{RR}=(A\ot\varepsilon_C)\co\Psi_{RR}\co(C\ot\eta_A).$ Similarly we can define
a left-left weak entwining structure $(A,C,\Psi_{LL})$ for an algebra
$A,$ a coalgebra $C$ and a morphism $\Psi_{LL}:A\ot C\rightarrow C\ot
A$ that verifies similar equalities to the previous ones with
$e_{LL}=(\varepsilon_C\ot A)\co\Psi_{LL}\co(\eta_A\ot C).$
 \end{definition}

 \begin{apart}
 Let $(A,C,\Psi_{RR})$ be a right-right weak entwining structure. Define $\Delta_{RR}:A\ot
C\rightarrow A\ot C$ by $\Delta_{RR}=(\mu_A\ot
C)\co(A\ot\Psi_{RR})\co(A\ot C\ot \eta_A).$ This morphism is idempotent
and so is the morphism $\nabla_{RR}:C\ot A\rightarrow C\ot A$ defined
by $\nabla_{RR}=(C\ot A\ot \varepsilon_C)\co(C\ot
\Psi_{RR})\co(\delta_C\ot A).$ The corresponding idempotent morphisms
for a left-left weak entwining structure will be denoted by
$\Delta_{LL}$ and $\nabla_{LL}.$
\end{apart}

\begin{definition}\label{iwes}
Let  $A$ be an algebra, $C$ a coalgebra and $\Psi_{RR}: C\ot
A\rightarrow A\ot C$  and $\Psi_{LL}: A\ot C\rightarrow C\ot A$
morphisms in $\CC.$ We say that $(C,A,\Psi_{RR},\Psi_{LL})$ is an
invertible weak entwining structure if the following conditions
hold:
\begin{itemize}
\item[(i)] $(A,C,\Psi_{RR})$ is a right-right weak entwining structure and
$(A,C,\Psi_{LL})$ is a left-left weak entwining structure.
\item[(ii)] $\Psi_{LL}\co\Psi_{RR}=\Delta_{LL}$ and $\Psi_{RR}\co\Psi_{LL}=\Delta_{RR}$
\end{itemize}

\end{definition}

The relation between weak Yang-Baxter operators and invertible weak
entwining structures can be expressed in terms of weak operators as
follows:

\begin{proposition} With the notation of Proposition \ref{B_DWO}, it holds that
\begin{itemize}
\item[(i)] $(B_D, D, s_{B_D}, s_{B_D}^\prime)$ is an invertible weak entwining
structure.
\item[(ii)] $(B_D, D, r_{B_D}^\prime, r_{B_D})$ is an invertible weak entwining
structure.
\end{itemize}
\end{proposition}
\begin{proof}
We prove part (i). Let's see that $(B_D, D,
s_{B_D})$ is a right-right weak entwining structure. First of all,
it was already demonstrated that the quadruple
$(r_{B_D},r_{B_D}^\prime,s_{B_D}, s_{B_D}^\prime )$ is a
$(B_D,D)$-WO, so we know that (\ref{rrwes3}) holds. On the other
hand, using that $(B,f,g)\in |{\mathcal P} roj (D)|$ and the
properties of the weak operator $t_{B,B}$ and [\cite{Proj}, Lemma 2.16] we obtain (\ref{rrwes1}).
Indeed:
\begin{itemize}
\item[ ]$\hspace{0.38cm} (\mu_{B_D}\ot D)\co (B_D\ot s_{B_D})\co (s_{B_D}\ot B_D)$
\item[] $=((p_D^B\co \mu_B\co ((i_D^B\co p_D^B)\ot (i_D^B\co p_D^B)))\ot g)\co (B\ot
t_{B,B})\co (B\ot (f\co g)\ot B)\co (t_{B,B}\ot B)\co (f\ot i_D^B\ot i_D^B)$
\item[] $=((p_D^B\co \mu_B)\ot g)\co (B\ot t_{B,B})\co (t_{B,B}\ot B)\co (f\ot
i_D^B\ot i_D^B)$
\item[] $=(p_D^B\ot g)\co t_{B,B}\co (B\ot \mu_B)\co (f\ot i_D^B\ot i_D^B)$
\item[] $=s_{B_D}\co (D\ot \mu_{B_D}).$
\end{itemize}

 To show (\ref{rrwes2}), note first that in this case
$e_{RR}=(p_D^B\ot\varepsilon_B)\co t_{B,B} \co(f\ot\eta_B),$ so
\begin{itemize}
\item[ ]$\hspace{0.38cm} (e_{RR}\ot D)\co\delta_D$
\item[] $=(p_D^B\ot\varepsilon_B\ot D)\co (t_{B,B}\ot D) \co(f\ot\eta_B\ot
D)\co\delta_D$
\item[] $=(p_D^B\ot\varepsilon_B\ot g)\co(\nabla_{B,B}\ot B)\co(\eta_B\ot f\ot
f)\co\delta_D$
\item[] $=(p_D^B\ot\varepsilon_B\ot g)\co(\nabla_{B,B}\ot
B)\co(\eta_B\ot(\delta_B\co f))$
\item[] $=(p_D^B\ot\varepsilon_B\ot g)\co(B\ot\delta_B)\co\nabla_{B,B}\co(\eta_B\ot f)$
\item[] $=(p_D^B\ot g)\co t_{B,B}\co(f\ot i_D^B)\co(D\ot\eta_{B_D})$
\item[] $=s_{B_D}\co(D\ot\eta_{B_D}).$
\end{itemize}
In the above calculations, the first and the fifth equalities are
just the definition of $e_{RR}$, the second one uses (\ref{b3-1}) and $g\co
f=id_D$; the third one follows because of $f$ is a coalgebra
morphism, and the fourth relies on (b2-3). By similar arguments we
get (\ref{rrwes4}).

The proof  showing that $(B_D, D, s_{B_D}^\prime)$ is a left-left
weak entwining is analogous.

Finally we use similar properties to see that $s_{B_D}\co s_{B_D}^\prime=\Delta_{RR}$:
\begin{itemize}
\item[ ]$\hspace{0.38cm} s_{B_D}\co s_{B_D}^\prime$
\item[ ]$=(p_D^B\ot g)\co t_{B,B}\co ((f\co g)\ot(i_D^B\co p_D^B))\co
t^{\prime}_{B,B}\co(i_D^B\ot f)$
\item[ ]$=(p_D^B\ot g)\co\nabla_{B,B}\co(i_D^B\ot f)$
\item[ ]$=(p_D^B\ot g)\co\nabla_{B,B}\co(\mu_B\ot B)\co(i_D^B\ot\eta_B\ot f)$
\item[ ]$=(p_D^B\ot g)\co(\mu_B\ot B)\co(B\ot\nabla_{B,B})\co(i_D^B\ot\eta_B\ot f)$
\item[ ]$=(p_D^B\ot D)\co(\mu_B\ot B)\co(i_D^B\ot(i_D^B\co p_D^B)\ot g)\co (B\ot
t_{B,B})\co (B_D\ot f\ot(i_D^B\co\eta_{B_D}))$
\item[ ]$=(\mu_{B_D}\ot D)\co(B_D\ot s_{B_D})\co(B_D\ot D\ot\eta_{B_D})$
\item[ ]$=\Delta_{RR}.$
\end{itemize}

It can be checked similarly that $s_{B_D}^\prime\co
s_{B_D}=\Delta_{LL}$. \end{proof}

\section{Adjoint (co)actions and Yetter-Drinfeld modules}

In the theory of Hopf algebras it is a well-known fact that, if $H$
is a Hopf algebra in an strict braided monoidal category with braid
$c$, the triple $(H, \varphi_{H}, \delta_{H})$ is an object of
$_H^H\mathcal{YD}$ where $\varphi_{H}:H\otimes H\rightarrow H$
denotes the adjoint action defined by
\[\varphi_{H}=\mu_{H}\circ (\mu_{H}\ot \lambda_{H})\circ (H\ot c_{H,H})\circ (\delta_{H}\ot H).\]
Also, the triple  $(H, \mu_{H}, \varrho_{H})$ is an object of
$_H^H\mathcal{YD}$ where $\varrho_{H}:H\rightarrow H\otimes H$
denotes the adjoint coaction defined by
\[\varrho_{H}=(\mu_{H}\otimes H)\circ  (H\otimes c_{H,H})\circ (\delta_{H}\otimes \lambda_{H})\circ \delta_{H}.\]

Unfortunately, in the weak setting, the previous assertions are not
true in general (see \cite{Accadj}). Indeed, being $H$ a weak Hopf algebra in ${\mathcal C}$, the
pair $(H, \varphi_{H})$ is not in general a left $H$-module because
the unit condition can fail, i.e.
\[\varphi_{H}\circ (\eta_{H}\ot H)=\mu_{H}\circ (H\ot (\lambda_{H}\circ \Pi_{H}^{L}))\circ \delta_{H}\neq id_{H},\]
and for the adjoint coaction the counit condition may be untrue
because
\[(\varepsilon_{H}\otimes H)\circ\varrho_{H}=\mu_{H}\circ (H\ot (\Pi_{H}^{L}\circ \lambda_{H}))\circ \delta_{H}\neq id_{H}.\]

In this section we shall show that for every WBHA $D$ the adjoint
action and the adjoint coaction induce idempotent morphisms and as a
consequence, using the factorizations of these idempotents, it is
possible to  construct new examples of objects in the category
$_D^D\mathcal{YD}$ defined in the second section of this paper.
Obviously, if $H$ is a Hopf algebra, the  idempotents associated to
the adjoint action and coaction are identities and we recover the
classical results.

\begin{proposition}
\label{adjoint-act-coact} Let $D$ be a WBHA  in ${\mathcal C}$. Let
$\varphi_{D}:D\ot D\rightarrow D$ and  $\varrho_{D}:D\rightarrow
D\ot D$ be the morphisms defined by
\[\varphi_{D}=\mu_{D}\circ (\mu_{D}\ot \lambda_{D})\circ (D\ot t_{D,D})\circ (\delta_{D}\ot D)\]
and
\[\varrho_{D}=(\mu_{D}\otimes D)\circ  (D\otimes t_{D,D})\circ (\delta_{D}\otimes \lambda_{D})\circ \delta_{D}.\]
Then
\[\omega_{D}^{a}=\varphi_{D}\circ (\eta_{D}\ot D):D\rightarrow D, \]
\[\omega_{D}^{c}=(\varepsilon_{D}\ot D)\circ \varrho_{D}:D\rightarrow D\]
are idempotent morphisms in ${\mathcal C}$  and
\[\omega_{D}^{a}=\mu_{D}\circ (D\ot (\lambda_{D}\circ \Pi_{D}^{L}))\circ \delta_{D},\]
\[\omega_{D}^{c}=\mu_{D}\circ (D\ot (\Pi_{D}^{L}\circ \lambda_{D}))\circ \delta_{D}.\]
\end{proposition}

\begin{proof} We prove the idempotent condition for $\omega_{D}^{a}$. The proof for $\omega_{D}^{c}$ is analogous and we leave the details to the reader.
\begin{itemize}
\item[ ]$\hspace{0.38cm} \omega_{D}^{a}\circ \omega_{D}^{a} $

\item[ ]$=\mu_{D}\circ (\mu_{D}\ot (\mu_{D}\circ (\lambda_{D}\otimes \lambda_{D})\circ t_{D,D}))\circ (\mu_{D}\ot t_{D,D}\ot D)\circ (D\ot t_{D,D}\ot t_{D,D})$
\item[]$\hspace{0.38cm}\circ (\delta_{D}\ot \delta_{D}\ot D)\circ (\eta_{D}\ot \eta_{D}\ot D) $

\item[ ]$=\mu_{D}\circ (\mu_{D}\ot \lambda_{D})\circ (D\ot t_{D,D})\circ (((\mu_{D}\ot \mu_{D})\circ (D\ot t_{D,D}\ot D)\circ (\delta_{D}\ot \delta_{D}))\ot D)  $
\item[]$\hspace{0.38cm}\circ(\eta_{D}\ot \eta_{D}\ot D)  $

\item[ ]$=\mu_{D}\circ (\mu_{D}\ot \lambda_{D})\circ (D\ot t_{D,D})\circ ((\delta_{D}\circ \mu_{D})\ot D)\circ (\eta_{D}\ot \eta_{D}\ot D)  $

\item[ ]$=\omega_{D}^{a}. $

\end{itemize}
The first equality follows by (b3-2)  and  associativity of
$\mu_{D}$. The second one is a consequence of [\cite{IND}, Proposition
2.20] and (b3-1). Finally, the third one follows
by (b4)  and the fourth one by the unit condition for $\mu_{D}$.

The equalities
\[\omega_{D}^{a}=\mu_{D}\circ (D\ot (\lambda_{D}\circ \Pi_{D}^{L}))\circ \delta_{D}\]
and
\[\omega_{D}^{c}=\mu_{D}\circ (D\ot (\Pi_{D}^{L}\circ \lambda_{D}))\circ \delta_{D}\]
follow from [\cite{IND}, Proposition 2.12].
\end{proof}

\begin{examples}
\begin{itemize}

\item[i)] Let $D=RG$ be the groupoid algebra considered in (1) of \ref{Yang-Baxter-examples}. Then the morphisms defined in the previous Proposition
are:
\[\omega_{RG}^{a}(\sigma)=\sigma \circ id_{t(\sigma)}= \left\{
\begin{array}{l} \sigma \; {\rm if}\; t(\sigma)=s(\sigma)\\ 0 \; {\rm if}\; t(\sigma)\neq s(\sigma)\end{array} \right.\]
\[\omega_{RG}^{c}(\sigma)=\sigma \circ id_{s(\sigma)}=\sigma.\]

In the particular case of the groupoid algebra on $n$-objects with
one invertible arrow between each ordered pair of objects, we obtain
that $RG$ is isomorphic to the $n\times n$ matrix $RG=M_{n}(R)$. The
weak Hopf algebra $H$ has the following structure. If $E_{ij}$
denote the $(i, j)$- matrix unit, $RG$ has counit given by
$\varepsilon_{RG}(E_{ij}) = 1$, comultiplication by
$\delta_{RG}(E_{ij}) = E_{ij}\otimes E_{ij}$ and antipode given by
$\lambda_{RG}(E_{ij}) = E_{ji}$ for each $i, j = 1, \cdots ,n$. In
this case, $\Pi_{RG}^{L}(E_{ij})=E_{ii}$,
$\Pi_{RG}^{R}(E_{ij})=E_{jj}$ and then $RG_{L}=RG_{R}$ is the
submodule of the diagonal matrices. Therefore, the image of
$\omega_{RG}^{a}$ is $RG_{L}$.

\item[ii)] In a general setting, if $D$ is a commutative ($\mu_{D}=\mu_{D}\circ t_{D,D}$) WBHA, then
$\Pi_{D}^{L}=\overline{\Pi}_{D}^{R},$ so by (\ref{Pielambda2}) and (\ref{Piconvolucion}), we have
\begin{itemize}
\item[ ]$\hspace{0.38cm} \omega_{D}^{a}$
\item[ ]$=\mu_{D}\circ (D\ot \Pi_{D}^{R})\co \delta_D$
\item[ ]$=id_D\wedge\Pi_{D}^{R}$
\item[ ]$=id_{D},$

\end{itemize}

and $\omega_{D}^{c}=id_D\wedge \Pi_{D}^{L}.$

In a similar way, if $D$ is a cocommutative ($\delta_{D}=
t_{D,D}\circ \delta_{D}$) WBHA, then
$\Pi_{D}^{L}=\overline{\Pi}_{D}^{L},$  and
$\omega_{D}^{a}=id_D\wedge \Pi_{D}^{L}$ and
$\omega_{D}^{c}=id_{D}$.

\item[iii)] In this example we assume that $\CC$ is braided with braiding $c$.
Let $A=(A,\eta_{A},\mu_{A},\varepsilon_{A},\delta_{A})$ be a separable Frobenius algebra in ${\mathcal C}$. Using the separability condition $\mu_{A}\circ \delta_{A}=id_{A}$, we get that $A\otimes A$ is a weak Hopf algebra in ${\mathcal C}$ (see \cite{PS}) where
    \[\eta_{A\ot A}=\eta_{A}\ot \eta_{A},\;\; \mu_{A\ot A}=((\mu_{A}\circ c_{A,A})\ot \mu_{A})\circ (A\ot c_{A,A}\ot A),\]
  \[\varepsilon_{A\ot A}=\varepsilon_{A}\circ \mu_{A},\;\; \delta_{A\ot A}=A\ot (\delta_{A}\circ \eta_{A})\ot A,\]
\[\lambda_{A\ot A}=(\varepsilon_{A\ot A}\ot A\ot A)\circ (A\ot c_{A,A}\ot A)\circ ((\delta_{A}\circ \eta_{A})\ot  c_{A,A}).\]
For $A\ot A$ we have
\[\Pi_{A\ot A}^{L}=((((\varepsilon_{A}\circ \mu_{A})\ot A)\circ (A\ot c_{A,A})\circ (\delta_{A}\ot A))\ot \eta_{A}),\]
and then
\[\omega_{A\ot A}^{a}=(A\ot (\mu_{A}\circ (\mu_{A}\ot A)\circ (A\ot c_{A,A})\circ ((\delta_{A}\circ \eta_{A})\ot A))),\]
\[\omega_{A\ot A}^{c}=  ((\varepsilon_{A}\circ \mu_{A})\ot (\mu_{A}\circ c_{A,A})\ot
A)\circ (A\ot c_{A,A}^{-1}\ot c_{A,A})\circ (c_{A,A}^{-1}\ot
c_{A,A}\ot A)\]
\[\circ(A\ot (\delta_{A}\circ \eta_{A})\ot (c_{A,A}^{-1}\circ \delta_{A})).\]

\item[iv)] Now we come back to section 3. Let  $D$ be a WBHA with invertible antipode and let $(B,f,g)$ be an object in ${\mathcal P}roj(D)$. Then the object
 $B_D$ defined in Proposition \ref{estruturasB_D} is a WBHA [\cite{Proj}, Theorem 3.4] with associated weak Yang-Baxter operator
\[t_{B_{D},B_{D}}=(\varphi_{B_{D}}\ot B_{D})\circ (D\ot r_{B_{D},B_{D}})\circ  (\rho_{B_D}\ot B_{D}): B_{D}\ot B_{D}\rightarrow B_{D}\ot B_{D},\]
where $r_{B_{D},B_{D}}=(p_D^B\ot p_D^B)\circ t_{B,B}\circ (i_D^B\ot
i_D^B)$, and antipode
\[\lambda_{B_{D}}=p_D^B\circ \mu_{B}\circ ((f\circ g)\ot \lambda_{B})\circ \delta_{B}\circ i_D^B.\]
Using the equalities proved in [\cite{Proj}, Lemma 2.16],
$\varepsilon_{B}\circ q_{D}^{B}=\varepsilon_{B}$ and
$\eta_{B}=q_{D}^{B}\circ \eta_{B}$ we obtain

\[\Pi_{B_{D}}^{L}=((\varepsilon_{B}\circ \mu_{B})\ot p_D^B)\circ (B\ot t_{B,B})\circ ((\delta_{B}\circ \eta_{B})\ot i_D^B=p_D^B\circ \Pi_{B}^{L} \circ i_D^B.\]

Then,
\[\omega_{B_{D}}^{a}=p_D^B\circ (q_{D}^{B}\wedge (i_D^B\circ \lambda_{B_{D}} \circ p_D^B\circ \Pi_{B}^{L}))\circ i_D^B.\]

Similarly,
\[\omega_{B_{D}}^{c}=p_D^B\circ (q_{D}^{B}\wedge (\Pi_{B}^{L}\circ i_D^B\circ \lambda_{B_{D}} \circ p_D^B))\circ i_D^B.\]

\end{itemize}
\end{examples}

\begin{proposition}
\label{pro-adjoint111} Let $D$ be a WBHA  in ${\mathcal C}$. Let
$\varphi_{D}:D\ot D\rightarrow D$ and  $\varrho_{D}:D\rightarrow
D\ot D$ be the morphisms defined in Proposition
\ref{adjoint-act-coact}. Then the following assertions hold:

\begin{equation}
\label{newphitesquerda}
t_{D,D}\co (\varphi_{D}\ot D)=(D\ot \varphi_{D})\co
(t_{D,D}\ot D)\co (D\ot t_{D,D}),
\end{equation}

\begin{equation}
\label{newphitdereita}
t_{D,D}\co (D\ot \varphi_{D})=(\varphi_{D}\ot D)\co (D\ot
t_{D,D})\co (t_{D,D}\ot D),
\end{equation}

\begin{equation}
\label{newrhotesquerda}
(\varrho_{D}\ot D)\co t_{D,D}=(D\ot t_{D,D})\co
(t_{D,D}\ot D)\co (D\ot \varrho_{D}),
\end{equation}

\begin{equation}
\label{newrhotdereita}
(D\ot \varrho_{D})\co t_{D,D}=(t_{D,D}\ot D)\co (D\ot
t_{D,D})\co (\varrho_{D}\ot D),
\end{equation}

\begin{equation}
\label{newphitprimaesquerda}
t^{\prime}_{D,D}\co (\varphi_{D}\ot D)=(D\ot \varphi_{D})\co
(t^{\prime}_{D,D}\ot D)\co (D\ot t^{\prime}_{D,D}),
\end{equation}

\begin{equation}
\label{newphitprimadereita}
t^{\prime}_{D,D}\co (D\ot \varphi_{D})=(\varphi_{D}\ot D)\co (D\ot
t^{\prime}_{D,D})\co (t^{\prime}_{D,D}\ot D),
\end{equation}

\begin{equation}
\label{newrhotprimaesquerda}
(\varrho_{D}\ot D)\co t^{\prime}_{D,D}=(D\ot t^{\prime}_{D,D})\co
(t^{\prime}_{D,D}\ot D)\co (D\ot \varrho_{D}),
\end{equation}

\begin{equation}
\label{newrhotprimadereita}
 (D\ot \varrho_{D})\co t^{\prime}_{D,D}=(t^{\prime}_{D,D}\ot D)\co (D\ot
t^{\prime}_{D,D})\co (\varrho_{D}\ot D).
\end{equation}

\end{proposition}

\begin{proof} We write by way of example the proof for (\ref{newphitprimaesquerda}); the others being analogous.
\begin{itemize}
\item[ ] $\hspace{0.38cm} t^{\prime}_{D,D}\co (\varphi_{D}\ot D)$
\item[ ] $=(D\ot \mu_{D})\circ (t^{\prime}_{D,D}\ot \mu_{D})\circ (D\ot t^{\prime}_{D,D}\ot \lambda_{D})\circ (D\ot D \ot t^{\prime}_{D,D})\circ (D\ot t_{D,D}\ot D)$
\item[]$\hspace{0.38cm}\circ (\delta_{D}\ot D\ot D)$
\item[ ] $=(D\ot \varphi_{D})\co
(t^{\prime}_{D,D}\ot D)\co (D\ot t^{\prime}_{D,D}).$
\end{itemize}
The first equality follows by (\ref{b1}) and by [\cite{IND}, Proposition 2.12]. The second one is a consequence of (\ref{tprim4}).
 \end{proof}

\begin{proposition}
\label{pro-adjoint} Let $D$ be a WBHA  in ${\mathcal C}$. Let
$\varphi_{D}:D\ot D\rightarrow D$ and  $\varrho_{D}:D\rightarrow
D\ot D$ be the morphisms defined in Proposition
\ref{adjoint-act-coact}. Then the following assertions hold:

\begin{equation}
\label{newphiphi}
\varphi_{D}\circ (D\ot \varphi_{D})=\varphi_{D}\circ (\mu_{D}\ot D),
\end{equation}

\begin{equation}
\label{newphidelta}
\delta_{D}\circ \varphi_{D}
=(\mu_D\ot D)\co(D\ot t_{D,D})\co(((\mu_D\ot\varphi_D)\co(D\ot t_{D,D}\ot
D)\co(\delta_D\ot\delta_{D}))\ot\lambda_D)\co(D\ot t_{D,D})\co(\delta_D\ot D),
\end{equation}

\begin{equation}
\label{newrhorho}
(D\ot \varrho_{D})\circ \varrho_{D}=(\delta_{D}\ot D)\circ \varrho_{D},
\end{equation}

\begin{equation}
\label{newrhomu}
\varrho_{D}\circ \mu_{D}
=(\mu_D\ot D)\co(D\ot t_{D,D})\co(((\mu_D\ot\mu_D)\co(D\ot t_{D,D}\ot
D)\co(\delta_D\ot\varrho_{D}))\ot\lambda_D)\co(D\ot t_{D,D})\co(\delta_D\ot D).
\end{equation}

\end{proposition}

\begin{proof}

The proof for (\ref{newphiphi}) is similar to the one used to prove
the idempotent character of $\omega_{D}^{a}$ removing in the
equalities the morphism $\eta_{D}\ot \eta_{D}\ot D$.

To see (\ref{newphidelta}), using that $D$ is a WBHA and (\ref{ant-delta}), we have
\begin{itemize}
\item[ ]$\hspace{0.38cm} (\mu_D\ot D)\co(D\ot t_{D,D})\co(((\mu_D\ot\varphi_D)\co(D\ot t_{D,D}\ot
D)\co(\delta_D\ot\delta_{D}))\ot \lambda_D)\co(D\ot t_{D,D})$
\item[]$\hspace{0.38cm}\co(\delta_D\ot D) $

\item[ ]$=(\mu_{D}\ot \mu_{D})\circ (D\ot t_{D,D}\ot D)\circ (((\mu_{D}\ot \mu_{D})\circ (D\ot t_{D,D}\ot D)\circ (\delta_D\ot \delta_D))$
\item[ ]$\hspace{0.38cm}\ot(t_{D,D}\circ (\lambda_D\ot \lambda_D)\circ \delta_D))\circ (D\ot t_{D,D})\circ (\delta_D\ot D)$

\item[ ]$=(\mu_{D}\ot \mu_{D})\circ (D\ot t_{D,D}\ot D)\circ ((\delta_{D}\circ \mu_{D})\ot (\delta_{D}\circ \lambda_{D}))\circ (D\ot t_{D,D})\circ (\delta_{D}\ot D) $

\item[ ]$= \delta_{D}\circ \varphi_{D}.$

\end{itemize}
The proofs for (\ref{newrhorho}) and (\ref{newrhomu})  are analogous and we leave the details to the reader.
\end{proof}

\begin{proposition}
\label{pro-adjoint2} Let $D$ be a WBHA  in ${\mathcal C}$. Let
$\omega_{D}^a$, $\omega_{D}^c$ be the idempotent morphisms defined
in Proposition \ref{adjoint-act-coact}. Then the following
assertions hold:

\begin{equation}
\label{newphiomega}
\varphi_{D}\circ (D\ot \omega_{D}^a)=\varphi_{D},
\end{equation}

\begin{equation}
\label{newdeltaomega}
(D\ot \omega_{D}^a)\circ \delta_{D}\circ \omega_{D}^a=\delta_{D}\circ \omega_{D}^a,
\end{equation}

\begin{equation}
\label{newrhoomega}
\varrho_{D}\circ (D\ot \omega_{D}^c)=\varrho_{D},
\end{equation}

\begin{equation}
\label{newmuomega}
\omega_{D}^c\circ \mu_{D}\circ (D\ot \omega_{D}^c)=\omega_{D}^c\circ \mu_{D}.
\end{equation}

\end{proposition}

\begin{proof}
As in the previous results we prove (\ref{newphiomega}) and (\ref{newdeltaomega}) leaving  the other equalities to the reader. The proof of (\ref{newphiomega}) is a direct consequence of (\ref{newphiphi}). To check (\ref{newdeltaomega}), first note that by (\ref{newphiomega}) the equality
\begin{equation}
(D\otimes \omega_{D}^a)\circ t_{D,D}\circ (\varphi_{D}\otimes
D)=t_{D,D}\circ (\varphi_{D}\otimes D)
\end{equation}
holds. Then, composing in (\ref{newphidelta}) with
$\eta_{D}\ot D$ and $D\otimes \omega_{D}^a$ we have
\begin{itemize}
\item[ ] $\hspace{0.38cm} (D\ot \omega_{D}^a)\circ \delta_{D}\circ \omega_{D}^a$
\item[ ] $=(D\ot \omega_{D}^a)\circ (\mu_D\ot D)\co(D\ot t_{D,D})\co(((\mu_D\ot\varphi_D)\co(D\ot t_{D,D}\ot
D)\co(\delta_D\ot\delta_{D}))\ot\lambda_D)$
\item[]$\hspace{0.38cm} \co(D\ot t_{D,D})\co((\delta_D\circ \eta_{D})\ot D)$
\item[]$=(\mu_D\ot D)\co(D\ot t_{D,D})\co(((\mu_D\ot\varphi_D)\co(D\ot t_{D,D}\ot
D)\co(\delta_D\ot\delta_{D}))\ot\lambda_D)\co(D\ot t_{D,D})$
\item[]$\hspace{0.38cm} \co((\delta_D\circ \eta_{D})\ot D)$
\item[]$=\delta_{D}\circ \omega_{D}^a.$
\end{itemize}
\end{proof}

\begin{notation}\label{notac}
Let $D$ be a WBHA  in ${\mathcal C}$. Let $\omega_{D}^a$,
$\omega_{D}^c$ be the idempotent morphisms defined in Proposition
\ref{adjoint-act-coact}. For $x\in\{a,c\}$, with $\Omega^{x}(D)$,
$p_{D}^{x}:D\rightarrow \Omega^{x}(D)$,
$i_{D}^x:\Omega^{x}(D)\rightarrow D$ we denote the object and the
morphisms such that $\omega_{D}^{x}=i_{D}^x\circ p_{D}^{x}$ and
$id_{\Omega^{x}(D)}=p_{D}^{x}\circ i_{D}{}^{x}.$
\end{notation}

\begin{proposition}
\label{mod-comod} Let $D$ be a WBHA  in ${\mathcal C}$. The
following assertions hold:
\begin{itemize}
\item[(i)] The object $\Omega^{a}(D)$ is a left $D$-module with action
\[\varphi_{\Omega^{a}(D)}=p_{D}^{a}\circ \varphi_{D}\circ (D\ot i_{D}^a):D\ot \Omega^{a}(D)\rightarrow \Omega^{a}(D)\]
and a left $D$-comodule with coaction
\[\rho_{\Omega^{a}(D)}=(D\ot p_{D}^{a})\circ \delta_{D}\circ i_{D}^a:\Omega^{a}(D)\rightarrow D\ot \Omega^{a}(D).\]
\item[(ii)] The object $\Omega^{c}(D)$ is a left $D$-module with action
\[\psi_{\Omega^{c}(D)}=p_{D}{}^{c}\circ \mu_{D}\circ (D\ot i_{D}^c):D\ot \Omega^{c}(D)\rightarrow \Omega^{c}(D)\]
and a left $D$-comodule with coaction
\[\varrho_{\Omega^{c}(D)}=(D\ot p_{D}{}^{c})\circ \varrho_{D}\circ i_{D}^c:\Omega^{c}(D)\rightarrow D\ot \Omega^{c}(D).\]
\end{itemize}

\end{proposition}

\begin{proof} We shall prove (i). The proof for the second assertion is analogous.

Firstly note that
\[\varphi_{\Omega^{a}(D)}\circ (\eta_{D}\ot \Omega^{a}(D))=p_{D}^{a}\circ \omega_{D}^{a}\circ i_{D}^a=id_{\Omega^{a}(D)}.\]

Secondly, by (\ref{newphiphi})  and (\ref{newphiomega}) , we have
\begin{itemize}
\item[ ] $\hspace{0.38cm}\varphi_{\Omega^{a}(D)}\circ(D\ot \varphi_{\Omega^{a}(D)}) $
\item[ ] $=p_{D}^{a}\circ\varphi_{D}\circ (D\ot \omega_{D}^{a})\circ (D\otimes \varphi_{D})\circ (D\ot D\ot i_{D}^a) $
\item[]$= p_{D}^{a}\circ\varphi_{D}\circ  (D\otimes \varphi_{D})\circ (D\ot D\ot i_{D}^a)  $
\item[]$=p_{D}^{a}\circ\varphi_{D}\circ  (\mu_{D}\otimes D)\circ (D\ot D\ot i_{D}^a)  $
\item[]$= \varphi_{\Omega^{a}(D)}\circ(\mu_{D}\ot \Omega^{a}(D)).$
\end{itemize}

On the other hand, trivially $(\varepsilon_{D}\ot
\Omega^{a}(D))\circ \rho_{\Omega^{a}(D)}=id_{\Omega^{a}(D)}.$
Finally, by (\ref{newdeltaomega}) we have
\begin{itemize}
\item[ ] $\hspace{0.38cm}(D\ot \rho_{\Omega^{a}(D)})\circ \rho_{\Omega^{a}(D)} $
\item[ ] $=(D\ot ((D\ot p_{D}^{a})\circ \delta_{D}))\circ (D\ot  \omega_{D}^{a})\circ \delta_{D}\circ \omega_{D}^{a}\circ  i_{D}^a$
\item[]$= (D\ot ((D\ot p_{D}^{a})\circ \delta_{D}))\circ \delta_{D}\circ  i_{D}^a  $
\item[]$=(\delta_{D}\ot \Omega^{a}(D))\circ \rho_{\Omega^{a}(D)} .$
\end{itemize}
\end{proof}

\begin{proposition}
\label{mod-comod} Let $D$ be a WBHA  in ${\mathcal C}$. The
following assertions hold:
\begin{itemize}
\item[(i)] The object $(\Omega^{a}(D),\varphi_{\Omega^{a}(D)}, \rho_{\Omega^{a}(D)})$ is in $_D^D\mathcal{YD}$.
\item[(ii)] The object $(\Omega^{c}(D),\psi_{\Omega^{c}(D)}, \varrho_{\Omega^{c}(D)})$ is in $_D^D\mathcal{YD}$.
\end{itemize}

\end{proposition}

\begin{proof} First note that by (b3) of Definition \ref{WBH}, the properties of the weak operator $t_{D,D}$
 and [\cite{IND}, Propositions 2.11, 2.12], the following identities hold

\begin{equation}
\label{omega-t} t_{D,D}\circ (\omega_{D}^{x}\ot D)=(D\ot
\omega_{D}^{x})\circ t_{D,D},\;\;\;\;t_{D,D}\circ (D\ot
\omega_{D}^{x})=( \omega_{D}^{x}\ot D)\circ t_{D,D},
\end{equation}
\begin{equation}
\label{omega-tp} t^{\prime}_{D,D}\circ (\omega_{D}^{x}\ot D)=(D\ot
\omega_{D}^{x})\circ t^{\prime}_{D,D},\;\;\;\;t^{\prime}_{D,D}\circ
(D\ot \omega_{D}{}^{x})=( \omega_{D}^{x}\ot D)\circ
t^{\prime}_{D,D},
\end{equation}
\begin{equation}
\label{omega-nabla} \nabla_{D,D}\circ (\omega_{D}^{x}\ot
D)=(\omega_{D}^{x}\ot D)\circ \nabla_{D,D},\;\;\;\;\nabla_{D,D}\circ
(D\ot \omega_{D}^{x})=( D\ot \omega_{D}^{x})\circ \nabla_{D,D},
\end{equation}
for $x\in\{a,c\}$.

Following the notation introduced in \ref{notac} we get that the
quadruple
\[(r_{\Omega^{x}(D)},r_{\Omega^{x}(D)}^\prime,s_{\Omega^{x}(D)},s_{\Omega^{x}(D)}^\prime),\]
with
\begin{itemize}
\item[]$r_{\Omega^{x}(D)}:=(D\ot p_D^x)\co t_{D,D}\co(i_D^x\ot D),\qquad r_{\Omega^{x}(D)}^\prime:=(p_D^x\ot D)\co t^{\prime}_{D,D}\co(D\ot i_D^x),$

\item[ ]$s_{\Omega^{x}(D)}:= (p_D^x\ot D)\co t_{D,D}\co(D\ot i_D^x) ,\qquad s_{\Omega^{x}(D)}^\prime:=(D\ot p_D^x)\co t^{\prime}_{D,D}\co(i_D^x\ot D),$

\end{itemize}
is a $(\Omega^{x}(D),D)$-WO. Indeed, the equalities contained in
(c1) and (c2) of Definition \ref{WO} are a consequence of
(\ref{omega-t}) and (\ref{omega-tp}) as well as the properties of
$t_{D,D}$ and $t^{\prime}_{D,D}$. The proof for the identities of
(c3) follows by (\ref{b3-1})-(\ref{b3-5}) and (\ref{omega-t}),
(\ref{omega-tp}). The eight equalities of (c4) follow from
(\ref{omega-t}), (\ref{omega-tp}) and (b3) of Definition \ref{WBH}.
Finally, (c5) is a consequence of [\cite{IND}, Proposition 2.12].

For $x=a$ we have that the quadruple
$(r_{\Omega^{a}(D)},r_{\Omega^{a}(D)}^\prime,s_{\Omega^{a}(D)},s_{\Omega^{a}(D)}^\prime)$
is compatible with the module-comodule structure induced by the
action $\varphi_{\Omega^{a}(D)}$ and the coaction
$\rho_{\Omega^{a}(D)}$. To prove this assertion, by Definition
\ref{WOmodulecompatible}, for the action $\varphi_{\Omega^{a}(D)}$
we must show the equalities
 \begin{equation}
\label{fin1} r_{\Omega^{a}(D)}\co(\varphi_{\Omega^{a}(D)}\ot
D)=(D\ot\varphi_{\Omega^{a}(D)})\co(t_{D,D}\ot \Omega^{a}(D))\co(
D\ot r_{\Omega^{a}(D)}),
\end{equation}
\begin{equation}
\label{fin2}
r_{\Omega^{a}(D)}^\prime\co(D\ot\varphi_{\Omega^{a}(D)})=(\varphi_{\Omega^{a}(D)}\ot
D)\co(D\ot r_{\Omega^{a}(D)}^\prime)\co(t^{\prime}_{D,D}\ot
\Omega^{a}(D))
\end{equation}
and the analogous equalities taking $t_{D,D}$, $t^{\prime}_{D,D}$,
$s_{\Omega^{a}(D)}$ and $s_{\Omega^{a}(D)}^\prime$ instead of
$t^{\prime}_{D,D}$, $t_{D,D},$ $r_{\Omega^{a}(D)}^\prime$ and
$r_{\Omega^{a}(D)}$ respectively. The proofs for the four equalities
are similar and then we only write one of them, for example
(\ref{fin2}):

\begin{itemize}
\item[ ] $\hspace{0.38cm}r_{\Omega^{a}(D)}^\prime\co(D\ot\varphi_{\Omega^{a}(D)}) $
\item[ ] $=(p_{D}{}^a\ot D)\circ t^{\prime}_{D,D} \circ (D\ot \varphi_D)\circ (D\ot D\ot i_{D}{}^a)$
\item[]$=((p_{D}{}^a\circ \varphi_D)\ot D)\circ (D\ot t^{\prime}_{D,D})\circ (t^{\prime}_{D,D}\ot i_{D}{}^a)$
\item[]$=(\varphi_{\Omega^{a}(D)}\ot D)\co(D\ot r_{\Omega^{a}(D)}^\prime)\co(t^{\prime}_{D,D}\ot \Omega^{a}(D)).$
\end{itemize}
In the last computations, the first and the third equalities follow by
(\ref{omega-tp}) and the idempotent character of $\omega_{D}^{a}$.
The second one is a consequence of (\ref{newphitprimadereita}).

For the coaction $\rho_{\Omega^{a}(D)}=(D\ot p_{D}^{a})\circ
\delta_{D}\circ i_{D}^a$ the proofs for
\begin{equation}
\label{fin3} (D\ot \rho_{\Omega^{a}(D)})\co
r_{\Omega^{a}(D)}=(t_{D,D}\ot \Omega^{a}(D))\co(D\ot
r_{\Omega^{a}(D)})\co(\rho_{\Omega^{a}(D)}\ot D),
\end{equation}
\begin{equation}
\label{fin4} (\rho_{\Omega^{a}(D)}\ot D)\co
r_{\Omega^{a}(D)}^\prime=(D\ot
r_{\Omega^{a}(D)}^\prime)\co(t^{\prime}_{D,D}\ot
\Omega^{a}(D))\co(D\ot\rho_{\Omega^{a}(D)}),
\end{equation}
and the analogous equalities taking $t_{D,D}$, $t^{\prime}_{D,D}$,
$s_{\Omega^{a}(D)}$ and $s_{\Omega^{a}(D)}^\prime$ instead of
$t^{\prime}_{D,D}$, $t_{D,D},$ $r_{\Omega^{a}(D)}^\prime$ and
$r_{\Omega^{a}(D)}$ respectively, are a direct consequence of
(\ref{omega-t}), (\ref{omega-tp}), (b3-3), (b3-4), (\ref{b3}) and
(\ref{b4}).

Finally, the triple $(\Omega^{a}(D),\varphi_{\Omega^{a}(D)},
\rho_{\Omega^{a}(D)})$ is a left-left Yetter-Drinfeld module over
$D$ because it satisfies (yd3). Indeed:
\begin{itemize}
\item[ ] $\hspace{0.38cm}(\mu_D\ot \Omega^{a}(D))\co(D\ot r_{\Omega^{a}(D)})\co(((\mu_D\ot\varphi_{\Omega^{a}(D)})\co(D\ot t_{D,D}\ot
D)\co(\delta_D\ot\rho_{\Omega^{a}(D)}))\ot\lambda_D)$
\item[ ] $\hspace{0.38cm}\co(D\ot s_{\Omega^{a}(D)})\co(\delta_D\ot \Omega^{a}(D)) $
\item[ ] $=(\mu_D\ot p_{D}^{a})\co(D\ot t_{D,D})\co(((\mu_D\ot\varphi_D)\co(D\ot t_{D,D}\ot
D)\co(\delta_D\ot\delta_{D}))\ot\lambda_D)\co(D\ot t_{D,D})$
\item[ ] $\hspace{0.38cm}\co(\delta_D\ot i_{D}^a)$
\item[]$= (D\ot p_{D}^{a})\co\delta_{D}\circ \varphi_{D}\co(D\ot i_{D}^a)  $
\item[]$=\rho_{\Omega^{a}(D)}\circ \varphi_{\Omega^{a}(D)}, $
\end{itemize}
where the first equality follows from (\ref{newphiomega}) and (\ref{omega-t}), the second one by (\ref{newphidelta}) of
and the last one by (\ref{newdeltaomega}).

The proof for the second assertion is similar and we leave the
details to the reader. \end{proof}

\section*{Acknowledgements}
The authors were supported by Ministerio de Educaci\'on (Project:
MTM2010-15634) and  FEDER.

\end{document}